\UseRawInputEncoding
\documentclass[11pt,a4paper]{article}
\usepackage{graphicx}
\usepackage{indentfirst}
\usepackage{float}
\usepackage{mathrsfs}
\usepackage{fancyhdr}
\usepackage{amsfonts,amssymb,amsmath,amsthm}
\usepackage{hyperref}
\newtheorem{thm}{Theorem}[section]
\newtheorem{lemma}[thm]{Lemma}
\newtheorem{cor}[thm]{Corollary}
\newtheorem{prop}[thm]{Proposition}

\newtheoremstyle{rem}{10pt}{10pt}{\rmfamily}{}{\bfseries}{.}{.5em}{}
\theoremstyle{rem}
\newtheorem{rem}[thm]{Remark}

\title{\bf Global Cauchy problem for the NLKG  in super-critical  spaces}

\author{ Baoxiang Wang}
\date{}

\begin{document}

\maketitle
\begin{abstract}
By introducing a class of new function spaces $B^{\sigma,s}_{p,q}$ as the resolution spaces,  we study the Cauchy problem for the nonlinear Klein-Gordon equation (NLKG) in all spatial dimensions $d \geqslant 1$,
$$
\partial^2_t u  + u-  \Delta u  + u^{1+\alpha}  =0, \ \ (u, \partial_t u)|_{t=0} = (u_0,u_1).
$$
We consider the initial data $(u_0,u_1)$ in super-critical function spaces $E^{\sigma,s} \times E^{\sigma-1,s}$ for which their norms are defined by
$$
\|f\|_{E^{\sigma,s}} = \|\langle\xi\rangle^\sigma 2^{s|\xi|}\widehat{f}(\xi)\|_{L^2}, \ s<0, \ \sigma \in \mathbb{R}.
$$
Any Sobolev space  $H^{\kappa}$ can be embedded into $ E^{\sigma,s}$,  i.e., $H^\kappa \subset E^{\sigma,s}$ for any $ \kappa,\sigma \in \mathbb{R}$ and $s<0$.  We show the global existence and uniqueness of the solutions of NLKG if the initial data belong  to some $E^{\sigma,s} \times E^{ \sigma-1,s}$ ($s<0, \ \sigma \geqslant \max (d/2-2/\alpha, \, 1/2), \ \alpha\in \mathbb{N}, \ \alpha \geqslant 4/d$) and their Fourier transforms are supported in the first octant,  the smallness conditions on the initial data in $E^{\sigma,s} \times E^{\sigma-1,s}$ are not required for the global solutions. Similar results hold for the sinh-Gordon equation $\partial^2_t u   -  \Delta u + \sinh u=0$ if the spatial dimensions $d \geqslant 2$. \\

		
	\end{abstract}

\section{Introduction}
We  consider the Cauchy problem for the nonlinear Klein-Gordon (NLKG) and sinh-Gordon equations
\begin{align}\label{NLKG}
\left\{
\begin{array}{l}
\partial^2_t u  + u-  \Delta u  + f(u)  =0,  \\
  u(x,0) = u_0(x), \ u_t (x,0) = u_1(x),
\end{array}
\right.
\end{align}
where $u(x,t)$ is a complex valued function of $(x,t)\in \mathbb{R}^{d+1} $ and
$ f(u) = \pm u^{1+\alpha}, \, \alpha \in \mathbb{N}$ or  $f(u)= \sinh u -u, \, \sinh u =(e^u - e^{-u})/2 $.

NLKG is a basic model in nonlinear relativistic equation of mathematical physics \cite{Sc1951,Wh1974}. The (complex) sinh-Gordon equation  comes from geometric analysis \cite{Ch1981},  integrable quantum field theory, kink
dynamics, fluid dynamics, and in many other scientific applications (see \cite{AbKaNeSe1973,KuSeCh2018,Wa2012} and references therein).  A large amount of work has been devoted to the study the global well-posedness and scattering for the defocusing NLKG in the energy subcritical and critical case $f(u)= |u|^\alpha u$ with $\alpha \leqslant 4/(d-2)$ for $d\geqslant 3$, $ \alpha <\infty$ for $d=1,2$. Ginibre and Velo \cite{GiVe1985,GiVe1989} showed the global well-posedness of finite energy solutions in the energy subcritical case. The existence of the scattering operators in the energy subcritical cases were studied by Strauss \cite{St1981}, Tsutsumi \cite{Ts1983} and Pecher \cite{Pe1984,Pe1985} for small data; by Brenner \cite{Br1984,Br1985} in dimensions  $d \geqslant 3$ and by Nakanishi \cite{Na1999a} in dimensions 1 and 2 for large data, respectively.  In 3D energy critical case $\alpha=4$, the classical global solutions were obtained by Struwe for the radial data \cite{Str1989} and by Grillakis \cite{Gr1990} for the general data. In the energy critical case $\alpha=4/(d-2)$ ($d \geqslant 3$), Shatah and Struwe \cite{ShStr1994} obtained the global
well-posedness  in the energy space. In \cite{Na1999}, Nakanishi showed the existence of the scattering operators in $H^1\times L^2$ for the energy critical case. In the focusing case $f(u)=-|u|^\alpha u$, the scattering threshold was studied in \cite{IbMaNa2011,IbMaNa2014} by Ibrahim,  Masmoudi  and Nakanishi for the energy  subcritical and critical nonlinearity. In the $L^2$ critical case $\alpha=2$ in 2D, the scattering threshold was studied in Killip,  Stovall and Visan \cite{KiStVi2012}.  Hayashi and Naumkin \cite{HaNa2016} considered the existence of the wave operator with small data for $\alpha=1$ and $d=2$.  However, for the energy supercritical case $\alpha >4/(d-2)$, the global well posedness and scattering for NLKG are only known for small Cauchy data in critical and subcritical Sobolev spaces $H^\sigma \times H^{\sigma -1}$ (cf. \cite{Wa1998,Wa1999,WaHud07}). If the nonlinearity has exponential growth,  Nakamura and Ozawa \cite{NaOz2001} obtained the global well posedness and scattering in $H^{\sigma} \times H^{\sigma-1}$ ($\sigma \geqslant d/2$) for small data, where their results contain sinh-Gordon equation  as special case. Struwe \cite{Str2013} obtained the global existence for the smooth solutions in 2D with the exponential nonlinearity $f(u) = e^{u^2} u-u$ for any smooth initial data.

The main purpose of this paper is to study NLKG \eqref{NLKG} in certain supercritical function spaces $E^{\sigma,s} \times E^{\sigma-1,s}$ ($s<0$), which include any Sobolev spaces $H^\kappa \times H^{\kappa -1}$ $(\kappa \in \mathbb{R})$ as subspaces and we will obtain its global existence and uniqueness of the solutions when the power of the nonlinearity $\alpha \geqslant 4/d$ and the Fourier transforms of the initial data are supported in the first octant. In particular, the energy supercritical case $\alpha > 4/(d-2)$ is contained in our results and the smallness conditions for the initial data in $E^{\sigma,s} \times E^{\sigma-1,s}$ are not required. Similar results hold for sinh-Gordon equation and $f(u) = e^{u^2} u-u$ for $d\geq 2$.  In order to get these results, we will apply a class of new function spaces $B^{\sigma,s}_{p,q}$ and $F^{\sigma,s}_{p,q}$ ($\sigma, \, s\in \mathbb{R}, \, 1\leqslant p,q \leqslant \infty$)  to handle the nonlinearity (see Section \ref{Functspaces}).

\subsection{Function spaces $E^{\sigma,s}$}

Let  $ \mathscr{S}$ be the Schwartz space  and $ \mathscr{S}'$ be its dual space. We write $|x|=|x_1|+...+|x_d|$ for $x=(x_1,...,x_d)\in \mathbb{R}^d$ and
$$
p_\lambda(f) = \sup_{x\in \mathbb{R}^d}e^{\lambda |x|}|f(x)|, \quad q_{\lambda}(f)= \sup_{\xi\in \mathbb{R}^d} e^{\lambda|\xi|}|\widehat{f}(\xi)|,
$$
\begin{equation*}
		\mathscr{S}_1:=\{f\in \mathscr{S}: p_\lambda(f)+q_\lambda(f)<\infty, ~\forall~\lambda>0 \}.
\end{equation*}	
	$ \mathscr{S}_1$ equipped with the system of semi-norms $\{p_\lambda+q_\lambda\}_{\lambda>0}$ is a complete locally convex linear topological space, which is said to be the Gelfand-Shilov space, cf. \cite{GeSh1968}. We denote by $ \mathscr{S}_1'$ the dual space of $ \mathscr{S}_1$.  One easily sees that
$$ \mathscr{S}_1 \subset  \mathscr{S}, \ \  \mathscr{S}' \subset  \mathscr{S}'_1.$$
$\mathscr{S}_1$ contains the translations and modulations of Gaussian $e^{-\mathrm{i}m x} e^{- |x-n|_2^2/2}$ and their linear combinations\footnote{We denote $|x|_2 = (x^2_1+...+x^2_d)^{1/2}$ for $x=(x_1,...,x_d) \in \mathbb{R}^d$}, which is dense in all Sobolev spaces $H^\sigma$ (cf. \cite{GeSh1968}).
The Fourier transforms on $ \mathscr{S}_1'$ can be defined by duality (cf. \cite{FeGrLiWa2021}).  Namely, for any $f\in  \mathscr{S}'_1$, its Fourier transform $\mathscr{F} f = \widehat{f}$ satisfies
$$
\langle \mathscr{F} f, \,  \varphi \rangle = \langle f, \,  \mathscr{F} \varphi \rangle, \ \ \forall \ \varphi  \in \mathscr{S}_1.
$$
Let us recall the function spaces $E^{\sigma,s}$ introduced in \cite{ChWaWa2022}. For $s, \sigma \in \mathbb{R}$, we denote
		\begin{equation*}
			E^{\sigma,s} :=\{f\in \mathscr{S}_1': 	\langle\xi \rangle^{\sigma}  2^{s|\xi|}\widehat{f}(\xi)\in L^2(\mathbb{R}^d)\}
		\end{equation*}
		 for which the norm is defined by
$\|f\|_{E^{\sigma,s}} = 	\|\langle\xi \rangle^{\sigma} 2^{s|\xi|}\widehat{f}(\xi)\|_{L^2}.$
$E^{\sigma,s}$ is a generalization of Sobolev spaces $H^\sigma:=(I-\Delta)^{-\sigma/2} L^2(\mathbb{R}^d)$ for which the norm is defined by
\begin{align} \label{spaceHs}
  \|f\|_{H^\sigma} := \|\langle \xi\rangle^\sigma \widehat{f} \|_{L^2}.
\end{align}
One easily sees that $H^\sigma = E^{\sigma,0}$. In the case $s>0$, $E^s:= E^{0,s}$ was introduced by Bj\"orck \cite{Bj1966} and it is a smooth function space. The spaces like $E^{\sigma,s}$ ($s>0$) have important application in studying the Gevrey regularity for the solutions of evolution equations (see e.g. Foias and Temam \cite{FoTe1989},  Bourdaud, Reissig and Sickel \cite{BoReSi2003}). However, in the case $s<0$,  $E^{\sigma,s}$ is much rougher than any $H^\kappa$  and we have the following inclusion (cf. \cite{ChWaWa2022}),
\begin{align} \label{embedding}
  H^{\kappa} \subset E^{\sigma,s}, \ \ \ \forall \ \kappa,\, \sigma \in \mathbb{R}, \, s<0.
\end{align}
By \eqref{embedding}, we see that $\cup_{\kappa \in \mathbb{R}} H^\kappa$ is a subset of $E^{\sigma,s}$ if $s<0$.

\subsection{Critical and super-critical spaces}

First, let us observe that if $u$ is a solution of NLKG \eqref{NLKG} with $f(u)= \pm u^{\alpha +1}$,
then $u_\lambda (x,t): = \lambda^{2/\alpha} u(\lambda x, \lambda t) $  solves
\begin{align}\label{NLKG2}
\left\{
\begin{array}{l}
\partial^2_t v  + \lambda^2 v-  \Delta v  \pm  v^{\alpha +1}  =0,  \\
  v(x,0) = u_{0,\lambda} (x) , \ v_t (x,0) = u_{1,\lambda} (x) ,
\end{array}
\right.
\end{align}
where $u_{0,\lambda} (x) = \lambda^{2/\alpha} u_0(\lambda x), \  u_{1,\lambda} (x) = \lambda^{2/\alpha +1} u_1(\lambda x).$ From the scaling argument, we see that $(u_\lambda, \partial_t u_\lambda)|_{t=0} $ is   invariant in homogeneous Sobolev spaces $\dot H^{\sigma_c} \times  \dot H^{\sigma_c-1}$ $(\dot H^\sigma := (-\Delta)^{-\sigma/2} L^2)$, where
\begin{align}
\sigma_c : = d/2-2/\alpha.  \label{criticalindex}
\end{align}
For any $\lambda >0$, $ H^{\sigma_c} \times  H^{\sigma_c-1}$ is said to be a {\it critical Sobolev space} for NLKG \eqref{NLKG2}.      $H^\sigma \times H^{\sigma-1}$ with $\sigma>\sigma_c$ ($\sigma<\sigma_c$) is said to be the {\it subcritical (supercritical) Sobolev space}. Any Banach function space $X$ satisfying $H^\sigma \times H^{\sigma-1} \subset X$ for some $\sigma<\sigma_c$   is said to be a {\it supercritical space} for NLKG.
The inclusion \eqref{embedding} implies that any supercritical Sobolev space $H^\kappa \times H^{\kappa-1}$ ($\kappa<\sigma_c$) is a subspace of $E^{ {\sigma},s} \times E^{ {\sigma}-1,s}$ if $s<0, \,  {\sigma} \in \mathbb{R}$. So, $E^{ {\sigma},s} \times E^{ {\sigma}-1,s}$ with $s<0, \, {\sigma} \in \mathbb{R}$ is a supercritical space.

\subsection{Besov and Triebel spaces} \label{BesovTriebel}

Let $\psi $ be a smooth radial bump function satisfying $\psi(\xi) =1$ for $|\xi| \leqslant 1$ and ${\rm supp} \, \psi = \{\xi: |\xi| \leqslant 5/4\}$.  Denote $\varphi(\xi) = \psi(\xi) -\psi (2 \xi )$ and $\varphi_j (\xi) = \varphi (2^{-j} \xi)$. One sees that
$$
\psi(\xi) +\sum_{j\geqslant 1}  \varphi_j (\xi) =1, \ \ \forall \   \xi  \in \mathbb{R}^d
$$
and
$$
 \sum_{j\in \mathbb{Z}}  \varphi_j (\xi) =1, \ \ \forall \   \xi  \in \mathbb{R}^d \setminus \{0\}.
$$
We write
$$
\triangle_0 = \mathscr{F}^{-1} \psi  \mathscr{F}, \ \  \ \triangle_j = \mathscr{F}^{-1} \varphi_j   \mathscr{F}, \ \ j\in \mathbb{N}
$$
and
$$
 \dot{\triangle}_j = \mathscr{F}^{-1} \varphi_j   \mathscr{F}, \ \ j\in \mathbb{Z}.
$$
Notice that $\dot{\triangle}_j =  {\triangle}_j$ for $j\geqslant 1$.  $ {\triangle}_j $ with $  j\in \mathbb{N} $   ($ \dot{\triangle}_j $ with $ j\in \mathbb{Z}$) are said to be the (homogeneous) dyadic decomposition operators. Let $\sigma \in \mathbb{R}, \, 1\leqslant p,q\leqslant \infty$. The norms for a tempered distribution $f\in \mathscr{S}'$ on Besov and Triebel-Lizorkin spaces can be defined in the following way (with a modification for $q=\infty$):
\begin{align} \label{Besovspace}
& \|f\|_{ B^{\sigma}_{p,q}  } =  \left(\sum_{j\in \mathbb{Z}_+} 2^{\sigma jq}   \| \triangle_j u  \|^q_{ L^p (\mathbb{R}^d)} \right)^{1/q},\\
& \|f\|_{ F^{\sigma}_{p,q}  } =  \left\|\Big(\sum_{j\in \mathbb{Z}_+} 2^{\sigma jq}    | \triangle_j u |^q \Big)^{1/q}\right\|_{ L^p (\mathbb{R}^d)}.  \label{Triebelspace}
\end{align}

\subsection{Main results}

For our purpose, we introduce the following resolution spaces (see Section \ref{Functspaces} for details)
\begin{align} \label{resolutionspace}
& \|u\|_{\widetilde{L} ^\gamma (\mathbb{R}; B^{\sigma,s}_{p,q} ) } =  \left(\sum^\infty_{j=0} 2^{\sigma jq}   \|2^{s|\nabla|}\triangle_j u  \|^q_{L^\gamma_t L^p_x (\mathbb{R} \times\mathbb{R}^d)} \right)^{1/q}.
\end{align}
Such a kind of spaces without the exponential derivatives $2^{s|\nabla|}$, i.e., $\widetilde{L} ^\gamma (\mathbb{R}; B^{\sigma,0}_{p,q} )$  were first introduced by Chemin-Lerner \cite{ChLe95}. Denote $ \|f\|_{H^{\sigma, s}_p} = \|\mathscr{F}^{-1} \langle \xi\rangle^\sigma 2^{s|\xi|} \mathscr{F} f\|_{L^p} $ and
\begin{align*}
\mathbb{R}^d_I & =\{\xi \in \mathbb{R}^d : \ \xi_j \geqslant 0, \ j=1,...,d \}.
\end{align*}
We have the following results.
	\begin{thm}\label{mainresult}
	Let $d\geqslant 1$, $f(u)=u^{1+\alpha}$, $ 4/d \leqslant \alpha<\infty$, $\alpha \in \mathbb{N}$,  $s < 0$ and
$\sigma\geqslant \max(\sigma_c,\, 1/2)$.
Suppose that $(u_0 , u_1) \in E^{\sigma,s} \times E^{\sigma-1,s}$ with
$ \mathrm{supp}~\widehat{u}_0, \, \mathrm{supp}~\widehat{u}_1 \subset \mathbb{R}^d_I\setminus \{0\}$ for $d \leqslant 6$, and $ \mathrm{supp}~\widehat{u}_0, \, \mathrm{supp}~\widehat{u}_1 \subset \mathbb{R}^d_I$ for $d \geqslant 7$.
Then there exists $s_0 \leqslant s$ such that NLKG \eqref{NLKG} has a unique mild solution $ u \in C(\mathbb{R}; E^{\sigma, s_0} ) \cap X^{\sigma,s_0}_p$, where the norm on $X^{\sigma,s}_p$ is defined by
\begin{align} \label{1workspaceX}
\|u\|_{X^{\sigma,s}_p} =
\left\{
\begin{array}{ll}
 \|u\|_{\widetilde{L}^{\infty} (\mathbb{R}, B^{\sigma,s}_{2, 2} )}
  +   \|u\|_{\widetilde{L}^{p} (\mathbb{R}, B^{\sigma-1/2,s}_{p, 2} )},  & \sigma >1/2; \\
 \|u\|_{{L}^{\infty} (\mathbb{R}, E^{1/2,s} )}
  +   \|u\|_{ {L}^{p} (\mathbb{R}, H^{0,s}_{p} )},  & \sigma =1/2,
 \end{array}
\right.
\end{align}	
where
$$
p=
\left\{
\begin{array}{ll}
2(d+1)/(d-1), & \mbox{for} \ \alpha \geqslant 4/(d-1); \\
2(d+1+\theta)/(d-1+\theta) , &  \mbox{for} \ \alpha =  4/(d-1+\theta), \, \theta \in (0,1].
\end{array}
\right.
$$
\end{thm}

	\begin{thm}\label{mainresult2}
	Let $d\geqslant 2$, $f(u)=\sinh u -u$, $s < 0,   \sigma \geqslant d/2 $, $(u_0 , u_1) \in E^{\sigma,s} \times E^{\sigma-1,s}$ with $\mathrm{supp}~\widehat{u}_0, \ \mathrm{supp}~\widehat{u}_1 \subset  \mathbb{R}^d_I \setminus \{0\} $ for $d\leqslant 4$, and $\mathrm{supp}~\widehat{u}_0, \ \mathrm{supp}~\widehat{u}_1 \subset  \mathbb{R}^d_I $ for $d\geqslant 5$. Then there exists $s_0 \leqslant s$ such that NLKG \eqref{NLKG} has a unique mild solution $u \in C(\mathbb{R}; E^{\sigma, s_0} ) \cap Y^{\sigma,s_0}_r$, where the norm on $Y^{\sigma,s}_r$ is defined by
\begin{align}
\|u\|_{Y^{\sigma,s}_r} =
 \|u\|_{\widetilde{L}^{\infty} (\mathbb{R}, B^{\sigma,s}_{2, 2} )}
  +   \|u\|_{\widetilde{L}^{r} (\mathbb{R}, B^{\sigma-1/2,s}_{r, 2} )}, \ \ \label{1workspaceX}
\end{align}	
where $r=2 + 4/(d-1)$ for $d \geqslant 3$; $r=2 + 4/d$ for $d =2$.
\end{thm}

\begin{rem} \label{Rmk1.3}
Theorem \ref{mainresult} needs several remarks.
\begin{itemize}

\item[(i)] In Theorem \ref{mainresult}, $s_0 \leqslant s$ comes from the scaling argument to the solution,  more precisely,  $s_0 = s_0(\|u_0\|_{E^{\sigma,s}}, \|u_1\|_{E^{\sigma-1,s}})$ for $d \geqslant 5$;  $s_0 = s_0(\|u_0\|_{E^{\sigma,s}}, \|u_1\|_{E^{\sigma-1,s}}, \, \varepsilon_0)$ for $d \leqslant 4$, where $\varepsilon_0 $ satisfies ${\rm supp} \, \widehat{u}_0, {\rm supp}\, \widehat{u}_0 \subset \mathbb{R}^d_I \setminus B(0, \varepsilon_0)$. It follows that the solution map is continuous in $E^{\sigma,s_0} \times E^{\sigma-1,s_0}$. Moreover, the solution is scattering in $E^{\sigma,s_0} \times E^{\sigma-1,s_0}$.

\item[(ii)] Using analogous way as in \cite{ChLuWa2023, ChWaWa2022}, we can give some initial data supported in $\mathbb{R}^d_I \cup (-\mathbb{R}^d_I)$ so that the solution maps in Theorems  \ref{mainresult} and \ref{mainresult2} are not smooth.

\item[(iii)]  Theorems \ref{mainresult} and \ref{mainresult2} contain rather rough initial data as examples. For instance, we consider spatial dimension $d=1$.  For any $k\in \mathbb{Z}_+$, $A\in \mathbb{C}$, let $\delta(x)$ be the Dirac measure and
$$
F(x) = \lim_{\varepsilon \downarrow 0} \frac{1}{x+  \mathrm{i} \varepsilon}  = \left( -\pi \mathrm{ i} \delta (x) + p.v. \frac{1}{x}\right)
$$
be the Sokhotski-Plemelj distribution (cf. \cite{So1873,Pl1908}). $F(x)$ plays a crucial role in the Sokhotski-Plemelj formula and it has a widely application in quantum field theory,  mechanics, robot-safety device system and numerical simulation (cf. \cite{Ke2022,MoLi1998,Pa1976,SaNa2014, Va2000}). Let $F^{(k)}$ be the $k$-order generalized derivative of $F$. Since $-\mathrm{i} \widehat{F}$ is the Heaviside function, one sees that
 $u_0 =e^{\mathrm{i} \epsilon x}   F^{(k)}(x)$ ($\forall \, k\in \mathbb{N}, \, \epsilon>0$) satisfies the conditions of Theorem \ref{mainresult}. Moreover,  Theorem \ref{mainresult} contains a class of initial values $u_0$ like
$$
u_0  = e^{\mathrm{i}\epsilon x} \sum^\infty_{k=0} \frac{(-{\rm i}\, \lambda)^k}{k!}   F^{(k)}(x), \ \  0< \lambda < |s|,
$$
 which is out of the control of any Sobolev spaces $H^\kappa$, $\kappa \in \mathbb{R}$.

 \item[(iv)] The results of Theorem \ref{mainresult2} also hold for the nonlinearity $f(u)= e^{u^2} u-u$ and $f(u) = \sin u-u$, the later nonlinearity is related to the well-known sine-Gordon equation $u_{tt} -\Delta u + \sin u =0$, cf. \cite{BrHaSt1978,Lu1977,Scot1969,Ru1970}.

\item[(v)] Theorem \ref{mainresult2} cannot cover the sinh-Gordon equation in one spatial dimension. However, if one takes $f(u) = \sinh u - u - u^3/6$, the results of Theorem \ref{mainresult2} also hold for $d=1$. Using the techniques as in \cite{ChLuWa2023}, it seems that we can also handle the sinh-Gordon in 1D and  NLKG in the case $\alpha <4/d$ to get the global existence and uniqueness for a class of very weak solutions.

\end{itemize}
\end{rem}

\subsection{Ideas in the proof of main results}

Now we indicate the crucial ideas how to obtain the solutions of NLKG in  $E^{\sigma,s}$.
Recall that one cannot solve nonlinear Klein-Gordon equation in supercritical Sobolev spaces $H^\sigma \times H^{\sigma-1}$, since the nonlinear estimate is problematic and the ill-posedness occurs in supercritical Sobolev spaces $H^\sigma \times H^{\sigma-1} $ with $\sigma <\sigma_c$.   However, $E^{\sigma,s}$ type spaces have some good algebraic structures when the frequency is localized in the first octant, so that the nonlinear estimates become available in $E^{\sigma,s}$ type spaces in the case $s<0$.

Let $\varphi \in E^{\sigma,s}$.  Observing that $\lim_{b\to -\infty} \|\varphi\|_{E^{\sigma,b}} =0$, we immediately get some $s_0 < s$ such that $\|\varphi\|_{E^{\sigma,s_0}} \leqslant \delta \ll 1$. So, we can solve NLKG and sinh-Gordon by the small data arguments as in \cite{Wa1998} and in Nakamura and Ozawa \cite{NaOz2001}. However, $s_0 = s_0(\varphi)$ will depends on both $\delta$ and $\varphi $ and it is unbounded for all $\varphi \in  E^{\sigma,s}$  (see Section \ref{lookingup}). In order to get $s_0= s_0(\|\varphi\|_{E^{\sigma,s}} )$, we use the scaling argument and
NLKG \eqref{NLKG} becomes Eq.~\eqref{NLKG2} after scalings. Observing the scaling solutions $u_\lambda (x,t)= \lambda^{2/\alpha} u(\lambda x, \lambda t)$ of Eq.~\eqref{NLKG2} in supercritical spaces $ H^\sigma$ (say, $\sigma > 1$), we have
\begin{align*}
\|(u_\lambda, \partial_t u_\lambda)|_{t=0}\|_{ H^\sigma \times H^{\sigma-1}}  \lesssim   \lambda^{\sigma -d/2+2/\alpha} \|(u_0, u_1)\|_{H^\sigma \times H^{\sigma-1}} \to 0, \ \ \lambda \to \infty.
\end{align*}
It follows that the scaling NLKG can have very small initial data in supercritical Sobolev spaces. The above observation is also adapted to the supercritical space $E^{\sigma,s} \times E^{\sigma-1,s}$ ($s<0 $). As $\lambda \to \infty$, $(u_\lambda, \partial_t u_\lambda)|_{t=0}$ will vanish in $E^{\sigma,s} \times E^{ \sigma-1,s}$  if $ (\widehat{u}_0, \widehat{u}_1) $ is supported in the first octant away from the origin. So, we can try to solve Eq.~\eqref{NLKG2} by assuming that initial data in $E^{\sigma,s} \times E^{\sigma-1, s}$ is sufficiently small.

Since Eq.~\eqref{NLKG2} contains a parameter $\lambda^2$, we have to consider the decaying estimates for the semi-group $U_\lambda(t):= e^{\mathrm{i} t \sqrt{\lambda^2 - \Delta}}$ by treating $\lambda $ as a variable parameter. By considering its frequency-localized version, we find that the decaying estimates of $U_\lambda(t)$ depend on $\lambda$, which become worse (better) in the lower (higher) frequency than those of $U_1(t)=e^{\mathrm{i} t \sqrt{I - \Delta}}$.

After obtaining the time decay of $U_\lambda(t)$, the Strichartz estimates of $U_\lambda(t)$ can be obtained by  $TT^*$ arguments. Since $U_\lambda(t)$  has different decaying estimates for the lower and higher frequency, it follows that the the Strichartz estimates of $U_\lambda(t)$ are also different for lower and higher frequencies.   On the other hand, in order to handle the supercritical data in $E^{\sigma,s}$, a new type of function spaces $B^{\sigma,s}_{p,q}$ and $\widetilde{L}^\gamma (\mathbb{R}, B^{\sigma,s}_{p,q})$ are involved in the Strichartz estimates, see Sections  \ref{Functspaces} and \ref{sectDecay}.

 One needs to further make the nonlinear mapping estimates in  $\widetilde{L}^\gamma (\mathbb{R}, B^{\sigma,s}_{p,q})$. Since the Strichartz estimates in the lower and higher frequency are different, we need to handle the nonlinearity in lower and higher frequency spaces in a separate way.

  Finally, the main results can be shown by contraction mapping arguments.

\subsection{Notations and organization}

 Throughout this paper, we denote by $L^p_x$ the Lebesgue space on $x\in \mathbb{R}^d$ and write
$ \|f\|_p :=\|f\|_{L^p(\mathbb{R}^d)}. $
For any function $g$ of $(x,t) \in \mathbb{R}^d \times \mathbb{R}$, we denote
$$
\|g\|_{L^\gamma_t L^p_x} = \|\|g\|_{L^p(\mathbb{R}^d) }\|_{L^\gamma_t} = \left(\int_{\mathbb{R}} \|g(\cdot,\, t)\|^\gamma_p dt \right)^{1/\gamma}.
$$
Let us write $\langle \nabla\rangle^s = \mathscr{F}^{-1} \langle \xi\rangle^s \mathscr{F}$, $2^{s|\nabla|}  = \mathscr{F}^{-1} 2^{s|\xi|} \mathscr{F}$, where we
write $\langle \xi\rangle=(1+\xi^2_1+...+\xi^2_d)^{1/2}$ and $|\xi|_\infty = \max_{1\leqslant i\leqslant d} |\xi_i|$  for $\xi=(\xi_1,...,\xi_d) \in \mathbb{R}^d$.  We will use the following notations. $C\gtrsim  1 $ and $ c\lesssim  1$ will denote constants which can be different at different places, we will use $A\lesssim B$ to denote   $A\leqslant CB$; $A\sim B$ means that $A\lesssim B$ and $B\lesssim A$, $A\vee B = \max(A,B), \, A\wedge B =\min (A,B)$. $a+ = a+ \varepsilon$ for $0< \varepsilon \ll 1$.   We denote by  $\mathscr{F}^{-1}f$ the inverse Fourier transform of $f$.  For any $1\leqslant p \leqslant \infty$,   $l^p$  stands for the sequence Lebesgue space.

The paper is organized as follows. In Section \ref{Functspaces} we introduce the notions of function spaces $B^{\sigma,s}_{p,q}$, $F^{\sigma,s}_{p,q}$, $\widetilde{L}^\gamma (\mathbb{R}, B^{\sigma,s}_{p,q})$ and  $\widetilde{L}^\gamma (\mathbb{R}, F^{\sigma,s}_{p,q})$, then consider their scaling properties, inclusion relations with Besov and Triebel spaces $B^{\sigma}_{p,q}$ and $F^{\sigma}_{p,q}$, respectively.   Section  \ref{sectDecay} is devoted to consider the time-decaying and the Strichartz estimates for the semi-group $ e^{\mathrm{i} t \sqrt{\lambda^2 - \Delta}}$. By establishing some nonlinear estimates, we prove our main results in Section \ref{GlobalNLKG}. Finally, in Section \ref{lookingup} we discuss the cases that the scaling argument is invalid for the lower power nonlineaities.

\section{Function spaces $B^{\sigma,s}_{p,q}$ and $F^{\sigma,s}_{p,q}$} \label{Functspaces}

\subsection{Function spaces $B^{\sigma,s}_{p,q}$}

Let $s, \sigma\in \mathbb{R}$, $1\leqslant p,q\leqslant \infty$. We introduce the following function spaces
\begin{align} \label{Fspace1}
B^{ \sigma,s}_{p,q} & = \left\{f \in \mathscr{S}'_1:  \|f\|_{  B^{\sigma, s}_{p,q} }=  \left(\sum^\infty_{j=0} 2^{\sigma jq}   \|2^{s|\nabla|}\triangle_j f  \|^q_{p} \right)^{1/q} <\infty   \right\}; \\
\widetilde{L}^{\gamma} (\mathbb{R},\,  B^{ \sigma,s}_{p,q}) & = \left\{f \in \mathscr{S}'(\mathbb{R},\, \mathscr{S}'_1):  \|f\|_{\widetilde{L}^{\gamma} (\mathbb{R},\,  B^{\sigma, s}_{p,q}) }=  \left(\sum^\infty_{j=0} 2^{\sigma jq}   \|2^{s|\nabla|}\triangle_j f  \|^q_{L^\gamma_t L^p_x} \right)^{1/q} <\infty   \right\}.   \label{gammaFspace1}
\end{align}
If $s=0$, then $ B^{ \sigma,0}_{p,q} = B^{ \sigma}_{p,q}$ is the Besov space, cf. \cite{BeLo1976,Tr1983,WaHuHaGu2011}. In the case $s>0$, the function spaces $B^{ \sigma,s}_{2,2}$ and $\widetilde{L}^{\gamma} (\mathbb{R},\,  B^{\sigma, s}_{p,q})$ for $s=c(t)$ related the Gevrey class are hidden in some literatures (see \cite{FoTe1989,BaBiTa2012} for instance). However, if $s<0$, $B^{ \sigma,s}_{p,q}$ and $\widetilde{L}^{\gamma} (\mathbb{R},\,  B^{\sigma, s}_{p,q})$ become  rather rough and it seems that such a kind of spaces are new ones and they are not included by any known function spaces.

We will use the multiplier on $L^p$.  We denote by $M_p$ the multiplier space on $L^p$, i.e.,
$$
\|m\|_{M_p} = \sup_{f\in \mathscr{S}, \,\|f\|_p=1} \|\mathscr{F}^{-1} m \mathscr{F} f\|_p.
$$
The following Bernstein's multiplier estimate is well-known. For $L>d/2$, $L\in \mathbb{N}$, $\theta= d/2L$ (cf. \cite{BeLo1976, WaHuHaGu2011}),
\begin{align} \label{multiplier}
\|m\|_{M_p} \lesssim \|\mathscr{F}^{-1} m\|_1 \lesssim \|m\|^{1-\theta}_2 \|m\|^{\theta}_{\dot H^L}.
\end{align}
First, we have
\begin{lemma} \label{Embeddingequal}
Let $s , \sigma  \in \mathbb{R}$.   Then we have $B^{ \sigma,s}_{2,2} =E^{\sigma,s}$ with equivalent norms.
\end{lemma}
\noindent {\bf Proof.} By Plancherel's identity,  $\psi+ \sum_{j\in \mathbb{N}} \varphi_j=1 $ and the supported set property of $\psi, \, \varphi_j$, we immediately have the result, as desired. $\hfill\Box$

\begin{lemma} \label{Embedding1}
Let $s_i, \sigma_i \in \mathbb{R}$, $1\leqslant p_i ,q_i \leqslant \infty$, $i=1,2$.   Then we have
\begin{align} \label{embedding1}
  B^{ \sigma_1,s_1}_{p_1,q_1} \subset  B^{ \sigma_2,s_2}_{p_2,q_2}, \ \ \ \mbox{ if } s_1> s_2, \ p_1 \leqslant p_2.
\end{align}
In particular,
\begin{align}
&  B^{ \sigma_1}_{p_1,q_1} \subset  B^{ \sigma_2,s}_{p_2,q_2}, \ \ \ \mbox{ if } s<0, \ p_1 \leqslant p_2, \label{embedding2} \\
& H^\kappa \subset E^{\sigma,s}, \  \ \ \  \mbox{ if } s<0, \ \kappa,\sigma \in \mathbb{R}. \label{embedding2a}
\end{align}
\end{lemma}
\noindent {\bf Proof.} Let us recall that (cf. \cite{BeLo1976,Tr1983,WaHuHaGu2011})
$$
\|\triangle_j f\|_{p_2} \lesssim 2^{jd(1/p_1-1/p_2)}  \|\triangle_j f\|_{p_1}.
$$
It follows that
$$
\|2^{s_2|\nabla|}\triangle_j f\|_{p_2} \lesssim 2^{jd(1/p_1-1/p_2)}  \| 2^{s_2|\nabla|} \triangle_j f\|_{p_1}.
$$
Using the almost orthogonality of $\varphi_j= \varphi(2^{-j} \, \cdot)$, one has that for $j\geqslant 1$,
\begin{align*}
  \| 2^{s_2|\nabla|} \triangle_j f\|_{p_1} & \lesssim  \sum^1_{\ell=-1} \|\mathscr{F}^{-1} \varphi_{j+\ell} 2^{(s_2-s_2)|\xi|}\|_1  \| 2^{s_1|\nabla|} \triangle_j f\|_{p_1}.
\end{align*}
Using the multiplier estimate \eqref{multiplier}, we have for some $c>0$,
\begin{align*}
    \|\mathscr{F}^{-1} \varphi_{j} 2^{(s_2-s_2)|\xi|}\|_1
    =  \|\mathscr{F}^{-1} \varphi  2^{(s_2-s_2)|2^{2j}\xi|}\|_1    \lesssim 2^{c(s_2-s_1) 2^{2j}}.
\end{align*}
So, we have
\begin{align*}
  \| 2^{s_2|\nabla|} \triangle_j f\|_{p_1}
   & \lesssim 2^{c(s_2-s_1) 2^{2j}}  \| 2^{s_1|\nabla|} \triangle_j f\|_{p_1}.
\end{align*}
for $j\geqslant 1$. It is easy to see that the above inequality also holds for $j=0$.  It follws that
\begin{align*}
\|f\|_{B^{ \sigma_2,s_2}_{p_2,1}} &  =\sum^\infty_{j=0} 2^{\sigma_2 j}  \| 2^{s_2|\nabla|} \triangle_j f\|_{p_2} \\
     & \lesssim \sum^\infty_{j=0} 2^{(\sigma_2-\sigma_1 +d(1/p_1-1/p_2)) j}  2^{c(s_2-s_1) 2^{2j}} 2^{\sigma_1j} \| 2^{s_1|\nabla|} \triangle_j f\|_{p_1} \\
     & \lesssim  \|f\|_{B^{ \sigma_1,s_1}_{p_1,\infty}} \sum^\infty_{j=0} 2^{(\sigma_2-\sigma_1 +d(1/p_1-1/p_2)) j}  2^{c(s_2-s_1) 2^{2j}}.
\end{align*}
Noticing that $s_2-s_1 <0$, one sees that for any $\sigma >0$
$$
\sup_{x\geqslant 0} x^\sigma  \  2^{c(s_2-s_1) x^2} < \infty.
$$
This implies that
$$
\sum^\infty_{j=0} 2^{(\sigma_2-\sigma_1 +d(1/p_1-1/p_2)) j}  2^{c(s_2-s_1) 2^{2j}}
$$
is a convergent series. Hence, we have $ B^{ \sigma_1,s_1}_{p_1,\infty} \subset  B^{ \sigma_2,s_2}_{p_2,1}$. Applying the embedding $l^1 \subset l^q \subset l^\infty$ for any $q\in [1,\infty]$, we have the result, as desired. \eqref{embedding2a} is a direct consequence of \eqref{embedding2} by taking into account $H^\kappa =B^\kappa_{2,2}$ and $E^{\sigma,s} =B^{\sigma,s}_{2,2}$.  $\hfill\Box$\\

By \eqref{embedding2}, we see that  $B^{ \sigma_1}_{p,q} \subset B^{ \sigma_2,s}_{p,q}$  for any $s<0$ and it is independent of $\sigma_1,\, \sigma_2 \in \mathbb{R}$, which means that $\cup_{\sigma \in \mathbb{R}} B^{ \sigma}_{p,q}$ is a subset of $B^{ \sigma_2,s}_{p,q}$ for any $\sigma_2 >0$.

\begin{lemma} \label{Isomorphism1}
Let $s, s', \sigma  \in \mathbb{R}$, $1\leqslant p ,q  \leqslant \infty$.   Then
\begin{align} \label{isomorph1}
2^{(s'-s)|\nabla|}: \,   B^{ \sigma, s'}_{p,q} \to  B^{ \sigma,s}_{p,q}
\end{align}
is an isometric mapping. In particular,
 \begin{align} \label{isomorph2}
2^{s|\nabla|}: \,   B^{ \sigma, s}_{p,q} \to  B^{ \sigma}_{p,q}
\end{align}
is an isometric mapping.
\end{lemma}
\noindent {\bf Proof.} Observing that
$$
\|2^{(s'-s)|\nabla|} f\|_{ B^{ \sigma,s}_{p,q}} = \|f\|_{B^{ \sigma, s'}_{p,q}},
$$
we have the result, as desired. $\hfill\Box$\\

Now we consider the scaling property on $B^{\sigma,s}_{p,q}$. Recall that Lebesgue spaces have a very simple scaling property $\|f(\lambda\, \cdot)\|_p = \lambda^{-d/p} \|f\|_p$. If $s<0$, we will show that the scaling on $B^{\sigma,s}_{p,q}$  is very close to $L^p$. More precisely, the scaling on $B^{\sigma,s}_{p,q}$ in the lower frequency is almost the same as in $L^p$ up to a logarithmic lose, while its scaling in the higher frequency enjoys an exponential decaying behavior as $\lambda \to \infty$. Indeed, we have the following

\begin{lemma} \label{Scaling1}
		Let $s< 0$, $\sigma \in \mathbb{R}$, $1 \leqslant p,q \leqslant \infty$. We have the following results.

\begin{itemize}
\item[\rm (i)]
For any $\epsilon >0$, there exists $\lambda_0 =\lambda_0(\epsilon, s,\sigma, q) >1$  such that for any $\lambda > \lambda_0$ and $f\in B^{\sigma,s}_{p,q}$,
\begin{align*}
& 	\|f(\lambda\, \cdot) \|_{B^{\sigma,s}_{p,q}} \lesssim   \lambda^{-d/p + \epsilon}   \|f\|_{B^{\sigma,s}_{p,q} }
\end{align*}
and the above inequality holds for $\epsilon=0$ if $\sigma <0$.

\item[\rm (ii)] Assume that ${\rm supp }\,\widehat{f} \subset \{\xi: |\xi|\geqslant \varepsilon_0\}$ for some $\varepsilon_0>0$. Then for $\lambda \gg 1$,
\begin{align*}
& 	\|f(\lambda\, \cdot) \|_{B^{\sigma,s}_{p,q}} \lesssim \lambda^{\sigma-d/p } 2^{s \lambda \varepsilon_0/3} \|f\|_{B^{\sigma,s}_{p,q} }, \ \ \sigma > 0; \\
& 			\|f(\lambda\, \cdot) \|_{B^{ \sigma,s}_{p,q} } \lesssim \lambda^{-d/p} 2^{s \lambda \varepsilon_0/3} \|f\|_{B^{ \sigma,s}_{p,q} }, \ \ \sigma < 0 ; \\
& 	\|f(\lambda\, \cdot) \|_{B^{\sigma,s}_{p,q}} \lesssim \lambda^{-d/p }(\ln \lambda)^{1/q} 2^{s \lambda \varepsilon_0/3} \|f \|_{B^{\sigma,s}_{p,q} }, \ \ \sigma = 0.
\end{align*}

\item[\rm (iii)] We have for $\lambda \gg 1$,
\begin{align*}
& \|f(\lambda^{-1}\, \cdot)\|_{B^{\sigma, \lambda s}_{p,q}  } \lesssim \lambda^{ d/p - \sigma}  \|\varphi \|_{B^{\sigma,  s}_{p,q}  }, \ \ \sigma < 0; \\
& \|f(\lambda^{-1}\, \cdot)\|_{B^{\sigma, \lambda s}_{p,q}  } \lesssim \lambda^{ d/p  }  \|f\|_{B^{\sigma,  s}_{p,q}  }, \ \ \sigma > 0; \\
& \|f(\lambda^{-1}\, \cdot)\|_{B^{\sigma, \lambda s}_{p,q}  } \lesssim \lambda^{ d/p  } (\log_2\lambda)^{1/q'} \|f \|_{B^{\sigma,  s}_{p,q}  }, \ \ \sigma = 0.
\end{align*}
\end{itemize}
\end{lemma}
\noindent {\bf Proof.} {\it Step 1.} We show the result of (i).  Let $\lambda \gg 1$  and $j_\lambda $ be the integer part of $\log_2 \lambda$. We have $\lambda = 2^{j_\lambda} c$ for some $1\leqslant c<2$.
By the definition,
\begin{align*}
 \|f(\lambda\, \cdot) \|_{B^{\sigma,s}_{p,q}} \leqslant        \|2^{s|\nabla|}\triangle_0 f (\lambda\, \cdot) \|_{p}
  & +   \left(\sum_{j\geqslant 1} 2^{\sigma jq}   \|2^{s|\nabla|}\triangle_j f(\lambda\, \cdot)  \|^q_{L^p(\mathbb{R}^d)} \right)^{1/q}:= I+II.
\end{align*}
It is easy to see that
\begin{align}
& \|2^{s|\nabla|}\triangle_0 f (\lambda\, \cdot) \|_{p} =  \lambda^{-d/p } \|\mathscr{F}^{-1} [2^{s\lambda |\xi|} \psi(\lambda \xi) \widehat{f}(\xi)]\|_p; \label{0Fspace2+} \\
& \|2^{s|\nabla|}\triangle_j f (\lambda\, \cdot) \|_{p} =  \lambda^{-d/p } \|\mathscr{F}^{-1} [2^{s\lambda |\xi|} \varphi(2^{-(j-j_\lambda)} c\xi) \widehat{f}(\xi)]\|_p. \label{Fspace2+}
\end{align}
Since $\lambda \gg 1$, by the invariance of $\|\widehat{f(\mu \, \cdot)}\|_1$ with respect to $\mu>0$, we have
\begin{align*}
I
& \leqslant  \lambda^{-d/p } \|\mathscr{F}^{-1}  2^{s(\lambda-1) |\xi|} \psi(\lambda \xi)  \|_1 \|\Delta_0 2^{s|\nabla|} f\|_p  \nonumber \\
& \leqslant  \lambda^{-d/p } \|\mathscr{F}^{-1}  2^{- |\xi|}\|_1 \|\mathscr{F}^{-1}\psi  \|_1 \|\Delta_0 2^{s|\nabla|} f\|_p  \nonumber \\
& \lesssim \lambda^{-d/p }   \|\Delta_0 2^{s|\nabla|} f\|_p.
\end{align*}
Using \eqref{Fspace2+},  one has that
\begin{align*}
II
  \leqslant \lambda^{\sigma-d/p} \left(\sum_{j\geqslant 1- j_\lambda } 2^{j \sigma q}   \|\mathscr{F}^{-1} [2^{s\lambda |\xi|} \varphi(2^{-j} c\xi) \widehat{f} ] \|^q_p \right)^{1/q}.
\end{align*}
Let $\varepsilon_0 \gg  \lambda^{-1}  $. We decompose $f$ into
 $$
 \widehat{f} = ( 1-\psi(2 \varepsilon_0^{-1}\, \cdot)) \, \widehat{f} + \psi(2 \varepsilon_0^{-1}\, \cdot) \widehat{f} := \widehat{f}_1 + \widehat{f}_2.
 $$
Then we have
\begin{align*}
II
  \leqslant \lambda^{\sigma-d/p} \sum_{k=1,2}  \left(\sum_{j\geqslant 1- j_\lambda } 2^{j \sigma q}   \|\mathscr{F}^{-1} [2^{s\lambda |\xi|} \varphi(2^{-j} c\xi) \widehat{f}_k ] \|^q_p \right)^{1/q}: = III+IV.
\end{align*}
By the support property of $\widehat{f}_1$, we have
\begin{align*}
III
 \leqslant  &  \lambda^{\sigma-d/p } \left(\sum_{j\geqslant 0 } 2^{j \sigma q}   \|\mathscr{F}^{-1} [2^{s\lambda |\xi|} \varphi(2^{-j} c\xi) \widehat{f}_1 ] \|^q_p \right)^{1/q} \nonumber\\
    &  + \lambda^{\sigma-d/p } \left(\sum_{\log_2 \varepsilon_0 -5 \leqslant  j <0 } 2^{j \sigma q}   \|\mathscr{F}^{-1} [2^{s\lambda |\xi|} \varphi(2^{-j} c\xi) \widehat{f}_1 ] \|^q_p \right)^{1/q}
    :=   V +VI.
\end{align*}
Write $\widetilde{\varphi}= \varphi_{-1}+ \varphi + \varphi_{1}$. Let $j> \log_2 \varepsilon_0 +5$.   It follows from the invariance of $\|\widehat{f(\mu \, \cdot)}\|_1$ with respect to $\mu>0$ and the multiplier estimate \eqref{multiplier} that
\begin{align*}
   \|\mathscr{F}^{-1} [2^{s\lambda |\xi|} \varphi(2^{-j} c\xi) \widehat{f}_1 ] \|_p
& \leqslant  \|\mathscr{F}^{-1} [2^{s(\lambda-1) |\xi|} \widetilde{\varphi}(2^{-j} c\xi)]  \|_1  \|\mathscr{F}^{-1} [2^{s |\xi|} \varphi(2^{-j} c\xi) \widehat{f}(\xi)]\|_p \nonumber\\
 &  \lesssim    2^{s\lambda \varepsilon_0   }   \|\mathscr{F}^{-1} [2^{s |\xi|} \varphi(2^{-j} c\xi) \widehat{f}(\xi)]\|_p.
\end{align*}
For   $j \in [\log_2 \varepsilon_0 -5, \, \log_2 \varepsilon_0 +5] $, in view of the invariance of $\|\widehat{f(\mu \, \cdot)}\|_1$ with respect to $\mu>0$ and \eqref{multiplier}, we have
\begin{align*}
&   \|\mathscr{F}^{-1} [2^{s\lambda |\xi|} \varphi(2^{-j} c\xi) \widehat{f}_1 ] \|_p \nonumber\\
& \leqslant  \|\mathscr{F}^{-1} [2^{s(\lambda-1)\varepsilon_0  |\xi|} \widetilde{\varphi}(2^{-j}\varepsilon_0  c\xi) (1- \psi(2  \xi))]  \|_1  \|\mathscr{F}^{-1} [2^{s |\xi|} \varphi(2^{-j} c\xi) \widehat{f}(\xi)]\|_p \nonumber\\
 &  \lesssim    2^{s\lambda \varepsilon_0/3   }   \|\mathscr{F}^{-1} [2^{s |\xi|} \varphi(2^{-j} c\xi) \widehat{f}(\xi)]\|_p.
\end{align*}
If $j <0$, by the   support property of $\varphi(2^{-j} c\cdot)$,  we have
\begin{align} \label{lowfretozero}
  \|\mathscr{F}^{-1} [2^{s |\xi|} \varphi(2^{-j} c\xi) \widehat{f}(\xi)]\|_p \lesssim  \|\triangle_0 2^{s |\nabla|} {f} \|_p.
\end{align}
Inserting the above estimate into $IV$ and $V$, we have
\begin{align*}
  V   & \lesssim \lambda^{\sigma-d/p } 2^{s \lambda \varepsilon_0 }  \|f\|_{B^{\sigma,s}_{p,q}} , \\
 VI   & \lesssim \lambda^{\sigma-d/p } 2^{s \lambda \varepsilon_0/3 }  \|\triangle_0 2^{s |\nabla|} {f} \|_p \left(\sum_{\log_2 \varepsilon_0 -5 \leqslant  j <0 } 2^{j \sigma q}    \right)^{1/q}.
\end{align*}
By  a straightforward calculation,
\begin{align*}
\left(\sum_{\log_2 \varepsilon_0 -5 \leqslant  j <0 } 2^{j \sigma q}    \right)^{1/q} \lesssim
\left\{
\begin{array}{ll}
1, & \sigma>0; \\
(-\log_2 \varepsilon_0)^{1/q}, & \sigma =0; \\
\varepsilon^\sigma_0, & \sigma <0.
\end{array}
\right.
\end{align*}
It follows that
\begin{align*}
 VI   & \lesssim   2^{s \lambda \varepsilon_0/3 } F(\lambda, \varepsilon_0) \|f\|_{B^{\sigma,s}_{p,q}},
\end{align*}
where
\begin{align*}
F(\lambda, \varepsilon_0) =
\left\{
\begin{array}{ll}
\lambda^{\sigma-d/p } , & \sigma>0; \\
\lambda^{ -d/p } (-\log_2 \varepsilon_0)^{1/q}, & \sigma =0; \\
\lambda^{\sigma-d/p } \varepsilon^\sigma_0, & \sigma <0.
\end{array}
\right.
\end{align*}
Combining the estimate of $V$ and $VI$, we have
\begin{align*}
 III   & \lesssim   2^{s \lambda \varepsilon_0/3 } F(\lambda, \varepsilon_0) \|f\|_{B^{\sigma,s}_{p,q}}.
\end{align*}

Next, we estimate $IV$. In view of \eqref{lowfretozero}, one easily sees that
\begin{align*}
 IV   & \lesssim \lambda^{\sigma-d/p }    \|\triangle_0 2^{s |\nabla|} {f} \|_p \left(\sum_{1- j_\lambda \leqslant j \leqslant \log_2 \varepsilon_0 +5} 2^{j \sigma q}    \right)^{1/q}.
\end{align*}
One can calculate that
By  a straightforward calculation,
\begin{align*}
\left(\sum_{1- j_\lambda \leqslant j \leqslant \log_2 \varepsilon_0 +5  } 2^{j \sigma q}    \right)^{1/q} \lesssim
\left\{
\begin{array}{ll}
\varepsilon^\sigma_0, & \sigma>0; \\
(\log_2 \lambda)^{1/q}, & \sigma =0; \\
\lambda^{-\sigma}, & \sigma <0.
\end{array}
\right.
\end{align*}
It follows that
\begin{align*}
 IV   & \lesssim    G(\lambda, \, \varepsilon_0)  \|f\|_{B^{\sigma,s}_{p,q}},
\end{align*}
where
\begin{align*}
G (\lambda, \varepsilon_0) =
\left\{
\begin{array}{ll}
\lambda^{\sigma-d/p } \varepsilon^\sigma_0 , & \sigma>0; \\
\lambda^{ -d/p } (\log_2 \lambda)^{1/q}, & \sigma =0; \\
\lambda^{-d/p } , & \sigma <0.
\end{array}
\right.
\end{align*}
For any $0<\epsilon <\sigma$, taking $\varepsilon_0 = \lambda^{-1+\epsilon/\sigma}$, one can find a $\lambda_0: = \lambda_0 (\epsilon, \sigma, q, s)>1$ such that
$$
G (\lambda, \varepsilon_0) + 2^{s \lambda \varepsilon_0/3 }  F(\lambda, \varepsilon_0)   \leqslant  \lambda^{-d/p +\epsilon}.
$$
Combining the estimates of $I$, $III$ and $IV$, we have shown (i).

{\it Step 2.} If ${\rm supp }\, \widehat{f} \subset \{\xi: |\xi|\geqslant \varepsilon_0\}$ for some $\varepsilon_0>0$, then we have $\widehat{f}_2 =0$ in the proof of Step 1. We have the result by the estimate of $III$.

{\it Step 3.} We prove (iii).  We have
\begin{align*}
 \|f(\lambda^{-1}\, \cdot) \|_{B^{\sigma,s\lambda}_{p,q}} \leqslant & \left(\sum_{j\geqslant 1} 2^{\sigma jq}   \|2^{s\lambda |\nabla|}\triangle_j f (\lambda^{-1}\, \cdot) \|^q_{L^p(\mathbb{R}^d)} \right)^{1/q} \\
 & \ \ \  +  \|2^{s\lambda|\nabla|}\triangle_0 f(\lambda^{-1}\, \cdot)  \|_p = A+ B.
\end{align*}
Using the same way as in \eqref{Fspace2+}, we have for $j\geqslant 1$,
\begin{align} \label{Fspace2++}
 \|2^{\lambda s|\nabla|}\triangle_j f (\lambda^{-1}\, \cdot) \|_{p} =  \lambda^{ d/p } \|\mathscr{F}^{-1} [2^{s |\xi|} \varphi(2^{-(j+j_\lambda)} c^{-1}\xi) \widehat{f}(\xi)]\|_p.
\end{align}
It follows that
\begin{align*}
 A & \leqslant \lambda^{ d/p }  \left(\sum_{j\geqslant 1} 2^{\sigma jq}   \|\mathscr{F}^{-1} [2^{s |\xi|} \varphi(2^{-(j+j_\lambda)} c^{-1}\xi) \widehat{f}(\xi)]\|^q_p   \right)^{1/q} \\
 & \lesssim \lambda^{ d/p -\sigma }  \left(\sum_{j\geqslant j_\lambda } 2^{\sigma jq}   \|\mathscr{F}^{-1} [2^{s |\xi|} \varphi(2^{-j} c^{-1}\xi) \widehat{f}(\xi)]\|^q_p   \right)^{1/q} \\
  & \lesssim \lambda^{ d/p -\sigma }  \|f \|_{B^{\sigma,s}_{p,q}}.
\end{align*}
For the estimate of $B$, we have
 \begin{align*}
 B & \leqslant \lambda^{ d/p }      \|2^{s |\nabla|} \mathscr{F}^{-1} \psi(\lambda^{-1} \xi) \widehat{f} \| _p  \\
 & \lesssim \lambda^{ d/p }  \sum_{j\leqslant j_\lambda +C}    \|2^{s |\nabla|} \triangle_j f \| _p \|\mathscr{F}^{-1} \psi(\lambda^{-1} \xi)\|_1 \\
  & \lesssim \lambda^{ d/p }  \sum_{j\leqslant j_\lambda +C}    \|2^{s |\nabla|} \triangle_j f \| _p \\
  & \lesssim \lambda^{ d/p } \|f \|_{B^{\sigma,s}_{p,q}}   \left(\sum_{j\leqslant j_\lambda +C}  2^{-j\sigma q'} \right)^{1/q'},
  \end{align*}
whence, we have the result, as desired. $\hfill \Box$\\

For the scaling in the space $ \widetilde{L} ^\gamma (\mathbb{R}; B^{\sigma,s}_{p,q} )$, we have

\begin{lemma} \label{Scaling2}
Let $s< 0$, $\sigma \in \mathbb{R}$. Denote $(D_\lambda f)(x,t) = f(x/\lambda, t/\lambda)$.  Then  for any $\lambda> 1$, we have
\begin{align*}
& \|D_\lambda f\|_{\widetilde{L} ^\gamma (\mathbb{R}; B^{\sigma, \lambda s}_{p,q})  } \lesssim \lambda^{1/\gamma + d/p - \sigma}  \|f \|_{ \widetilde{L} ^\gamma (\mathbb{R}; B^{\sigma,  s}_{p,q} )  }, \ \ \sigma < 0; \\
& \|D_\lambda f\|_{\widetilde{L} ^\gamma (\mathbb{R}; B^{\sigma, \lambda s}_{p,q})  }  \lesssim \lambda^{1/\gamma +  d/p  }  \|f\|_{ \widetilde{L} ^\gamma (\mathbb{R}; B^{\sigma,  s}_{p,q} )  }, \ \ \sigma > 0; \\
& \|D_\lambda f\|_{\widetilde{L} ^\gamma (\mathbb{R}; B^{\sigma, \lambda s}_{p,q})  }  \lesssim \lambda^{1/\gamma + d/p  } (\log_2\lambda)^{1/q'} \|f \|_{ \widetilde{L} ^\gamma (\mathbb{R}; B^{\sigma,  s}_{p,q} )  }, \ \ \sigma = 0.
\end{align*}
\end{lemma}
\noindent {\bf Proof.} Let us follow Lemma \ref{Scaling1}. By definition,
\begin{align*}
 \|D_\lambda f \|_{ \widetilde{L} ^\gamma (\mathbb{R}; B^{\sigma,s\lambda}_{p,q})} \leqslant  \left(\sum_{j\geqslant 1} 2^{\sigma jq}   \|2^{s\lambda |\nabla|}\triangle_j D_\lambda f  \|^q_{L^\gamma_t L^p_x} \right)^{1/q}
   +  \|2^{s\lambda|\nabla|}\triangle_0 D_\lambda f  \|_{L^\gamma_t L^p_x}=V+VI .
\end{align*}
Observing the identity
\begin{align*}
 \|2^{\lambda s|\nabla|}\triangle_j D_\lambda f  \|_{L^\gamma_tL^p_x} =  \lambda^{1/\gamma + d/p } \|\mathscr{F}^{-1} [2^{s |\xi|} \varphi(2^{-(j+j_\lambda)} c^{-1}\xi) \widehat{f}]\|_{L^\gamma_tL^p_x},
\end{align*}
one can use the same way as in the estimate of $III$ in Lemma \ref{Scaling1} to obtain  that
$$
V \lesssim \lambda^{1/\gamma + d/p -\sigma }  \|f \|_{ \widetilde{L} ^\gamma (\mathbb{R}; B^{\sigma,  s}_{p,q} )  }
$$
The estimate of $VI$ follows a similar way as in the estimate of $IV$ in Lemma \ref{Scaling1}. $\hfill\Box$

\begin{rem}
Let $\mathbb{A} \subset \mathbb{Z}_+ = \mathbb{N} \cup \{0\}$, we write
\begin{align} \label{fresolutionspace}
& \|u\|_{\widetilde{L} ^\gamma (\mathbb{R}; B^{\sigma,s}_{p,q} (\mathbb{A})) } =  \left(\sum_{j\in \mathbb{A}} 2^{\sigma jq}   \|2^{s|\nabla|}\triangle_j u  \|^q_{L^\gamma_t L^p_x (\mathbb{R} \times\mathbb{R}^d)} \right)^{1/q}.
\end{align}
Following the proof of Lemmas \ref{Scaling1} and \ref{Scaling2}, we have shown that
\begin{align*}
& \|D_\lambda f\|_{\widetilde{L} ^\gamma (\mathbb{R}; B^{\sigma, \lambda s}_{p,q}(\mathbb{N}))  } \lesssim \lambda^{1/\gamma + d/p - \sigma}  \|f \|_{ \widetilde{L} ^\gamma (\mathbb{R}; B^{\sigma,  s}_{p,q} (\mathbb{Z}^c_\lambda))  }; \\
& \|2^{s\lambda|\nabla|}\triangle_0 D_\lambda f  \|_{L^\gamma_t L^p_x}  \lesssim \lambda^{1/\gamma +  d/p+  }  \|f\|_{ \widetilde{L} ^\gamma (\mathbb{R}; B^{\sigma,  s}_{p,q} (\mathbb{Z}_\lambda) )  }, \ \ \sigma \geq 0; \\
& \|2^{s\lambda|\nabla|}\triangle_0 D_\lambda f  \|_{L^\gamma_t L^p_x}  \lesssim \lambda^{1/\gamma +  d/p -\sigma  }  \|f\|_{ \widetilde{L} ^\gamma (\mathbb{R}; B^{\sigma,  s}_{p,q} (\mathbb{Z}_\lambda) )  }, \ \ \sigma < 0;
\end{align*}
where $\mathbb{Z}_\lambda  = \{j\geqslant 0: \, 2^j \lesssim  \lambda \}, \ \mathbb{Z}^c_\lambda  = \{j\geqslant 0: \, 2^j \gtrsim  \lambda \}$.
\end{rem}

\subsection{Function spaces $F^{\sigma,s}_{p,q}$ and $H^{\sigma,s}_{p}$}

Let $s, \sigma\in \mathbb{R}$, $1\leqslant p,q\leqslant \infty$.  We introduce the following
\begin{align} \label{TFspace1}
&  F^{ \sigma,s}_{p,q} = \{f \in \mathscr{S}'_1:  \|f\|_{F^{\sigma, s}_{p,q} } <\infty \}, \ \
  \|f\|_{F^{ \sigma,s}_{p,q}   } = \left\| \left(\sum^\infty_{j=0} 2^{\sigma jq}    |\triangle_j 2^{s|\nabla|} f  |^q \right)^{1/q}\right\|_{L^p }, \\
&  H^{ \sigma,s}_{p} = \{f \in \mathscr{S}'_1:  \|f\|_{H^{\sigma, s}_{p} } <\infty \},  \ \
  \|f\|_{H^{ \sigma,s}_{p}   } = \left\|2^{s|\nabla|} (I-\Delta)^{\sigma/2} f  \right\|_{L^p }. \label{TFspace2}
\end{align}

\begin{lemma} \label{Embedding2}
Let $s, s_i, \sigma_i \in \mathbb{R}$, $1\leqslant p,  p_i ,q_i \leqslant \infty$, $i=1,2$.   Then we have
\begin{align} \label{tembedding2}
 & B^{ \sigma,s}_{p, p\wedge q } \subset  F^{ \sigma,s}_{p,q} \subset  B^{ \sigma,s}_{p, p\vee q }, \\
 & F^{ \sigma_1,s_1}_{p_1,q_1} \subset  F^{ \sigma_2,s_2}_{p_2,q_2}, \ \ \ \mbox{ if } s_1> s_2, \ p_1 \leqslant p_2. \label{tembedding1}
\end{align}
\end{lemma}
\noindent {\bf Proof.} By $l^{q_1} \subset l^{q_2}$ for $q_1\leqslant q_2$ and Minkowski's inequality, we have \eqref{tembedding2}. Using Lemma \ref{Embedding1} and \eqref{tembedding2}, we have
\begin{align} \label{tembedding3}
 F^{ \sigma_1,s_1}_{p_1,q_1}  \subset  B^{ \sigma_1,s_1}_{p_1,\infty} \subset  B^{ \sigma_2,s_2}_{p_2,1} \subset  F^{ \sigma_2,s_2}_{p_2,q_2}
\end{align}
if $ s_1> s_2, \ p_1 \leqslant p_2$. $\hfill\Box$

Now we consider the scaling property on $F^{\sigma,s}_{p,q}$. First, let us recall a multiplier estimate in $L^p(l^q)$ (cf. \cite{Tr1983}).
\begin{lemma} \label{Multiplier}
Let $1\leqslant p,q \leqslant \infty$, $\{\Omega_k\}^\infty_{k=0}$ be a sequence of compact subset of $\mathbb{R}^d$, $d_k>0$ be the diameter of $\Omega_k$, $k\in \mathbb{Z}_+$. Let $\kappa > d/2 + d/(p\wedge q)$. Then we have
$$
\|\mathscr{F}^{-1} M_k \mathscr{F}\|_{L^p(l^q)} \lesssim \sup_{j\geqslant 0} \|M_j (d_j \, \cdot)\|_{H^\kappa} \|f_k\|_{L^p(l^q)}
$$
for ${\rm supp} \widehat{f}_k \subset \Omega_k, \ M_k \in H^\kappa, \ k\in  \mathbb{Z}_+,$ $\|f_k\|_{L^p(l^q)} = \left\|\|\{f_k \}^\infty_{k=0}\|_{l^q} \right\|_{L^p}$.
\end{lemma}

\begin{lemma} \label{Scaling3}
		Let $s< 0$, $\sigma \in \mathbb{R}$, $1\leqslant p,q \leqslant \infty$.  We have the following results.

\begin{itemize}
\item[\rm (i)]
For any $\epsilon >0$, there exists $\lambda_0 =\lambda_0(\epsilon, s,\sigma, q) >1$  such that for any $\lambda > \lambda_0$ and $f\in F^{\sigma,s}_{p,q}$,
\begin{align*}
& 	\|f(\lambda\, \cdot) \|_{F^{\sigma,s}_{p,q}} \lesssim   \lambda^{-d/p + \epsilon}   \|f\|_{F^{\sigma,s}_{p,q} }
\end{align*}
and the above inequality holds for $\epsilon=0$ if $\sigma <0$.

\item[\rm (ii)]
 Assume that ${\rm supp }\, \widehat{f} \subset \{\xi: |\xi|\geqslant \varepsilon_0\}$ for some $\varepsilon_0>0$. Then  for any $\lambda \gg  1$, we have
\begin{align*}
& 	\|f(\lambda\, \cdot) \|_{F^{\sigma,s}_{p,q}} \lesssim \lambda^{\sigma-d/p } 2^{s  \lambda  \varepsilon_0/3} \|f\|_{F^{\sigma,s}_{p,q} }, \ \ \sigma > 0; \\
& 			\|f(\lambda\, \cdot) \|_{F^{ \sigma,s}_{p,q} } \lesssim \lambda^{-d/p} 2^{s  \lambda \varepsilon_0/3} \|f\|_{F^{ \sigma,s}_{p,q} }, \ \ \sigma < 0 ; \\
& 	\|f(\lambda\, \cdot) \|_{F^{\sigma,s}_{p,q}} \lesssim \lambda^{-d/p }\log_2 \lambda  2^{s  \lambda  \varepsilon_0/3} \|f \|_{F^{\sigma,s}_{p,q} }, \ \ \sigma = 0.
\end{align*}

\item[\rm (iii)]
We have for any $f\in F^{\sigma, s}_{p,q} $,
\begin{align*}
& \|f(\lambda^{-1}\, \cdot)\|_{F^{\sigma, \lambda s}_{p,q}  } \lesssim \lambda^{ d/p - \sigma}  \|f \|_{F^{\sigma,  s}_{p,q}  }, \ \ \sigma < 0; \\
& \|f(\lambda^{-1}\, \cdot)\|_{F^{\sigma, \lambda s}_{p,q}  } \lesssim \lambda^{ d/p  }  \|f\|_{F^{\sigma,  s}_{p,q}  }, \ \ \sigma >  0; \\
& \|f(\lambda^{-1}\, \cdot)\|_{F^{\sigma, \lambda s}_{p,q}  } \lesssim \lambda^{ d/p  } \log_2 \lambda \, \|f\|_{F^{\sigma,  s}_{p,q}  }, \ \ \sigma =  0.
\end{align*}
\end{itemize}
\end{lemma}
\noindent {\bf Proof.}
{\it Step 1.} We prove (i). By the norm definition on $F^{\sigma,s}_{p,q}$,
\begin{align} \label{Trspace1}
  \|f(\lambda\,\cdot)\|_{F^{ \sigma,s}_{p,q}   } = \left\| \left(\sum^\infty_{j=0} 2^{\sigma jq}    |\triangle_j 2^{s|\nabla|} f(\lambda\,\cdot) |^q \right)^{1/q}\right\|_{L^p }.
\end{align}
Let $1\ll \lambda= 2^{j_\lambda} c$ with $1\leqslant c <2$. We have
\begin{align*}
& \triangle_0 2^{s|\nabla|} f(\lambda\,\cdot) = \left(\mathscr{F}^{-1} \varphi( \lambda  \xi) 2^{s\lambda |\xi|} \widehat{f} \right)(\lambda\, \cdot),\\
& \triangle_j 2^{s|\nabla|} f(\lambda\,\cdot) = \left(\mathscr{F}^{-1} \varphi(2^{-(j-j_\lambda)}c \xi) 2^{s\lambda |\xi|} \widehat{f} \right)(\lambda\, \cdot), \ \ j\geqslant 1.
\end{align*}
It follows that
\begin{align} \label{Trspace2}
  \|f(\lambda\,\cdot)\|_{F^{ \sigma,s}_{p,q}   } \leqslant & \lambda^{-d/p}  \|\mathscr{F}^{-1} \psi(\lambda \xi) 2^{s\lambda |\xi|} \widehat{f}  \|_{L^p} \nonumber\\
  & + \lambda^{\sigma -d/p}  \left\| \left(\sum^\infty_{j=1-j_\lambda} 2^{\sigma q j}    | \mathscr{F}^{-1} \varphi(2^{-j}c \xi) 2^{s\lambda |\xi|} \widehat{f}  |^q \right)^{1/q}\right\|_{L^p }   := I+II.
\end{align}
Similar to Lemma \ref{Scaling1},
\begin{align} \label{Trspace3}
  I
  \lesssim   \lambda^{  -d/p}   \|\triangle_0 2^{s|\nabla|} f\|_{p} \leqslant   \lambda^{  -d/p} \|f\|_{F^{\sigma,s}_{p,q}}.
\end{align}
Let $\varepsilon_0 \gg \lambda^{-1}$ and we decompose $f$ in the same way as in  the proof of Lemma \ref{Scaling1},
 $$
 \widehat{f} = ( 1-\psi(2 \varepsilon_0^{-1}\, \cdot)) \, \widehat{f} + \psi(2 \varepsilon_0^{-1}\, \cdot) \widehat{f} := \widehat{f}_1 + \widehat{f}_2.
 $$
By \eqref{Trspace2} we have
\begin{align} \label{Trspace2++}
  II \leqslant  \lambda^{\sigma -d/p} \sum_{k=1,2}  \left\| \left(\sum^\infty_{j=1-j_\lambda} 2^{\sigma q j}    | \mathscr{F}^{-1} \varphi(2^{-j}c \xi) 2^{s\lambda |\xi|} \widehat{f}_k  |^q \right)^{1/q}\right\|_{L^p } := III+IV.
\end{align}
Let us consider the estimate of $III$. We have
\begin{align} \label{Trspace2III}
  III \leqslant &  \lambda^{\sigma -d/p}    \left\| \left(\sum^\infty_{\log_2 \varepsilon_0 -5 \leqslant j \leqslant 5} 2^{\sigma q j}    | \mathscr{F}^{-1} \varphi(2^{-j}c \xi) 2^{s\lambda |\xi|} \widehat{f}_1  |^q \right)^{1/q}\right\|_{L^p } \nonumber\\
   & + \lambda^{\sigma -d/p}    \left\| \left(\sum^\infty_{j >5} 2^{\sigma q j}    | \mathscr{F}^{-1} \varphi(2^{-j}c \xi) 2^{s\lambda |\xi|} \widehat{f}_1  |^q \right)^{1/q}\right\|_{L^p }:= V+VI.
  \end{align}
Denote
$$
M_j (\xi) =  \varphi(2^{-j}c \xi) 2^{s(\lambda-1) |\xi|} ( 1-\psi(2 \varepsilon_0^{-1}\xi)).
$$
It is easy to see that
\begin{align*}
& \mathscr{F}^{-1} \varphi(2^{-j}c \xi) 2^{s\lambda |\xi|} \widehat{f}_1  = \mathscr{F}^{-1} M_j(\xi)  \mathscr{F} \sum_{-2\leqslant \ell \leqslant 1} \triangle_{j+\ell} 2^{s|\nabla|} f,\\
& M_j (\xi) = \varphi(2^{-j}c \xi) 2^{s(\lambda-1) |\xi|} , \ \ \mbox{for} \ j >  \log_2 \varepsilon_0 + 5, \\
& M_j (\xi) =0 , \ \ \mbox{for} \ j<  \log_2 \varepsilon_0 -5.
\end{align*}
Notice that ${\rm diam} \, ({\rm supp} \varphi_j) = 5\cdot 2^{j-2}$.  By Lemma \ref{Multiplier} we have for $\lambda \gg 1$, $\kappa >d/2+ d/(p\wedge q)$,
\begin{align} \label{Trspace4}
  VI
  \lesssim &   \lambda^{\sigma -d/p} \sup_{j>5} \|\varphi( \xi ) 2^{s(\lambda-1) c^{-1}2^{j} |\xi|} \|_{H^\kappa}  \|f\|_{F^{ \sigma,s}_{p,q}  } \nonumber\\
  \lesssim &   \lambda^{\sigma -d/p}  2^{s\lambda}  \|f\|_{F^{ \sigma,s}_{p,q}  }.
\end{align}
We now estimate $V$.  Noticing that ${\rm supp } \, \widehat{f} \subset \mathbb{R}^d \setminus B(0, \varepsilon_0)$ and using \eqref{lowfretozero}, we have
\begin{align} \label{Trspace4+}
  V
  \lesssim &  \,  \lambda^{\sigma -d/p} \sum_{\log_2 \varepsilon_0 -5 \leqslant j \leqslant 5} 2^{\sigma   j} \|\mathscr{F}^{-1}   2^{s \lambda |\xi|} \varphi (2^{-j} c \xi)   \widehat{f}_1\|_{p} \nonumber\\
  \lesssim &  \,  \lambda^{\sigma -d/p} \sum_{\log_2 \varepsilon_0 -5 \leqslant j \leqslant 5} 2^{\sigma   j}\|\mathscr{F}^{-1} M_j\|_1  \sum_{-2\leqslant \ell \leqslant 1} \|\triangle_{j+\ell} 2^{s|\nabla|}  \widehat{f} \|_{p} \nonumber\\
  \lesssim &  \,  \lambda^{\sigma -d/p}\|f\|_{F^{ \sigma,s}_{p,q}} \sum_{\log_2 \varepsilon_0 -5 \leqslant j \leqslant 5} 2^{\sigma   j} \|\mathscr{F}^{-1} M_j\|_1.
 \end{align}
In the proof of Lemma \ref{Scaling1}, we have shown that
$$
\|\mathscr{F}^{-1} M_j\|_1 \lesssim  2^{s \varepsilon_0 \lambda/3}
$$
Hence, we have
\begin{align} \label{Trspace4++}
  V
   \lesssim & \,   \lambda^{  -d/p} F(\lambda, \varepsilon_0)  2^{s\lambda \varepsilon_0 /3}  \|f\|_{F^{ \sigma,s}_{p,q}  },
 \end{align}
$$
F(\lambda, \varepsilon_0) =
\left\{
\begin{array}{ll}
\lambda^{\sigma-d/p}, & \sigma >0,\\
\lambda^{ -d/p} \log_2 \varepsilon^{-1}_0£¬  & \sigma =0,\\
\lambda^{\sigma-d/p} \varepsilon^{\sigma}_0, & \sigma <0.
\end{array}
\right.
$$
Using the embedding $B ^{ \sigma,s}_{p,1} \subset  F^{ \sigma,s}_{p,q}$, we see that
\begin{align} \label{Trspace4+3}
  IV
  \lesssim &  \,  \lambda^{\sigma -d/p} \sum_{1-j_\lambda \leqslant j \leqslant \log_2 \varepsilon_0 + 5  } 2^{\sigma   j} \|\mathscr{F}^{-1}   2^{s \lambda |\xi|} \varphi (2^{-j} c \xi)   \widehat{f}_2\|_{p}.
 \end{align}
Following the proof of Lemma \ref{Scaling1}, one has that
\begin{align} \label{Trspace4++4}
  IV
   \lesssim & \,   G (\lambda, \varepsilon_0)     \|f\|_{F^{ \sigma,s}_{p,q}  },
 \end{align}
 where
$$
G (\lambda, \varepsilon_0) =
\left\{
\begin{array}{ll}
\lambda^{\sigma-d/p} \varepsilon^{\sigma}_0, & \sigma >0,\\
\lambda^{ -d/p} \log_2 \lambda,  & \sigma =0,\\
\lambda^{ -d/p}, & \sigma <0.
\end{array}
\right.
$$
Taking $\varepsilon_0 = \lambda^{-1 + \epsilon/\sigma}$, we have the result of (i), as desired.

{\it Step 2.} We prove (ii). Noticing that $f= f_2$ in the proof of (i), we see that (ii) is a straightforward consequence of (i).

{\it Step 3.} We show the results of (iii).
Next, we consider the scaling of $f(\lambda^{-1} \, \cdot)$ with $\lambda \gg 1$.  By definition on $F^{\sigma,s}_{p,q}$,
\begin{align} \label{+Trspace1}
  \|f(\lambda^{-1}\,\cdot)\|_{F^{ \sigma, \lambda s}_{p,q}   } = \left\| \left(\sum^\infty_{j=0} 2^{\sigma jq}    |\triangle_j 2^{s\lambda |\nabla|} f(\lambda^{-1}\,\cdot) |^q \right)^{1/q}\right\|_{L^p }.
\end{align}
Let $\lambda= 2^{j_\lambda} c$ with $1\leqslant c <2$. We have for $j\geqslant 1$,
\begin{align*}
& \triangle_j 2^{s \lambda |\nabla|} f(\lambda^{-1} \,\cdot) = \left(\mathscr{F}^{-1} \varphi(2^{-(j+j_\lambda)}c \xi) 2^{s  |\xi|} \widehat{f} \right)(\lambda^{-1}\, \cdot)  \\
& \triangle_0 2^{s \lambda |\nabla|} f(\lambda^{-1} \,\cdot) = \left(\mathscr{F}^{-1} \psi ( \lambda^{-1}  \xi) 2^{s  |\xi|} \widehat{f} \right)(\lambda^{-1}\, \cdot).
\end{align*}
It follows that
\begin{align} \label{+Trspace2}
  \|f(\lambda^{-1}\,\cdot)\|_{F^{ \sigma, \lambda s}_{p,q}}
  \leqslant & \lambda^{ d/p}  \left\| \left(\sum^\infty_{j=1} 2^{\sigma jq}    | \mathscr{F}^{-1} \varphi(2^{-(j+j_\lambda)}c \xi) 2^{s |\xi|} \widehat{f}  |^q \right)^{1/q}\right\|_{L^p } \nonumber\\
  & +  \lambda^{ d/p}  \|\mathscr{F}^{-1} \psi(\lambda^{-1} \xi) 2^{s |\xi|} \widehat{f}  \|_{L^p}
   := A+B.
\end{align}
By Lemma \ref{Multiplier}, using an analogous way as in the above, we have
\begin{align} \label{+Trspace3}
A  \leqslant & \lambda^{ d/p -\sigma }  \left\| \left(\sum^\infty_{j> j_\lambda} 2^{\sigma jq}    | \mathscr{F}^{-1} \varphi(2^{-j}c \xi) 2^{s |\xi|} \widehat{f}  |^q \right)^{1/q}\right\|_{L^p } \nonumber\\
   \leqslant & \lambda^{ d/p -\sigma }  \|f\|_{F^{\sigma,s}_{p,q}}.
\end{align}
For the estimate of $B$, we have
\begin{align} \label{+Trspace2}
 B \leqslant &   \lambda^{ d/p} \sum^{j_\lambda +2}_{j=0}  \| \triangle_j \mathscr{F}^{-1} \psi(\lambda^{-1} \xi) 2^{s |\xi|} \widehat{f}  \|_{L^p} \nonumber\\
 \leqslant &   \lambda^{ d/p} \sum^{j_\lambda +2}_{j=0}  \| \triangle_j   2^{s |\nabla|} f \|_{L^p}   \nonumber\\
  \leqslant &   \lambda^{ d/p} G(\lambda)  \|f\|_{F^{\sigma,s}_{p,q}},
\end{align}
where
$$
G(\lambda) =
\left\{
\begin{array}{ll}
\lambda^{-\sigma}, & \sigma <0\\
\log_2 \lambda£¬  & \sigma =0\\
1, & \sigma >0.
\end{array}
\right.
$$
The result follows. $\hfill\Box$

\begin{lemma} \label{Isomorphism2}
Let $s, s', \sigma  \in \mathbb{R}$, $1\leqslant p ,q  \leqslant \infty$.   Then
\begin{align} \label{isomorph3}
2^{(s'-s)|\nabla|}: \,   F^{ \sigma, s'}_{p,q} \to  F^{ \sigma,s}_{p,q}
\end{align}
is an isometric mapping. In particular,
 \begin{align} \label{isomorph4}
2^{s|\nabla|}: \,   F^{ \sigma, s}_{p,q} \to  F^{ \sigma}_{p,q}
\end{align}
is an isometric mapping.
\end{lemma}
\noindent {\bf Proof.} It is a direct consequence of
$$
\|2^{(s'-s)|\nabla|} f\|_{ F^{ \sigma,s}_{p,q}} = \|f\|_{F^{ \sigma, s'}_{p,q}}. \eqno{\Box}
$$
Recall that $ F^{ \sigma }_{p,2} = H^\sigma_p $ for all $1<p<\infty, \, \sigma \in \mathbb{R}$ (cf. \cite{Gr2004}). Using the fact
$$
2^{s|\nabla|} F^\sigma_{p,2} = F^{\sigma, s}_{p,2}, \ \  2^{s|\nabla|} H^\sigma_{p} = H^{\sigma, s}_{p},
$$
we immediately have

\begin{lemma} \label{Isomorphism3}
Let $ \sigma, s  \in \mathbb{R}$, $1< p < \infty$.   Then
$  F^{ \sigma, s}_{p,2} =  H^{ \sigma, s }_{p}$
with equivalent norms.
\end{lemma}
Using the scaling properties on $F^{ \sigma, s}_{p,2}$, we have the following
\begin{cor} \label{Scaling4}
		Let $s< 0$, $\sigma \in \mathbb{R}$, $1< p <\infty$. We have the following results.

\begin{itemize}
\item[\rm (i)]
For any $\epsilon >0$, there exists $\lambda_0 =\lambda_0(\epsilon, s,\sigma) >1$  such that for any $\lambda > \lambda_0$ and $f\in H^{\sigma,s}_{p}$,
\begin{align*}
& 	\|f(\lambda\, \cdot) \|_{H^{\sigma,s}_{p}} \lesssim   \lambda^{-d/p + \epsilon}   \|f\|_{H^{\sigma,s}_{p} }
\end{align*}
and the above inequality holds for $\epsilon=0$ if $\sigma <0$. In particular,
\begin{align*}
& 	\|f(\lambda\, \cdot) \|_{E^{\sigma,s}} \lesssim   \lambda^{-d/2 + \epsilon}   \|f\|_{E^{\sigma,s} }.
\end{align*}

\item[\rm (ii)]
  Assume that ${\rm supp }\, f \subset \{\xi: |\xi|\geqslant \varepsilon_0\}$ for some $\varepsilon_0>0$. Then  for any $\lambda \gg  1$,
\begin{align*}
& 	\|f(\lambda\, \cdot) \|_{H^{\sigma,s}_{p}} \lesssim \lambda^{\sigma-d/p } 2^{s  \lambda  \varepsilon_0/3} \|f\|_{H^{\sigma,s}_{p} }, \ \ \sigma > 0; \\
& 	\|f(\lambda\, \cdot) \|_{H^{ \sigma,s}_{p} } \lesssim \lambda^{-d/p} 2^{s  \lambda \varepsilon_0/3} \|f\|_{H^{ \sigma,s}_{p} }, \ \ \sigma \leqslant 0 .
\end{align*}

\item[\rm (iii)]
We have for any $f\in H^{\sigma, s}_{p} $, $\lambda \gg 1$,
\begin{align*}
& \|f(\lambda^{-1}\, \cdot)\|_{H^{\sigma, \lambda s}_{p}  } \lesssim \lambda^{ d/p - \sigma}  \|f \|_{H^{\sigma,  s}_{p}  }, \ \ \sigma < 0; \\
& \|f(\lambda^{-1}\, \cdot)\|_{H^{\sigma, \lambda s}_{p}  } \lesssim \lambda^{ d/p  }  \|f\|_{H^{\sigma,  s}_{p}  }, \ \ \sigma \geqslant  0.
\end{align*}
\end{itemize}
\end{cor}

\begin{proof}
It suffices to consider the case $\sigma =0$, which is easily seen by an observation $\|f\|_{H^{0,s}_{p}  } = \|\mathscr{F}^{-1}2^{s|\xi|} \widehat{f}\|_p$.
\end{proof}

\begin{cor} \label{Scaling2a}
Let $s< 0$, $\sigma \in \mathbb{R}$. Denote $(D_\lambda f)(x,t) = f(x/\lambda, t/\lambda)$.  Then  for any $\lambda> 1$, we have
\begin{align*}
& \|D_\lambda f\|_{ {L} ^\gamma (\mathbb{R}; H^{\sigma, \lambda s}_{p})  } \lesssim \lambda^{1/\gamma + d/p - \sigma}  \|f \|_{ {L} ^\gamma (\mathbb{R}; H^{\sigma,  s}_{p} )  }, \ \ \sigma < 0; \\
& \|D_\lambda f\|_{ {L} ^\gamma (\mathbb{R}; H^{\sigma, \lambda s}_{p})  }  \lesssim \lambda^{1/\gamma +  d/p  }  \|f\|_{ {L} ^\gamma (\mathbb{R}; H^{\sigma,  s}_{p} )  }, \ \ \sigma \geqslant 0;
\end{align*}
\end{cor}

\section{Decaying estimates for the variable parameter $\lambda>1$} \label{sectDecay}

Let us consider the decaying estimates for the solutions of the scaling linear Klein-Gordon equations
\begin{align}\label{LKG1}
\left\{
\begin{array}{l}
\partial^2_t u  + \lambda^2 u-  \Delta u  = f,  \\
  u(x,0) = u_0(x), \ \ u_t (x,0) = u_1(x).
\end{array}
\right.
\end{align}
As indicated in the introduction, one needs to take $\lambda >1$ sufficiently large to guarantee that the initial data are small enough in $E^{\sigma,s} \times E^s_{\sigma-1}$. So, we need to treat $\lambda >1$ as a variable parameter in the decaying estimates of the solutions of \eqref{LKG1}.
For convenience, we write
 \begin{align}\label{LKG2}
K_\lambda (t):= \frac{\sin t (\lambda^2-\Delta)^{1/2}}{(\lambda^2-\Delta)^{1/2}}, \ \  K'_\lambda (t):=  \cos t (\lambda^2-\Delta)^{1/2}.
\end{align}
It is known that the  solution of \eqref{LKG1} can be rewritten as
\begin{align}
   \label{LKG3}
u(t)= K'_\lambda (t) u_0 + K_\lambda (t) u_1 + \int^t_0 K_\lambda (t-\tau ) f(\tau) d\tau.
\end{align}
Denote
 \begin{align}\label{LKG4}
U_\lambda (t):=  e^{{\rm i} t (\lambda^2-\Delta)^{1/2} }.
\end{align}
It is easy to see that
$$
K'_\lambda (t) = \frac{1}{2}(U_\lambda (t) + U_\lambda (-t)), \ \  K_\lambda (t) = \frac{ U_\lambda (t) - U_\lambda (-t) }{2{\rm i}(\lambda^2-\Delta)^{1/2}}.
$$
To estimate $K'_\lambda (t)$ and $K_\lambda (t)$, it suffices to consider the decaying estimate of $U_\lambda (t)$.   In order to have an illustration for the decay of $U_\lambda (t)$ in higher spatial dimensions, we first consider  1D case, which is much easier than that of higher spatial dimensions.

\subsection{Decay in 1D}

We consider the decaying estimates of $U_\lambda (t)$ for $\lambda>1$ in 1D. First, let us recall the Van der Corput lemma \cite{Stein1993}:
\begin{lemma} \label{VanderCorput}
Let $\phi \in C^\infty_0 (\mathbb{R})$ and there exists $k\geqslant 2$ such that $P\in C^{k} (\mathbb{R})$ satisfies $|P^{(k)}(\xi)| \geqslant \rho $ for all $\xi \in {\rm supp}\, \phi$. Then we have
 \begin{align}\label{LKG5}
\left|\int e^{{\rm i}  P(\xi)}  \phi(\xi) d\xi\right| \leqslant \rho^{-1/k} (\|\phi\|_\infty + \|\phi'\|_1)
\end{align}
\end{lemma}
It is easy to see that for $j\in \mathbb{N}$,
\begin{align}\label{LKG6}
\triangle_0 U_\lambda (t) f = \mathscr{F}^{-1} (e^{{\rm i}t \sqrt{\lambda^2+|\xi|^2}} \psi(\xi) ) * f, \ \   \triangle_j U_\lambda (t) f = \mathscr{F}^{-1} (e^{{\rm i}t \sqrt{\lambda^2+|\xi|^2}} \varphi_j (\xi) ) * f.
\end{align}
By Young's inequality,
\begin{align}\label{LKG7}
\|\triangle_0 U_\lambda (t) f \|_\infty  \leqslant \| \mathscr{F}^{-1} (e^{{\rm i}t \sqrt{\lambda^2+|\xi|^2}} \psi(\xi) )\|_\infty  \|f\|_1.
\end{align}
Let us  write $p_\lambda (\xi) = x\xi + t \sqrt{\lambda^2+ \xi^2}$. One easily sees that
$$
p_\lambda' (\xi) = x+ t\xi  (\lambda^2+ \xi^2)^{-1/2}, \ \  p_\lambda'' (\xi) =   t\lambda^2   (\lambda^2+ \xi^2)^{-3/2}.
$$
Applying Van der Corput's lemma, one has that
\begin{align}\label{LKG8}
 \| \mathscr{F}^{-1} (e^{{\rm i}t \sqrt{\lambda^2+|\xi|^2}} \psi(\xi) )\|_\infty  \lesssim \lambda^{1/2} |t|^{-1/2}.
\end{align}
It follows from \eqref{LKG7} and \eqref{LKG8} that
\begin{align}\label{LKG9}
\|\triangle_0 U_\lambda (t) f \|_\infty  \lesssim \lambda^{1/2} |t|^{-1/2}  \|f\|_1.
\end{align}
Next, by scaling and Young's inequality, we have
\begin{align}\label{LKG10}
\|\triangle_j U_\lambda (t) f \|_\infty  \leqslant 2^{j}\| \mathscr{F}^{-1} (e^{{\rm i}t \sqrt{\lambda^2+|2^j\xi|^2}} \varphi(\xi) )\|_\infty  \|f\|_1.
\end{align}
Denote $p_{j,\lambda} (\xi) = x\xi + t \sqrt{\lambda^2+ |2^j\xi|^2}$. We see that
$$
p'_{j,\lambda} (\xi) = x+ 2^{2j} t\xi  (\lambda^2+ 2^{2j}\xi^2)^{-1/2}, \ \  p_{j,\lambda} '' (\xi) =   t 2^{2j}\lambda^2   (\lambda^2+ |2^j\xi|^2)^{-3/2}.
$$
It follows that
$$
| p_{j,\lambda} '' (\xi)|   \gtrsim   |t| 2^{2j}\lambda^2   (\lambda^3+  2^{3j})^{-1} .
$$
By Van der Corput lemma,
\begin{align}\label{LKG11}
 \| \mathscr{F}^{-1} (e^{{\rm i}t \sqrt{\lambda^2+|2^j \xi|^2}} \varphi(\xi) )\|_\infty  \lesssim (\lambda^{1/2}2^{-j} + \lambda^{-1}2^{j/2}) |t| ^{-1/2}.
\end{align}
So, from \eqref{LKG9}, \eqref{LKG10} and \eqref{LKG11} it follows that for $j\geqslant 0$,
\begin{align}\label{LKG12}
\|\triangle_j U_\lambda (t) f \|_\infty \lesssim  (\lambda^{1/2}  + \lambda^{-1}2^{3j/2}) |t| ^{-1/2} \|f\|_1.
\end{align}
Since $\psi$ and $\varphi$ have compact support sets, we have for $j\geqslant 0$,
\begin{align}\label{LKG13}
\|\triangle_j U_\lambda (t) f \|_\infty \lesssim  2^j \|f\|_1.
\end{align}
By Plancherel's identity, for $j\geqslant 0$,
\begin{align}\label{LKG14}
\|\triangle_j U_\lambda (t) f \|_2 \leqslant  \|f\|_2.
\end{align}
By interpolation,  for $\delta(p) = 1/2-1/p$ we have for $2\leqslant p \leqslant \infty$,
\begin{align}\label{LKG15}
& \|\triangle_j U_\lambda (t) f \|_p \lesssim  (\lambda^{\delta(p)}  + \lambda^{-2 \delta(p)}2^{3\delta(p)j}) |t| ^{- \delta(p)} \|f\|_{p'}, \\
\label{LKG16}
& \|\triangle_j U_\lambda (t) f \|_p \lesssim   2^{2\delta(p)j}  \|f\|_{p'}.
\end{align}

\begin{lemma} \label{1Ddecay}
Let $2\leqslant p \leqslant \infty$, $\delta (p) =1/2-1/p$. For any $\theta \in [0,1]$, we have
\begin{align}\label{LKG17}
& \|\triangle_j U_\lambda (t) f \|_p \lesssim  (\lambda^{\theta \delta(p)} 2^{2(1-\theta)\delta(p) j}  + \lambda^{-2 \theta \delta(p)}2^{(2+\theta)\delta(p)j}) |t| ^{- \theta \delta(p)} \|f\|_{p'}.
\end{align}

\end{lemma}

\subsection{Decay in Higher D}

If we take $\lambda=1$,  the dispersion estimates for the Klein-Gordon equation is well-known, cf. \cite{Br1984,Br1985,Br1989,GiVe1985,GiVe1995,KeTa1998}. Since we treat $\lambda>1$ as a variable parameter, the estimates of the decaying rates of $U_\lambda(t)$ depends on $\lambda>0$. We will apply the techniques as in \cite{BeKoSa2000,GuPeWa2008} to calculate the decay of  $U_\lambda(t)$. By scaling and Young's inequality, we have
\begin{align}\label{HLKG1}
\|\dot{\triangle}_j U_\lambda (t) f \|_\infty  \leqslant 2^{j d }\| \mathscr{F}^{-1} (e^{{\rm i}t \sqrt{\lambda^2+|2^j\xi|^2}} \varphi(\xi) )\|_\infty  \|f\|_1.
\end{align}
By Hausdorff-Young's inequality, we have
\begin{align}\label{HLKG0}
\|\dot{\triangle}_j U_\lambda (t) f \|_\infty  \leqslant 2^{j d }  \|f\|_1.
\end{align}
For a radial function $f$, it is known that
\begin{align}\label{HLKG2}
  (\mathscr{F}^{-1}  f) (x) = c \pi \int^\infty_0 f(r) r^{d-1} (r|x|)^{-(d-2)/2} J_{(d-2)/2} (r|x|) dr,
\end{align}
where $J_m (x)$ is the Bessel function defined by
\begin{align}\label{HLKG3}
    J_{m} (r ) =\frac{(r/2)^m}{\Gamma(m+1/2)\pi^{1/2}} \int^1_{-1} e^{{\rm i} rt} (1-t^2)^{m-1/2} dt, \ \ m>-1/2.
\end{align}
Recall that the Bessel function satisfying (cf. \cite{GuPeWa2008,Jo1981,WaHuHaGu2011})

\begin{prop} \label{Bessel}
Let $m>-1/2$ and $0<r<\infty$. We have the following results.
\begin{itemize}

\item[\rm (i)] $r^{-m} J_m (r) \leqslant C$.

\item[\rm (ii)] $\frac{d}{dr} \left(r^{-m} J_m (r) \right) =- r^{-m} J_{m+1} (r)$.

\item[\rm (iii)]  $r^{-(d-2)/2} J_{(d-2)/2} (r) =c_d \mathfrak{Re} \left( e^{{\rm i}r} h(r) \right)$, where $h(r)$ satisfies
\begin{align} \label{HLKG4}
\left|\frac{d^k}{dr^k } h(r) \right| \leqslant C_k (1+r)^{-(d-1)/2-k}, \ \ k\in \mathbb{Z}_+.
\end{align}
\end{itemize}
\end{prop}
Using \eqref{HLKG2}, we have
$$
 \mathscr{F}^{-1}(e^{{\rm i}t \sqrt{\lambda^2+|2^j\xi|^2}} \varphi(\xi) ) (x) = 2\pi \int^\infty_0  e^{{\rm i}t \sqrt{\lambda^2+|2^jr|^2}} \varphi(r) r^{d-1} K(r|x|)dr,
$$
where $K(\varrho): = \varrho^{-(d-2)/2} J_{(d-2)/2} (\varrho)$. Now we divide the proof into the following two cases

{\it Case} 1. $0\leqslant |x| \leqslant 1$.
 Noticing that ${\rm supp} \, \varphi \subset \{\xi: |\xi| \in [1/2,2]\}$, we see that for $0\leqslant |x| \leqslant 1$,
\begin{align}\label{HLKG5}
|\partial^k_r (r^{d-1}\varphi(r) K(r|x|))| \lesssim 1.
\end{align}
Denote $p_{j,\lambda} (r) = \sqrt{ \lambda^2+ 2^{2j} r^2}$. Integrating by part, we have
\begin{align}\label{HLKG6}
\int^\infty_0  &  e^{{\rm i}t p_{j,\lambda} (r) } \varphi(r) r^{d-1} K(r|x|)dr  \nonumber\\
& =- \int^2_{1/2}    e^{{\rm i}t p_{j,\lambda} (r) } \frac{\partial}{\partial r} \left( \frac{1}{{\rm i}t p'_{j,\lambda} (r)} \varphi(r) r^{d-1} K(r|x|) \right) dr \nonumber\\
& = \int^2_{1/2}    e^{{\rm i}t p_{j,\lambda} (r) } \frac{\partial}{\partial r} \left( \frac{1}{{\rm i}t p'_{j,\lambda} (r)} \frac{\partial}{\partial r} \left( \frac{1}{{\rm i}t p'_{j,\lambda} (r)} \varphi(r) r^{d-1} K(r|x|) \right) \right)  dr \nonumber\\
& = ... \ .
\end{align}
Repeating the argument above and using Lemma  \ref{Deriv-control2}, we have for any $q\in \mathbb{N}$,
\begin{align}\label{HLKG6a}
&\left|\int^\infty_0  e^{{\rm i}t p_{j,\lambda} (r) } \varphi(r) r^{d-1} K(r|x|)dr \right| \nonumber\\
& \lesssim |t|^{-q} \sup_{r\in [1/2,2]} \sum_{\alpha_1+...+\alpha_q \leqslant q} \left| \left(\frac{1}{p'_{j,\lambda} (r)}\right)^{(\alpha_1)} ... \left(\frac{1}{p'_{j,\lambda} (r)}\right)^{(\alpha_q)}    \right|.
\end{align}
Noticing that $p_{j,\lambda} (r) = 2^j ((\lambda/2^j)^2 + r^2)^{1/2}$,   $ p'_{j,\lambda} (r) \sim 2^{2j} (\lambda^2 + 2^{2j}) ^{-1/2}$ for $r\in [1/2,2]$, by Lemmas  \ref{Deriv-control} and \ref{Deriv-control3}, we obtain that
\begin{align}\label{HLKG7}
\left|\int^\infty_0    e^{{\rm i}t p_{j,\lambda} (r) } \varphi(r) r^{d-1} K(r|x|)dr \right|  \lesssim |t|^{-q}(\lambda 2^{-2j} + 2^{-j})^q, \ \ \forall \ q\geqslant 0.
 \end{align}
So, for $j\in \mathbb{Z}$ and $|x| \leqslant 1$,
\begin{align}\label{HLKG8}
 |\dot{\triangle}_j U_\lambda (t) f (x)|    \leqslant 2^{j d }  |t|^{-q}(\lambda 2^{-2j} + 2^{-j})^q \|f\|_1.
\end{align}
In particular, for $j\in \mathbb{Z}$ and $|x| \leqslant 1$,
\begin{align}\label{HLKG9}
 |\dot{\triangle}_j U_\lambda (t) f (x)|    \leqslant   |t|^{-d/2}(\lambda^{d/2}   + 2^{dj/2})  \|f\|_1.
\end{align}

{\it Case 2}. $|x| \geqslant 1$. Using (iii) of Proposition \ref{Bessel}, one can rewrite
\begin{align} \label{HLKG10}
 \mathscr{F}^{-1}(e^{{\rm i}t \sqrt{\lambda^2+|2^j\xi|^2}} \varphi(\xi) ) = &  c_d  \pi \int^\infty_0  e^{{\rm i}t \sqrt{\lambda^2+|2^jr|^2} + i r|x|} \varphi(r) r^{d-1} h (r|x|)dr \nonumber\\
 & +   c_d \pi \int^\infty_0  e^{{\rm i}t \sqrt{\lambda^2+|2^jr|^2} - i r|x|} \varphi(r) r^{d-1} h  (r|x|)dr \nonumber\\
 := A_{j,\lambda} (x) +  B_{j,\lambda} (x).
\end{align}
It suffices to consider the estimate of $A_{j,\lambda} (x)$. Denote $ h_{j,\lambda} (r) = t \sqrt{\lambda^2+|2^jr|^2} +  r|x|$. Let us observe that
\begin{align} \label{HLKG10a}
 A_{j,\lambda} (x) =    c_d  \pi \int^\infty_0  e^{ {\rm i}h_{j,\lambda}(r) } \varphi(r) r^{d-1} h (r|x|)dr.
  \end{align}
We emphasize that $h_{j,\lambda} $ play  similar roles as  $p_{j,\lambda} $ in \eqref{HLKG6}.    It is easy to see that
$$
 h'_{j,\lambda} (r) = |x| + \frac{t 2^{2j} r}{\sqrt{\lambda^2+|2^jr|^2} }.
$$
If $r \in [1/2, 2]$, we see that
$$
  \frac{|t| 2^{2j} r}{\sqrt{\lambda^2+|2^jr|^2} } \sim  \frac{|t| 2^{2j} }{ \lambda +  2^j  }.
$$
We can assume that $t<0$, since the opposite case is easier to handle. We consider the following three subcases.

{\it Case} 2a.  $|x| \geqslant 1 \vee C |t| 2^{2j} / (\lambda +  2^j)$.
It follows that
$$
| h'_{j,\lambda} (r)|   \gtrsim   |t| (\lambda 2^{-2j} + 2^{-j})^{-1}.
$$
Using the same way as in \eqref{HLKG6} and \eqref{HLKG6a}, integrating by part to $A_{j,\lambda}$, we get that
$$
|A_{j,\lambda} (x)| \lesssim |t|^{-q}   (\lambda 2^{-2j} + 2^{-j})^{q}.
$$

{\it Case} 2b.  $|x| \geqslant 1, \ |x| \leqslant c |t| 2^{2j} / (\lambda +  2^j)$.  We still have
$$
| h'_{j,\lambda} (r)|   \gtrsim   |t| (\lambda 2^{-2j} + 2^{-j})^{-1}.
$$
Using an analogous way as in Case 2a, we have
$$
|A_{j,\lambda} (x)| \lesssim |t|^{-q}   (\lambda 2^{-2j} + 2^{-j})^{q}.
$$

{\it Case} 2c.  $|x| \geqslant 1, \ |x| \sim  |t| 2^{2j} / (\lambda +  2^j)$.
Noticing that for $r\in [1/2,2]$
$$
| h''_{j,\lambda} (r)| = \frac{|t|2^j (\lambda/2^j)^2 }{((\lambda/2^j)^2 +r^2)^{3/2} } \sim  \frac{|t| \lambda^2 2^{2j} }{ \lambda^3 + 2^{3j} }
$$
and for $k=0,1,$
\begin{align} \label{HLKG11}
|\partial^k_r h(r|x|)| \lesssim |x|^{-(d-1)/2}  \lesssim \left( \frac{|t| 2^{2j}}{\lambda+2^j}\right)^{-(d-1)/2},
\end{align}
by Van der Corput Lemma we have
\begin{align}
|A_{j,\lambda} (x)| & \lesssim  \left(\frac{|t| \lambda^2 2^{2j} }{ \lambda^3 + 2^{3j} } \right)^{-1/2} \left( \frac{t2^{2j}}{\lambda+2^j}\right)^{-(d-1)/2} \nonumber\\
 & \lesssim   |t|^{-d/2} \lambda^{-1} 2^{-dj}   (\lambda  + 2^{ j})^{(d+2)/2} \nonumber\\
 & \lesssim   |t|^{-d/2}   2^{-dj}   (\lambda^{d/2}  + \lambda^{-1} 2^{(d+2) j/2 }).  \label{HLKG12}
\end{align}
On the other hand, using \eqref{HLKG11}, one can get that for $|x| \sim |t| 2^{2j}/(\lambda+2^j)$,
\begin{align} \label{HLKG13}
|A_{j,\lambda} (x)|  \lesssim |x|^{-(d-1)/2}   \lesssim
 |t|^{-(d-1)/2} 2^{- j(d-1)}( \lambda +   2^{j} )^{(d-1)/2}.
  \end{align}
  In summary, we have obtain that
\begin{align} \label{HLKG14}
\|\dot{\triangle}_j U_\lambda (t) f\|_\infty \lesssim \min\left(
 |t|^{-d/2} ( \lambda^{d/2}  + \lambda^{-1} 2^{ j(d+2)/2} ),   |t|^{-(d-1)/2} ( \lambda^{(d-1)/2} 2^{j}   +   2^{j(d+1)/2}  ) \right)\|f\|_1.
  \end{align}
So, for any $\theta\in [0,1]$,
\begin{align} \label{HLKG15}
\|\dot{\triangle}_j U_\lambda (t) f\|_\infty \lesssim
 |t|^{-(d-1+\theta)/2} ( \lambda^{(d-1+\theta) /2} 2^{j(1-\theta)}  + \lambda^{-\theta} 2^{ j(d+1+\theta)/2} )    \|f\|_1.
\end{align}
Interpolating \eqref{HLKG15} with
$$
\|\dot{\triangle}_j U_\lambda (t) f\|_2 \leqslant \|  f\|_2,
$$
we have for $\delta(p) =1/2-1/p$, $2\leqslant p \leqslant \infty$,
\begin{align} \label{HLKG16}
\|\dot{\triangle}_j U_\lambda (t) f\|_p  \lesssim
 |t|^{-(d-1+\theta)\delta(p)} ( \lambda^{(d-1+\theta)\delta(p)} 2^{2j(1-\theta)\delta(p)}  + \lambda^{-2\theta\delta(p)} 2^{ j(d+1+\theta)\delta(p)} )    \|f\|_{p'}.
\end{align}

Combining \eqref{HLKG16} and Lemma \ref{1Ddecay}, we have
\begin{lemma} \label{HDdecay}
Let $d\geqslant 1$, $2\leqslant p \leqslant \infty$, $\delta (p) =1/2-1/p$. For any $\theta \in [0,1]$, we have
\begin{align} \label{HLKG160}
\| {\triangle}_j U_\lambda (t) f\|_p  \lesssim
 |t|^{-(d-1+\theta)\delta(p)} ( \lambda^{(d-1+\theta)\delta(p)} 2^{2j(1-\theta)\delta(p)}  + \lambda^{-2\theta\delta(p)} 2^{ j(d+1+\theta)\delta(p)} )    \|f\|_{p'}
\end{align}
 for all $f\in L^{p'}$, $j\in \mathbb{Z}_+$.
\end{lemma}
\begin{proof}
If $j\geqslant 1$, we have the result from \eqref{HLKG16}. So, it suffices to show \eqref{HLKG160} for the case $j=0$.  We have from $\psi(\xi) = \sum_{j\leqslant 0} \varphi_j (\xi), \ \xi\neq 0$  that
\begin{align*}
\|\triangle_0 U_\lambda(t) f\|_\infty
   \leqslant \left\| \sum_{j\leqslant 0} 2^{jd} \left|\mathscr{F}^{-1} \left(\varphi (\xi) e^{{\rm i}t \sqrt{\lambda^2+| 2^j\xi|^2}} \right) (2^j\,\cdot)\right| \right\|_\infty  \| f\|_1.
\end{align*}
First, we have
\begin{align*}
 \sum_{j\leqslant 0} 2^{jd} \left|\mathscr{F}^{-1} \left(\varphi (\xi) e^{{\rm i}t \sqrt{\lambda^2+| 2^j\xi|^2}} \right) (2^j\,\cdot) \right| \leqslant \sum_{j\leqslant 0} 2^{jd} \lesssim 1.
\end{align*}
Let $y\in \mathbb{R}^d$ and
$$
\mathbb{Z}_y = \{j \in \mathbb{Z}_-: \, |2^j y| \leqslant 1\}.
$$
If $j\in \mathbb{Z}_y $, then we have
$$
2^{jd} \left|\mathscr{F}^{-1} \left(\varphi (\xi) e^{{\rm i}t \sqrt{\lambda^2+| 2^j\xi|^2}} \right) (2^j y) \right| \lesssim |t|^{-q} \lambda^q 2^{j(d-2q)}.
$$
Denote $\mathbb{Z}_{y,1} =\{ j\in \mathbb{Z}_y: \ 2^j \leqslant \lambda^{1/2} |t|^{-1/2}\}$, $\mathbb{Z}_{y,2} = \mathbb{Z}_{y} \setminus \mathbb{Z}_{y,1}$. Choosing $q>d/2$, we have
$$
\sum_{j\in \mathbb{Z}_y} 2^{jd} \left|\mathscr{F}^{-1} \left(\varphi (\xi) e^{{\rm i}t \sqrt{\lambda^2+| 2^j\xi|^2}} \right) (2^j y) \right| \leqslant \sum_{j\in \mathbb{Z}_{y,1}} 2^{jd} + \sum_{j\in \mathbb{Z}_{y,2}}   |t|^{-q} \lambda^q 2^{j(d-2q)} \lesssim \lambda^{d/2}  |t|^{-d/2}.
$$
Next, we consider the summation on $j\in \mathbb{Z}^c_y = \mathbb{Z}_- \setminus \mathbb{Z}_y$.  If there exists $j$ satisfying $|y| \leqslant C |t|2^{j} (\lambda +2^{j})^{-1}$, then we can find a smallest integer $j_0: = j_0 (y,\lambda,t)$ satisfying $|y| \leqslant C |t|2^{j_0} (\lambda +2^{j_0})^{-1}$.  We can decompose the summation overall $j\in \mathbb{Z}^c_y$ by
\begin{align}
\sum_{j\in \mathbb{Z}^c_y} 2^{jd} & \left | \mathscr{F}^{-1} \left(\varphi (\xi) e^{{\rm i}t \sqrt{\lambda^2+| 2^j\xi|^2}} \right) (2^j y) \right |  \nonumber \\
& = \left( \sum_{j\in \mathbb{Z}^c_y, \, |j-j_0| \leqslant C} +  \sum_{j\in \mathbb{Z}^c_y, \, |j-j_0| > C} \right) 2^{jd} \left | \mathscr{F}^{-1} \left(\varphi (\xi) e^{{\rm i}t \sqrt{\lambda^2+| 2^j\xi|^2}} \right)(2^j y) \right |.
\end{align}
By \eqref{HLKG14} we see that
\begin{align}
 \sum_{j\in \mathbb{Z}^c_y, \,|j-j_0| \leqslant C}    2^{jd} \left | \mathscr{F}^{-1} \left(\varphi (\xi) e^{{\rm i}t \sqrt{\lambda^2+| 2^j\xi|^2}} \right)(x) \right |
  \lesssim  \lambda^{d/2}  |t|^{-d/2}.
 \end{align}
If such a $j_0$ does not exist or $j_0 \gg 1$, then  $\{j\in \mathbb{Z}^c_y: \, |j-j_0| \leqslant C\}$ is an empty set.    The summation $\sum_{j\in \mathbb{Z}^c_y, \,|j-j_0| > C}$ can be handled by following the proof to the case $j \in \mathbb{Z}_y$. Indeed, in view of Cases 2a and 2b in the estimates of $A_{j,\lambda}$,  $|j-j_0| > C$ implies that  $A_{j,\lambda}(2^jy)$ has the same decaying rates as in the case $j \in \mathbb{Z}_y$. Summarizing the above estimate, we obtain that
\begin{align*}
\|\triangle_0 U_\lambda(t) f\|_\infty
   \leqslant \min(1, \, \lambda^{d/2}  |t|^{-d/2} ) \| f\|_1,
\end{align*}
which implies the result of \eqref{HLKG160} for the case $j=0$.
 \end{proof}

\subsection{Strichartz estimates with the parameter $\lambda>1$}

By Lemma \ref{HDdecay}, we have for any $2\leqslant p\leqslant \infty$, $\delta(p) =1/2-1/p$,
\begin{align} \label{Str1}
\|\triangle_j U_\lambda (t) f\|_p  \lesssim
\left\{
\begin{array}{ll}
 |t|^{-(d-1+\theta)\delta(p)}   \lambda^{(d-1+\theta)\delta(p)} 2^{2j(1-\theta)\delta(p)}  \|f\|_{p'}, & 2^j \lesssim \lambda, \\
    |t|^{-(d-1+\theta)\delta(p)} \lambda^{-2\theta\delta(p)} 2^{ j(d+1+\theta)\delta(p)} )    \|f\|_{p'}, & 2^j \gtrsim \lambda.
 \end{array}
 \right.
\end{align}
For convenience, we write
\begin{align}
\mathscr{A}_\lambda f (t) :=   \int^t_0 U_\lambda (t-\tau) f(\tau) d\tau  \label{Str-notat}
\end{align}
By Standard $TT^*$ method (see Appendix B), we can show the following

\begin{lemma}\label{Strichartz}
Let $2\leqslant p,r \leqslant \infty$, $\delta(p)=1/2-1/p$, $\theta\in [0,1]$, $2/\gamma(\theta, p) =(d-1+\theta)\delta(p)$ and $2\sigma (\theta,p) = (d+1+\theta)\delta(p)$. Assume that $2/\gamma(\theta, p), 2/\gamma(\theta, r) \in [0,1)$.  We have the following Strichartz estimates
\begin{itemize}

\item[\rm (i)] If $2^j\leqslant \lambda$, then we have
\begin{align}
& \lambda^{-1/\gamma(\theta,p)} 2^{ -j(1-\theta)\delta(p)} \|\triangle_j U_\lambda (t) u_0\|_{L^{\gamma(\theta,p)} (\mathbb{R}, L^p)} \lesssim \|u_0\|_2, \label{Str-a}\\
& \lambda^{-1/\gamma(\theta,p)} 2^{ -j(1-\theta)\delta(p)} \left\|\triangle_j \mathscr{A}_\lambda f \right\|_{L^{\gamma(\theta,p)} (\mathbb{R}, L^p)}
   \lesssim \lambda^{1/\gamma(\theta,r)} 2^{ j(1-\theta)\delta(r)}\|f\|_{L^{\gamma(\theta,r)'} (\mathbb{R}, L^{r'})}.  \label{Str-b}
\end{align}
\item[\rm (ii)] If $2^j> \lambda$, then we have
\begin{align}
\lambda^{ \theta \delta(p)} 2^{-j\sigma(\theta,p)} & \|\triangle_j U_\lambda (t) u_0\|_{L^{\gamma(\theta,p)} (\mathbb{R}, L^p)} \lesssim \|u_0\|_2, \label{Str-c}\\
\lambda^{ \theta \delta(p) } 2^{ -j \sigma(\theta,p)} & \left\|\triangle_j \mathscr{A}_\lambda f \right\|_{L^{\gamma(\theta,p)} (\mathbb{R}, L^p)}
   \lesssim  \lambda^{ -\theta \delta(r) } 2^{ j\sigma(\theta,r)} \|f\|_{L^{\gamma(\theta,r)'} (\mathbb{R}, L^{r'})}. \label{Str-d}
\end{align}

\end{itemize}
\end{lemma}
Let us write for $\lambda >1$, $j_\lambda := \log_2 \lambda$
$$
\mathbb{Z}_\lambda = \{j\in \mathbb{Z}_+: \ 0\leqslant j\leqslant j_\lambda \}, \ \  \ \mathbb{Z}^c_\lambda  = \mathbb{Z}_+ \setminus  \mathbb{Z}_\lambda.
$$
By Lemma \ref{Strichartz}, we have

\begin{prop}\label{StrichartzG}
Let $2\leqslant p,r \leqslant \infty$, $\delta(p)=1/2-1/p$, $\theta\in [0,1]$, $2/\gamma(\theta, p) =(d-1+\theta)\delta(p)$ and $2\sigma (\theta,p) = (d+1+\theta)\delta(p)$. Assume that $2/\gamma(\theta, p), 2/\gamma(\theta, r) \in [0,1)$. Then we have
\begin{align}
& \lambda^{-1/\gamma(\theta,p)}  \| U_\lambda (t) u_0\|_{\widetilde{L}^{\gamma(\theta,p)} (\mathbb{R}, B^{\sigma-(1-\theta)\delta(p),s}_{p,2} (\mathbb{Z}_\lambda)) } \lesssim \|u_0\|_{E^{\sigma,s}}, \label{GStr-a}\\
& \lambda^{-1/\gamma(\theta,p)}  \left\| \mathscr{A}_\lambda f \right\|_{\widetilde{L}^{\gamma(\theta,p)} (\mathbb{R}, B^{\sigma-(1-\theta)\delta(p),s}_{p,2} (\mathbb{Z}_\lambda))}
   \lesssim \lambda^{1/\gamma(\theta,r)} \|f\|_{\widetilde{L}^{\gamma(\theta,r)'} (\mathbb{R}, B^{\sigma+(1-\theta)\delta(r),s}_{r',2} (\mathbb{Z}_\lambda))}.  \label{GStr-b}
\end{align}
\begin{align}
& \lambda^{ \theta \delta(p)}  \| U_\lambda (t) u_0\|_{\widetilde{L}^{\gamma(\theta,p)} (\mathbb{R}, B^{\sigma-\sigma(\theta,p),s}_{p,2} (\mathbb{Z}^c_\lambda)) } \lesssim \|u_0\|_{E^{\sigma,s}}, \label{GStr-c}\\
& \lambda^{ \theta \delta(p) }   \left\| \mathscr{A}_\lambda f \right\|_{\widetilde{L}^{\gamma(\theta,p)} (\mathbb{R}, B^{\sigma-\sigma(\theta,p),s}_{p,2} (\mathbb{Z}^c_\lambda))}
  \lesssim  \lambda^{ -\theta \delta(r) }  \|f\|_{\widetilde{L}^{\gamma(\theta,r)'} (\mathbb{R}, B^{\sigma+\sigma(\theta,r),s}_{r',2} (\mathbb{Z}^c_\lambda))}. \label{GStr-d}
\end{align}

\end{prop}
Applying Keel and Tao's arguments in \cite{KeTa1998}, Proposition \ref{StrichartzG} also holds for the endpoint cases $\gamma (\theta,p) =2$ or  $\gamma (\theta,r) =2$, $p,r \neq \infty$ (cf.~\cite{WaHuHaGu2011}).
Comparing the cases $\lambda\gg 1$ with $\lambda=1$, one sees that the Strichartz estimates for $\lambda\gg 1$ in the lower (higher) frequency part become worse (better) than those for $\lambda=1$.   Taking $\theta =1$ in Proposition \ref{StrichartzG}, we immediately have

\begin{cor}\label{StrichartzG1}
Let $2\leqslant p,r \leqslant \infty$, $\delta(p)=1/2-1/p$, $2/\gamma( p) =d \delta(p)$ and $2\sigma (p) = (d+2)\delta(p)$. Assume that $2/\gamma(p), 2/\gamma(r) \in [0,1)$. Then we have
\begin{align}
\lambda^{-1/\gamma(p)}  & \| U_\lambda (t) u_0\|_{\widetilde{L}^{\gamma(p)} (\mathbb{R}, B^{\sigma,s}_{p,2} (\mathbb{Z}_\lambda)) } \lesssim \|u_0\|_{E^{\sigma,s}}, \label{G1Str-a}\\
\lambda^{-1/\gamma(p)}  & \left\| \mathscr{A}_\lambda f \right\|_{\widetilde{L}^{\gamma(p)} (\mathbb{R}, B^{\sigma,s}_{p,2} (\mathbb{Z}_\lambda))}
  \lesssim \lambda^{1/\gamma(r)} \|f\|_{\widetilde{L}^{\gamma(r)'} (\mathbb{R}, B^{\sigma,s}_{r',2} (\mathbb{Z}_\lambda))}.  \label{G1Str-b}
\end{align}
\begin{align}
\lambda^{\delta(p)}  & \| U_\lambda (t) u_0\|_{\widetilde{L}^{\gamma(p)} (\mathbb{R}, B^{\sigma-\sigma(p),s}_{p,2} (\mathbb{Z}^c_\lambda)) } \lesssim \|u_0\|_{E^{\sigma,s}}, \label{G1Str-c}\\
\lambda^{\delta(p) }   & \left\| \mathscr{A}_\lambda f \right\|_{\widetilde{L}^{\gamma(p)} (\mathbb{R}, B^{\sigma-\sigma(p),s}_{p,2} (\mathbb{Z}^c_\lambda))}
   \lesssim  \lambda^{ -\delta(r) }  \|f\|_{\widetilde{L}^{\gamma(r)'} (\mathbb{R}, B^{\sigma+\sigma(r),s}_{r',2} (\mathbb{Z}^c_\lambda))}. \label{G1Str-d}
\end{align}

\end{cor}

\begin{cor}\label{StrichartzG2}
Let $2\leqslant p,r \leqslant \infty$, $\delta(p)=1/2-1/p$, $2/\gamma_1( p) =(d-1) \delta(p)$ and $2\sigma_1 (p) = (d+1)\delta(p)$. Assume that $2/\gamma_1(p), 2/\gamma_1(r) \in [0,1)$. Then we have
\begin{align}
\lambda^{-1/\gamma_1 (p)}  & \| U_\lambda (t) u_0\|_{\widetilde{L}^{\gamma_1(p)} (\mathbb{R}, B^{\sigma-\delta(p),s}_{p,2} (\mathbb{Z}_\lambda)) } \lesssim \|u_0\|_{E^{\sigma,s}}, \label{G2Str-a}\\
\lambda^{-1/\gamma_1 (p)}  & \left\| \mathscr{A}_\lambda f \right\|_{\widetilde{L}^{\gamma_1(p)} (\mathbb{R}, B^{\sigma-\delta(p),s}_{p,2} (\mathbb{Z}_\lambda))}
  \lesssim \lambda^{1/\gamma_1(r)} \|f\|_{\widetilde{L}^{\gamma_1(r)'} (\mathbb{R}, B^{\sigma+\delta(r),s}_{r',2} (\mathbb{Z}_\lambda))}.  \label{G2Str-b}
\end{align}
\begin{align}
 &  \| U_\lambda (t) u_0\|_{\widetilde{L}^{\gamma_1 (p)} (\mathbb{R}, B^{\sigma-\sigma_1 (p),s}_{p,2} (\mathbb{Z}^c_\lambda)) } \lesssim \|u_0\|_{E^{\sigma,s}}, \label{G2Str-c}\\
  & \left\| \mathscr{A}_\lambda f \right\|_{\widetilde{L}^{\gamma_1 (p)} (\mathbb{R}, B^{\sigma-\sigma_1 (p),s}_{p,2} (\mathbb{Z}^c_\lambda))}
   \lesssim    \|f\|_{\widetilde{L}^{\gamma_1 (r)'} (\mathbb{R}, B^{\sigma+\sigma_1 (r),s}_{r',2} (\mathbb{Z}^c_\lambda))}. \label{G2Str-d}
\end{align}

\end{cor}

\section{Global solution of NLKG} \label{GlobalNLKG}

\subsection{Estimates for $u^{1+\alpha}$ with $\alpha \geqslant 4/(d-1)$} \label{GlobalNLKGA}

Taking $p=2+4/(d-1), \, d \geqslant 2$ in Corollary \ref{StrichartzG2},  one sees that
$$
\delta (p) = \frac{1}{d+1}, \ \ \sigma_1(p) =\frac{1}{2}, \ \ \gamma_1(p) = 2+ \frac{4}{d-1}.
$$
Let $\lambda \gg 1.$ According to the Strichartz estimates in  Corollary \ref{StrichartzG2}, we choose the following function spaces $X$ to solve NLKG:
\begin{align}
\|v\|_X = & \lambda^{-(d-1)/2(d+1)} \|v\|_{\widetilde{L}^{p} (\mathbb{R}, B^{\sigma-1/(d+1),s}_{p, 2}(\mathbb{Z}_\lambda))} +
 \|v\|_{\widetilde{L}^{\infty} (\mathbb{R}, B^{\sigma,s}_{2, 2}(\mathbb{Z}))} \nonumber\\
& +   \|v\|_{\widetilde{L}^{p} (\mathbb{R}, B^{\sigma-1/2,s}_{p, 2}(\mathbb{Z}^c_\lambda))}, \label{1workspace}
\end{align}
where
\begin{align*}
\mathbb{Z}_\lambda & =\{j\in \mathbb{Z}_+ : \ j\leqslant \log_2 \lambda \}, \ \  \mathbb{Z}^c_\lambda  =  \mathbb{Z}_+ \setminus  \mathbb{Z}_\lambda.
\end{align*}
For convenience, we denote
\begin{align}
\|v\|_{X_1} = &  \lambda^{-(d-1)/2(d+1)} \|u\|_{\widetilde{L}^{p} (\mathbb{R}, B^{\sigma-1/(d+1),s}_{p, 2}(\mathbb{Z}_\lambda))} +
 \|v\|_{\widetilde{L}^{\infty} (\mathbb{R}, B^{\sigma,s}_{2, 2}(\mathbb{Z}_\lambda))}, \label{1workspace1} \\
 \|v\|_{X_2} = &  \|u\|_{\widetilde{L}^{p} (\mathbb{R}, B^{\sigma-1/2,s}_{p, 2}(\mathbb{Z}^c_\lambda))}
 + \|v\|_{\widetilde{L}^{\infty} (\mathbb{R}, B^{\sigma,s}_{2, 2}(\mathbb{Z}^c_\lambda))}.
  \label{1workspace2}
\end{align}
Our goal is to solve the integral version of NLKG \eqref{NLKG2}:
\begin{align}
   \label{1intNLKG3}
v(t)= K'_\lambda(t) u_{0,\lambda} + K_\lambda (t) u_{1,\lambda} + \int^t_0 K_\lambda (t-\tau ) f(v(\tau)) d\tau.
\end{align}
We consider the mapping
\begin{align}
   \label{1intNLKG3map}
\mathcal{T}: v(t) \to  K'_\lambda(t) u_{0,\lambda} + K_\lambda (t) u_{1,\lambda} + \int^t_0 K_\lambda (t-\tau ) f(v(\tau)) d\tau.
\end{align}
in the space
\begin{align}
   \label{1metricspace}
\mathcal{D} & =\{v \in X: \ \|v\|_X \leqslant M_\lambda, \ {\rm supp }\, \widehat{v(t)} \subset \mathbb{R}^d_I \},  \nonumber\\
 & d(v,w) = \|v-w\|_X,
\end{align}
where $M_\lambda>0$ will be chosen later. By the multiplier estimate \eqref{multiplier},
$$
\left\|\triangle_j \frac{1}{\sqrt{\lambda^2 -\Delta}} f  \right\|_p \lesssim    2^{-j} \|f\|_p,   \ \ 2^j  \gtrsim \lambda,
$$
Using the Strichartz estimates, Corollary \ref{StrichartzG1},  we have
\begin{align} \label{1workspace3}
\|\mathcal{T}v\|_{X_2}  \lesssim &  \|u_{0,\lambda}\|_{E^{\sigma,s}}   +    \|u_{1,\lambda}\|_{E^{\sigma-1,s}} +  \left\| \frac{1}{\sqrt{\lambda^2 -\Delta}} f(v) \right\|_{\widetilde{L}^{p'} (\mathbb{R}, B^{\sigma+1/2,s}_{p', 2}(\mathbb{Z}^c_\lambda))} \nonumber\\
 \lesssim &  \|u_{0,\lambda}\|_{E^{\sigma,s}}   +    \|u_{1,\lambda}\|_{E^{\sigma-1,s}} +   \| f(v)\|_{\widetilde{L}^{p'} (\mathbb{R}, B^{\sigma-1/2,s}_{p', 2}(\mathbb{Z}^c_\lambda))}.
\end{align}
So, we need to make a nonlinear estimate to $ \|f(v)\|_{\widetilde{L}^{p'} (\mathbb{R}, B^{\sigma-1/2,s}_{p', 2}(\mathbb{Z}^c_\lambda))}$. First, we have

\begin{lemma} [Regularity Transference] \label{RegTransf}
Suppose that ${\rm supp} \, \widehat{u}_i \subset \mathbb{R}^d_I, \ i=1,...,\alpha+1$. Then we have
$$
2^{s|\nabla|} (u_1...u_{1+\alpha}) = (2^{s|\nabla|} u_1)(2^{s|\nabla|} u_2)...(2^{s|\nabla|} u_{\alpha+1}).
$$
\end{lemma}
\begin{proof}
Let us observe that
\begin{align}
\mathscr{F}(2^{s|\nabla|} (u_1...u_{1+\alpha})(\xi) =  2^{s|\xi|} \int  \widehat{u}_1(\xi_1)...\widehat{u}_\alpha(\xi_\alpha) \widehat{u}_{\alpha+1}(\xi-\xi_1-...-\xi_{\alpha+1}) d\xi_1...d\xi_\alpha.
\end{align}
Noticing that for any $\xi_1,...,\xi_{\alpha}, \xi-\xi_1-...-\xi_\alpha \in \mathbb{R}^d_I$, we have
$$
|\xi| = |\xi_1|+...+ |\xi_\alpha| + |\xi-\xi_1-...-\xi_\alpha|.
$$
We immediately have
\begin{align}
\mathscr{F}&   (2^{s|\nabla|} (u_1...u_{1+\alpha})(\xi) \nonumber\\
&=   \int_{(\mathbb{R}^d_I)^\alpha} 2^{s|\xi_1|}\widehat{u}_1(\xi_1)...2^{s|\xi_\alpha|}\widehat{u}_\alpha(\xi_\alpha) 2^{s|\xi-\xi_1-...-\xi_\alpha|} \widehat{u}_{\alpha+1}(\xi-\xi_1-...-\xi_{\alpha+1}) d\xi_1...d\xi_\alpha \nonumber\\
&=  \mathscr{F} [(2^{s|\nabla|} u_1)(2^{s|\nabla|} u_2)...(2^{s|\nabla|} u_{\alpha+1})] (\xi),
\end{align}
which implies the result, as desired.
\end{proof}
Now, we estimate
$  \|v^{1+\alpha}\|_{\widetilde{L}^{p'} (\mathbb{R}, B^{\sigma-1/2,s}_{p', 2}(\mathbb{Z}^c_\lambda))}.$  We have the following

\begin{lemma}\label{NonlinearestH1}
Let $d\geqslant 2$, $p= 2(d+1)/(d-1)$, $\alpha \in \mathbb{N}, \ \alpha \geqslant 4/(d-1)$, $\sigma \geqslant \sigma_c =d/2- 2/\alpha$, $\sigma>1/2$, $s< 0$. Assume that ${\rm supp} \, \widehat{v} \subset \mathbb{ R}^d_I$.  Then we have
\begin{align}
\|v^{1+\alpha} & \|_{\widetilde{L}^{p'} (\mathbb{R}, B^{\sigma-1/2,s}_{p', 2}(\mathbb{Z}^c_\lambda))} \nonumber\\
 \leqslant \
& (C \sqrt{\alpha})^{\alpha+C}   \left( \|v\|_{\widetilde{L}^p(\mathbb{R}, B^{\sigma-1/2,s}_{p,2}(\mathbb{Z}^c_{\lambda}))} \|v \|^{ 4/(d-1)}_{\widetilde{L}^p(\mathbb{R}, B^{\sigma_c -1/2,s}_{p,2}(\mathbb{Z}^c_{\lambda})) \cap \widetilde{L}^p(\mathbb{R}, B^{\sigma_c-\delta(p),s}_{p,2}(\mathbb{Z}_{\lambda}))} \right.
      \nonumber\\
 & \ \ \ \ \ \ \   + \left.  \lambda^{-1/p}  \|v\|_{\widetilde{L}^p(\mathbb{R}, B^{\sigma-\delta(p),s}_{p,2}(\mathbb{Z}_{\lambda}))} \|v\|^{ 4/(d-1)}_{\widetilde{L}^p(\mathbb{R}, B^{\sigma_c-\delta(p),s}_{p,2}(\mathbb{Z}_{\lambda}))} \right)
    \|v\|^{\alpha -4/(d-1)}_{\widetilde{L}^\infty(\mathbb{R}, B^{\sigma_c,s}_{2,2})}.  \label{NLE0}
\end{align}
\end{lemma}

By the definition, we have
\begin{align} \label{1NLE1}
\left\| u_1...u_{1+\alpha} \right\|_{\widetilde{L}^{p'}  (\mathbb{R}, B^{\sigma-1/2,s}_{p', 2}(\mathbb{Z}^c_\lambda))} = \left( \sum_{j>        j_\lambda}  2^{2(\sigma-1/2)j}  \|\triangle_j 2^{s|\nabla|} (u_1 u_2...u_{\alpha+1})\|^2_{L^{p'}_{x,t}} \right)^{1/2}.
\end{align}
Let ${\rm supp} \, \widehat{u}_i \subset \mathbb{R}^d_I, \ v_i = 2^{s|\nabla|}u_i, \ i=1,...,\alpha+1$. By Lemma \ref{RegTransf}, we have
\begin{align}
 \| u_1... & u_{1+\alpha}  \|_{\widetilde{L}^{p'}  (\mathbb{R}, B^{\sigma-1/2,s}_{p', 2}(\mathbb{Z}^c_\lambda))} \nonumber\\
 & = \left( \sum_{j>j_\lambda}  2^{2(\sigma-1/2)j}  \|\triangle_j   (v_1  ...v_{\alpha+1})\|^2_{L^{p'}_{x,t}} \right)^{1/2} \nonumber\\
  & \leqslant \left( \sum_{j>j_\lambda}  2^{2(\sigma-1/2)j} \left(\sum_{j_1,...,j_{1+\alpha}\in \mathbb{Z}_+}  \|\triangle_j   (\triangle_{j_1} v_1  ... \triangle_{j_{1+\alpha}} v_{\alpha+1})\|_{L^{p'}_{x,t}} \right)^2 \right)^{1/2}.
  \label{1NLE2}
\end{align}

{\it Step } 1. We consider the case $d\geqslant 5$. First, we assume that $\alpha \geqslant 2$. By symmetry, we can suppose that
\begin{align} \label{symmetry}
j_1 = \max_{1\leqslant k \leqslant \alpha+1} j_k, \ \  j_2 = \max_{2\leqslant k \leqslant \alpha+1} j_k, \ \ j_3 = \max_{3\leqslant k \leqslant \alpha+1} j_k
\end{align}
in the summation $\sum_{j_1,...,j_{1+\alpha}\in \mathbb{Z}_+} $ in \eqref{1NLE2}. Noticing that ${\rm supp} \ \widehat{v}_j \subset \mathbb{R}^d_I $, we see that
$$
\triangle_j ( \triangle_{j_1}v_1 \triangle_{j_2}v_2... \triangle_{j_{\alpha+1}}v_{\alpha+1}) \neq 0
$$
implies that $j-c(\alpha) \leqslant j_1 \leqslant j+1$, $c(\alpha) := [\log_2 (1+\alpha)] + 1$.  Put
\begin{align}
 \mathcal{Z}_0 & =\{(j_1,...,j_{1+\alpha}): \ j_\lambda \geqslant j_1 \geqslant 0, \ j-c(\alpha) < j_1 \leqslant j +1 \}, \nonumber\\
 \mathcal{Z}_1 & =\{(j_1,...,j_{1+\alpha}): \ j_1 > j_\lambda \geqslant j_2, \ j-c(\alpha) < j_1 \leqslant j  +1 \}, \nonumber\\
  \mathcal{Z}_2 & =\{(j_1,...,j_{1+\alpha}): \ j_2 > j_\lambda \geqslant j_3, \ j-c(\alpha) < j_1 \leqslant j  +1 \}, \nonumber\\
 \mathcal{Z}_3 & =\{(j_1,...,j_{1+\alpha}): \ j_3 > j_\lambda \geqslant \max_{k=4,...,1+\alpha} j_k, \ j-c(\alpha) < j_1 \leqslant j +1\}, \nonumber\\
 \mathcal{Z}_4 & =\{(j_1,...,j_{1+\alpha}): \   \max_{k=4,...,1+\alpha} j_k > j_\lambda , \ j-c(\alpha) < j_1 \leqslant j  +1 \}. \nonumber
\end{align}
Since $\varphi_j$ is a multiplier on $L^{p'}$, we have
\begin{align}
 \| u_1... & u_{1+\alpha}  \|_{\widetilde{L}^{p'}  (\mathbb{R}, B^{\sigma-1/2,s}_{p', 2}(\mathbb{Z}^c_\lambda))} \nonumber\\
  & \leqslant \sum^4_{\ell=0}  \left( \sum_{j>j_\lambda}  2^{2(\sigma-1/2)j} \left(\sum_{(j_1,...,j_{1+\alpha}) \in  \mathcal{Z}_\ell }  \| \triangle_{j_1} v_1  ... \triangle_{j_{1+\alpha}} v_{\alpha+1} \|_{L^{p'}_{x,t}} \right)^2 \right)^{1/2} \nonumber\\
  & = I +...+V.
  \label{1NLE2+}
\end{align}
Let us consider the estimate of $III$. By H\"older's and Bernstein's inequalities,  for $(j_1,...,j_{1+\alpha}) \in  \mathcal{Z}_2$,
\begin{align} \label{1NLE3+}
\|& \triangle_{j_1} v_1 ... \triangle_{j_{1+\alpha}}v_{\alpha+1} \|_{L^{p'}_{x,t}} \nonumber\\
  & \leqslant \|  \triangle_{j_1} v_1 \|_{L^{p}_{x,t}} \|  \triangle_{j_2} v_2 \|^{4/(d-1)}_{L^{p}_{x,t}}  \|  \triangle_{j_2} v_2 \|^{1- 4/(d-1)}_{L^{\infty}_{x,t}}    \prod^{\alpha+1}_{k=3} \|\triangle_{j_k} v_{k} \|_{L^{\infty}_{x,t}}  \nonumber\\
   & \leqslant C^\alpha \|  \triangle_{j_1} v_1 \|_{L^{p}_{x,t}} \|  \triangle_{j_2} v_2 \|^{4/(d-1)}_{L^{p}_{x,t}}  (2^{dj_2/2}\|  \triangle_{j_2} v_2 \|_{L^{\infty}_{t}L^2_x }) ^{1- 4/(d-1)}    \prod^{\alpha+1}_{k=3} (2^{dj_k/2} \|\triangle_{j_k} v_{k} \|_{L^{\infty}_{t}L^2_x}).
\end{align}
Let us observe that for $ \sigma_c = d/2-  2/\alpha$,
\begin{align} \label{1NLE4}
 \sum_{j_k \leqslant j_3}    2^{dj_k/2} \|\triangle_{j_k} v_{k} \|_{L^{\infty}_{t}L^2_x} \leqslant \, \sqrt{\alpha} \,  2^{ 2j_3 /\alpha } \|u_k\|_{\widetilde{L}^\infty(\mathbb{R}, B^{\sigma_c ,s}_{2,2})}.
\end{align}
It follows that
\begin{align} \label{1NLE5}
 \sum_{j_4,...,j_{1+\alpha} \leqslant j_3}   \prod^{1+\alpha}_{k=4}  ( 2^{dj_k/2} \|\triangle_{j_k} v_{k} \|_{L^{\infty}_{t}L^2_x} )  \leqslant (C \sqrt{\alpha})^{\alpha}   2^{ 2j_3  (1-2/ \alpha)} \prod^{1+\alpha}_{k=4} \|u_k\|_{\widetilde{L}^\infty(\mathbb{R}, B^{\sigma_c,s}_{2,2})}.
\end{align}
Hence, we have from  H\"older's and Young's inequalities that
\begin{align} \label{1NLE6}
 \sum_{(j_1,...,j_{1+\alpha})  \in \mathcal{Z}_2}  &   \| \triangle_{j_1} v_1  ... \triangle_{j_{1+\alpha}} v_{\alpha+1} \|_{L^{p'}_{x,t}}  \nonumber\\
  \leqslant & (C\sqrt{\alpha})^{\alpha}  \sum_{(j-c(\alpha)) \vee j_\lambda <  j_1 \leqslant j+1}  \| \triangle_{j_1} v_1\|_{L^{p}_{x,t}}   \nonumber\\
     &\times \sum_{j_\lambda <j_2\leqslant j_1} (2^{j_2(\sigma_c -1/2)} \|  \triangle_{j_2} v_2 \|_{L^{p}_{x,t}})^{4/(d-1)}  (2^{j_2 \sigma_c } \|  \triangle_{j_2} v_2 \|_{L^{\infty}_t L^2_x })^{1- 4/(d-1)} \nonumber\\
     &\times \sum_{j_3\leqslant j_2}  2^{2 (j_3 -j_2) (1-1/\alpha)}   2^{j_3 \sigma_c } \|\triangle_{j_3} v_3\|_{L^{\infty}_t L^2_x}   \prod^{1+\alpha}_{k=4} \|u_k\|_{\widetilde{L}^\infty(\mathbb{R}, B^{\sigma_c,s}_{2,2})}   \nonumber\\
   \leqslant & (C\sqrt{\alpha})^{\alpha}  \sum_{(j-c(\alpha)) \vee j_\lambda  < j_1 \leqslant j+1}  \| \triangle_{j_1} v_1\|_{L^{p}_{x,t}}   \nonumber\\
     &\times \|u_2\|^{4/(d-1)}_{\widetilde{L}^p(\mathbb{R}, B^{\sigma_c -1/2,s}_{p,2}(\mathbb{Z}^c_{\lambda}))}   \|u_2\|^{1-4/(d-1)}_{\widetilde{L}^\infty(\mathbb{R}, B^{\sigma_c ,s}_{2,2})}   \prod^{1+\alpha}_{k=3} \|u_k\|_{\widetilde{L}^\infty(\mathbb{R}, B^{\sigma_c,s}_{2,2})}.
\end{align}
Inserting the above estimate \eqref{1NLE6} into $III$, by Young's inequality, we have
\begin{align} \label{1NLE7}
III \leqslant \  & (C \sqrt{\alpha})^{\alpha} \|u_2\|^{4/(d-1)}_{\widetilde{L}^p(\mathbb{R}, B^{\sigma_c -1/2,s}_{p,2}(\mathbb{Z}^c_{\lambda}) )}   \|u_2\|^{1-4/(d-1)}_{\widetilde{L}^\infty(\mathbb{R}, B^{\sigma_c,s}_{2,2})}   \prod^{1+\alpha}_{k=3} \|u_k\|_{\widetilde{L}^\infty(\mathbb{R}, B^{\sigma_c,s}_{2,2})} \nonumber\\
     & \times \left( \sum_{j > j_\lambda } 2^{2(\sigma-1/2) j} \left(\sum_{(j-c(\alpha)) \vee j_\lambda  < j_1 \leqslant j+1}  \| \triangle_{j_1} v_1\|_{L^{p}_{x,t}} \right)^2  \right)^{1/2}  \nonumber\\
\leqslant \  & (C \sqrt{\alpha})^{  \alpha +C}  \|u_1\|_{\widetilde{L}^p(\mathbb{R}, B^{\sigma-1/2,s}_{p,2}(\mathbb{Z}^c_{\lambda}))}   \|u_2\|^{4/(d-1)}_{\widetilde{L}^p(\mathbb{R}, B^{\sigma_c-1/2,s}_{p,2}(\mathbb{Z}^c_{\lambda}) )} \nonumber\\
   & \times \|u_2\|^{1-4/(d-1)}_{\widetilde{L}^\infty(\mathbb{R}, B^{\sigma_c,s}_{2,2})}   \prod^{1+\alpha}_{k=3} \|u_k\|_{\widetilde{L}^\infty(\mathbb{R}, B^{\sigma_c,s}_{2,2})}.
\end{align}
From the estimates of $III$ we see that $IV,\ V$ can be handled in the same way as the manner of $III$, we omit the details of the proof.

Now we estimate $I$.   Using similar way as above and noticing that $\delta(p) \leqslant 1/2$, we obtain that
\begin{align} \label{1NLE8}
I \leqslant   & (C \sqrt{\alpha})^{\alpha}   \left(\sum_{j > j_\lambda } 2^{2(\sigma-1/2) j}  \left( \sum_{(j -c(\alpha))\vee 0 \leqslant j_1 \leqslant j+1 } \| \triangle_{j_1} v_1\|_{L^{p}_{x,t}}\right) ^2 \right)^{1/2} \nonumber\\
& \times \|u_2\|^{4/(d-1)}_{\widetilde{L}^p(\mathbb{R}, B^{\sigma_c-1/2,s}_{p,2}(\mathbb{Z}_{\lambda}) )}   \|u_2\|^{1-4/(d-1)}_{\widetilde{L}^\infty(\mathbb{R}, B^{\sigma_c,s}_{2,2})}   \prod^{1+\alpha}_{k=3} \|u_k\|_{\widetilde{L}^\infty(\mathbb{R}, B^{\sigma_c,s}_{2,2})} \nonumber\\
\leqslant   & (C \sqrt{\alpha})^{\alpha}  \lambda^{-1/p} \left(\sum_{j > j_\lambda } 2^{2(\sigma-\delta(p)) j}  \left( \sum_{(j -c(\alpha))\vee 0 \leqslant j_1 \leqslant j+1 } \| \triangle_{j_1} v_1\|_{L^{p}_{x,t}}\right) ^2 \right)^{1/2} \nonumber\\
& \times \|u_2\|^{4/(d-1)}_{\widetilde{L}^p(\mathbb{R}, B^{\sigma_c-1/2,s}_{p,2}(\mathbb{Z}_{\lambda}) )}   \|u_2\|^{1-4/(d-1)}_{\widetilde{L}^\infty(\mathbb{R}, B^{\sigma_c,s}_{2,2})}   \prod^{1+\alpha}_{k=3} \|u_k\|_{\widetilde{L}^\infty(\mathbb{R}, B^{\sigma_c,s}_{2,2})} \nonumber\\
  \leqslant  & (C \sqrt{\alpha})^{\alpha  +C }   \lambda^{-1/p}  \|u_1\|_{\widetilde{L}^p(\mathbb{R}, B^{\sigma-\delta(p),s}_{p,2}(\mathbb{Z}_{\lambda}) )} \nonumber\\
  & \times \|u_2\|^{4/(d-1)}_{\widetilde{L}^p(\mathbb{R}, B^{\sigma_c-\delta(p),s}_{p,2}(\mathbb{Z}_{\lambda}) )}   \|u_2\|^{1-4/(d-1)}_{\widetilde{L}^\infty(\mathbb{R}, B^{\sigma_c,s}_{2,2})}   \prod^{1+\alpha}_{k=3} \|u_k\|_{\widetilde{L}^\infty(\mathbb{R}, B^{\sigma_c,s}_{2,2})}
\end{align}
The estimate of $II$ is similar to that of $I$.  In order to remove the assumption \eqref{symmetry}, one needs to  consider $(\alpha+1)\alpha (\alpha-1) \sim \alpha^3$ many similar cases.

Next, we consider the case $d\geqslant 5, \ \alpha=1$.  Denote
$$
S_j u = \sum^j_{k=0} \triangle_j u, \ \ S_{[j_1,j_2]} u = S_{j_2}u - S_{j_1} u, \ \ S_{-1} u=0.
$$
Using Bony's para-product decomposition, we have
$$
v_1v_2 = \sum^\infty_{j_1=0} (\triangle_{j_1} v_1 S_{j_1} v_2 + S_{j_1-1} v_1 \triangle_{j_1} v_2  ).
$$
Hence,
\begin{align}
 \| u_1   u_2  \|  _{\widetilde{L}^{p'}  (\mathbb{R}, B^{\sigma-1/2,s}_{p', 2}(\mathbb{Z}^c_\lambda))}
  \leqslant & \left( \sum_{j>j_\lambda}  2^{2(\sigma-1/2)j} \left(\sum^\infty_{ j_1=0}  \| \triangle_{j} (\triangle_{j_1} v_1 S_{j_1} v_2)   \|_{L^{p'}_{x,t}} \right)^2 \right)^{1/2} \nonumber\\
  & + \left( \sum_{j>j_\lambda}  2^{2(\sigma-1/2)j} \left(\sum^\infty_{ j_1=0}  \| \triangle_{j} (S_{j_1-1} v_1 \triangle_{j_1} v_2)   \|_{L^{p'}_{x,t}} \right)^2 \right)^{1/2} \nonumber\\
  & = I +II.
  \label{1NLE2++}
\end{align}
It suffices to estimate $I$. In order to $\triangle_{j} (\triangle_{j_1} v_1 S_{j_1} v_2) \neq 0 $, we have $j< j_1 +4$.  It follows that
\begin{align}
I \leqslant & \left( \sum_{j>j_\lambda}  2^{2(\sigma-1/2)j} \left(\sum_{j-4 <  j_1 \leqslant j_\lambda  }  \| \triangle_{j} (\triangle_{j_1} v_1 S_{j_1} v_2)   \|_{L^{p'}_{x,t}} \right)^2 \right)^{1/2} \nonumber\\
& +  \left( \sum_{j>j_\lambda}  2^{2(\sigma-1/2)j} \left(\sum_{  j_1 > (j-4) \vee j_\lambda  }  \| \triangle_{j} (\triangle_{j_1} v_1 S_{j_1} v_2)   \|_{L^{p'}_{x,t}} \right)^2 \right)^{1/2}  \nonumber\\
=: & I_1 +I_2.
  \label{1NLE2++1}
\end{align}
One can further decompose $I_2$ by
\begin{align}
I_2 \leqslant &  \left( \sum_{j>j_\lambda}  2^{2(\sigma-1/2)j} \left(\sum_{  j_1 > (j-4) \vee j_\lambda  }  \| \triangle_{j} (\triangle_{j_1} v_1 S_{[j_\lambda , \ j_1]} v_2)   \|_{L^{p'}_{x,t}} \right)^2 \right)^{1/2}  \nonumber\\
 &  +   \left( \sum_{j>j_\lambda}  2^{2(\sigma-1/2)j} \left(\sum_{  j_1 > (j-4) \vee j_\lambda  }  \| \triangle_{j} (\triangle_{j_1} v_1 S_{ j_\lambda } v_2)   \|_{L^{p'}_{x,t}} \right)^2 \right)^{1/2} \nonumber\\
 =: & I_{21} + I_{22}.
  \label{1NLE2++2}
\end{align}
Now, let us consider the estimate of $I_{21}$. One sees that
$$
\frac{1}{p'} =  \frac{1}{p} + \frac{4}{(d-1)} \left(\frac{1}{p} - \frac{\sigma_c -1/2}{d}\right) + \left( 1- \frac{4}{(d-1)} \right)\left(\frac{1}{2} - \frac{\sigma_c}{d}\right)
$$
for $\sigma_c =d/2-2$.  Taking
$$
\frac{1}{p_1} =  \frac{1}{p} - \frac{\sigma_c-1/2}{d}, \ \ \frac{1}{p_2}  = \frac{1}{2} - \frac{\sigma_c}{d} = \frac{2}{d},
$$
by H\"older's inequality, we have for $j_1 > (j-4) \vee j_\lambda $,
\begin{align}
   \| \triangle_{j} & (\triangle_{j_1} v_1 S_{[j_\lambda , \ j_1]} v_2)   \|_{L^{p'}_{x,t}}     \nonumber\\
 &  \leqslant \|\triangle_{j_1} v_1\|_{L^p_{x,t}}  \|S_{[j_\lambda , \ j_1]} v_2 \|^{4/(d-1)}_{L^p_{t}L^{p_1}_x } \|S_{[j_\lambda , \ j_1]} v_2 \|^{1- 4/(d-1)}_{L^\infty_{t}L^{p_2}_x } .
  \label{1NLE2++3}
\end{align}
 Using the Sobolev embedding $B^{\sigma_c -1/2}_{p,2} \subset  L^{p_1}, \ B^{\sigma_c }_{2,2} \subset  L^{p_2}$, and $S_j: L^r \to L^r$,  we have from Minkowski's inequality that
\begin{align}
   \| \triangle_{j} & (\triangle_{j_1} v_1 S_{[j_\lambda , \ j_1]} v_2)   \|_{L^{p'}_{x,t}}     \nonumber\\
 &  \leqslant \|\triangle_{j_1} v_1\|_{L^p_{x,t}}  \| v_2 \|^{4/(d-1)}_{L^p (\mathbb{R} , B^{\sigma_c -1/2}_{p,2} (\mathbb{Z}^c_\lambda) } \| v_2 \|^{1- 4/(d-1)}_{L^\infty (\mathbb{R}, B^{\sigma_c}_{2,2})}   \nonumber\\
  &  \leqslant C   \|\triangle_{j_1} v_1\|_{L^p_{x,t}}  \| u_2 \|^{4/(d-1)}_{\widetilde{L}^p (\mathbb{R} , B^{\sigma_c-1/2,s}_{p,2} (\mathbb{Z}^c_\lambda) } \| u_2 \|^{1- 4/(d-1)}_{\widetilde{L}^\infty (\mathbb{R}, B^{\sigma_c,s}_{2,2})}.
  \label{1NLE2++4}
\end{align}
It follows that
\begin{align}
I_{21}  \leqslant C   \| u_1\|_{\widetilde{L}^p (\mathbb{R} , B^{\sigma-1/2,s}_{p,2} (\mathbb{Z}^c_\lambda) }  \| u_2 \|^{4/(d-1)}_{\widetilde{L}^p (\mathbb{R} , B^{\sigma_c-1/2,s}_{p,2} (\mathbb{Z}^c_\lambda) } \| u_2 \|^{1- 4/(d-1)}_{\widetilde{L}^\infty (\mathbb{R}, B^{\sigma_c,s}_{2,2})}.
  \label{1NLE2++5}
\end{align}
Similarly,
\begin{align}
I_{22}  & \leqslant C   \| u_1\|_{\widetilde{L}^p (\mathbb{R} , B^{\sigma-1/2,s}_{p,2} (\mathbb{Z}^c_\lambda) }  \| u_2 \|^{4/(d-1)}_{\widetilde{L}^p (\mathbb{R} , B^{\sigma_c-\delta(p),s}_{p,2} (\mathbb{Z}_\lambda) } \| u_2 \|^{1- 4/(d-1)}_{\widetilde{L}^\infty (\mathbb{R}, B^{\sigma_c,s}_{2,2})},
  \label{1NLE2++6} \\
 I  & \leqslant C  \lambda^{-1/p} \| u_1\|_{\widetilde{L}^p (\mathbb{R} , B^{\sigma-\delta(p),s}_{p,2} (\mathbb{Z}_\lambda) }  \| u_2 \|^{4/(d-1)}_{\widetilde{L}^p (\mathbb{R} , B^{\sigma_c -\delta(p),s}_{p,2} (\mathbb{Z}_\lambda) } \| u_2 \|^{1- 4/(d-1)}_{\widetilde{L}^\infty (\mathbb{R}, B^{\sigma_c,s}_{2,2})}.
  \label{1NLE2++7}
\end{align}

{\it Step 2.} We consider the case $d=4$. One has that $\alpha \geqslant 2$. We can assume that \eqref{symmetry} holds. Let $\mathcal{Z}_0,..., \mathcal{Z}_4$ be the same ones as in Step 1.  Let us connect the proof with \eqref{1NLE2+}. First, we consider the estimate of $IV$.  By H\"older's and Bernstein's inequalities,  for $(j_1,...,j_{1+\alpha}) \in  \mathcal{Z}_3$,
\begin{align} \label{1NLE3+a}
\|& \triangle_{j_1} v_1 ... \triangle_{j_{1+\alpha}}v_{\alpha+1})\|_{L^{p'}_{x,t}} \nonumber\\
  & \leqslant \|  \triangle_{j_1} v_1 \|_{L^{p}_{x,t}} \|  \triangle_{j_2} v_2 \|_{L^{p}_{x,t}} \|  \triangle_{j_3} v_3 \|^{1/3}_{L^{p}_{x,t}} \|  \triangle_{j_3} v_3 \|^{2/3}_{L^{\infty}_{x,t}}    \prod^{\alpha+1}_{k=4} \|\triangle_{j_k} v_{k} \|_{L^{\infty}_{x,t}}  \nonumber\\
   & \leqslant C^\alpha \|  \triangle_{j_1} v_1 \|_{L^{p}_{x,t}} \|\triangle_{j_2} v_2 \|_{L^{p}_{x,t}} \|  \triangle_{j_3} v_3 \|^{1/3}_{L^{p}_{x,t}}  (2^{dj_2/2}\|  \triangle_{j_3} v_3 \|_{L^{\infty}_{t}L^2_x }) ^{2/3}    \prod^{\alpha+1}_{k=4} (2^{dj_k/2} \|\triangle_{j_k} v_{k} \|_{L^{\infty}_{t}L^2_x}).
\end{align}
It follows from \eqref{1NLE3+a}, H\"older's and Young's inequality that
\begin{align} \label{1NLE3+b}
&\sum _{(j_1,...,j_{1+\alpha}) \in \mathcal{Z}_3} \| \triangle_{j_1} v_1 ... \triangle_{j_{1+\alpha}}v_{\alpha+1})\|_{L^{p'}_{x,t}} \nonumber\\
 &  \leqslant (C\sqrt{\alpha})^{\alpha}  \prod^{\alpha+1}_{k=4}   \|u_{k} \|_{\widetilde{L}^\infty (\mathbb{R}, B^{\sigma_c,s}_{2,2})}  \sum_{ (j-c(\alpha))\vee j_\lambda <j_1 \leqslant j+1  } \|  \triangle_{j_1} v_1 \|_{L^{p}_{x,t}} \sum_{ j_2\leqslant j_1} 2^{(\sigma_c-1/2)j_2}\|  \triangle_{j_2} v_2 \|_{L^{p}_{x,t}} \nonumber\\
   &\ \ \  \times \sum_{j_\lambda < j_3\leqslant j_2 } 2^{(j_3-j_2)(3/2-2/\alpha)} (2^{\sigma_c j_3}\|  \triangle_{j_3} v_3 \|_{L^{\infty}_{t}L^2_x }) ^{2/3} (2^{(\sigma_c-1/2)j_3}\|  \triangle_{j_3} v_3 \|_{L^{p}_{x,t}})^{1/3} \nonumber\\
 &  \leqslant C(C\sqrt{\alpha})^{\alpha} \|u_2\|_{\widetilde{L}^p (\mathbb{R}, B^{\sigma_c-1/2,s}_{p,2}(\mathbb{Z}^c_\lambda) )} \|u_3\|^{1/3}_{\widetilde{L}^p (\mathbb{R}, B^{\sigma_c-1/2,s}_{p,2}(\mathbb{Z}^c_\lambda) )}  \|u_3 \|^{2/3}_{\widetilde{L}^\infty (\mathbb{R}, B^{\sigma_c,s}_{2,2})} \nonumber\\
  & \ \ \ \times \prod^{\alpha+1}_{k=4}   \|u_{k} \|_{\widetilde{L}^\infty (\mathbb{R}, B^{\sigma_c,s}_{2,2})} \sum_{ (j-c(\alpha))\vee j_\lambda <j_1 \leqslant j+1  } \|  \triangle_{j_1} v_1 \|_{L^{p}_{x,t}}.
\end{align}
By Young's inequality,
\begin{align} \label{1NLE3+c}
IV
 &  \leqslant  (C \sqrt{\alpha})^{\alpha +C} \|u_1\|_{\widetilde{L}^p (\mathbb{R}, B^{\sigma -1/2,s}_{p,2}(\mathbb{Z}^c_\lambda) )}  \|u_2\|_{\widetilde{L}^p (\mathbb{R}, B^{\sigma_c-1/2,s}_{p,2}(\mathbb{Z}^c_\lambda) )} \|u_3\|^{1/3}_{\widetilde{L}^p (\mathbb{R}, B^{\sigma_c-1/2,s}_{p,2}(\mathbb{Z}^c_\lambda) )} \nonumber\\
  & \ \ \ \times \|u_3 \|^{2/3}_{\widetilde{L}^\infty (\mathbb{R}, B^{\sigma_c,s}_{2,2})} \prod^{\alpha+1}_{k=6}  \|u_{k} \|_{\widetilde{L}^\infty (\mathbb{R}, B^{\sigma_c,s}_{2,2})}  .
\end{align}
The other terms can be estimated by using the ideas in \eqref{1NLE3+c} and in Step 1, we omitted the details.

{\it Step 3.} For $d=2, 3$, the proof is easier than that of the case $d=4$, since $4/(d-1)=4$ for $d=2 $ and $4/(d-1)=2$ for $d=3$  are integers. In fact,  one can use the estimate in 2D,
\begin{align} \label{1NLE3+d}
\|  \triangle_{j_1} v_1 ... \triangle_{j_{1+\alpha}}v_{\alpha+1} \|_{L^{p'}_{x,t}}
   \leqslant \prod^5_{\ell=1} \|  \triangle_{j_\ell} v_\ell \|_{L^{p}_{x,t}}     \prod^{\alpha+1}_{k=4} \|\triangle_{j_k} v_{k} \|_{L^{\infty}_{x,t}}
\end{align}
instead of \eqref{1NLE3+a}. Similar to Step 2, we can assume that
$$
j_k = \max_{k \leqslant \ell \leqslant 1+\alpha} j_\ell, \ \ \ \ k=1,...,5.
$$
Then we can use an analogous way as in Step 2  to get the result, as desired. $\hfill\Box$ \\

\begin{cor}\label{corNonlinearestH1}
Let $d\geqslant 2$, $p= 2(d+1)/(d-1)$, $\alpha \in \mathbb{N}, \ \alpha \geqslant 4/(d-1)$, $\sigma \geqslant \sigma_c =d/2- 2/\alpha$, $\sigma>1/2$, $s< 0$, $\lambda\gg 1$. Let $X_1$ and $X_2$ be as in \eqref{1workspace1} and \eqref{1workspace2}.  Assume that ${\rm supp} \, \widehat{v} \subset \mathbb{ R}^d_I$.  Then we have
\begin{align}
\|v^{1+\alpha} \|_{\widetilde{L}^{p'} (\mathbb{R}, B^{\sigma-1/2,s}_{p', 2}(\mathbb{Z}^c_\lambda))} \leqslant (C \sqrt{\alpha})^{\alpha + C}   \lambda^{2/(d+1)} \|v\|^{1+\alpha}_{X_1 \cap X_2}.    \label{corNLE0}
\end{align}
\end{cor}

Now we consider the estimates of $\|\mathcal{T} v\|_{X_1}$. By the Strichartz estimates in Corollary \ref{StrichartzG2}  and the following multiplier estimate
$$
\left\|\triangle_j \frac{1}{\sqrt{\lambda^2 -\Delta}} f  \right\|_p \lesssim \lambda^{-1} \|f\|_p ,   \ \ 2^j\leqslant \lambda,
$$
we have
\begin{align} \label{2NLE1}
\|\mathcal{T}v\|_{X_1}  \lesssim &  \|u_{0,\lambda}\|_{E^{\sigma,s}}   +    \|u_{1,\lambda}\|_{E^{\sigma-1,s}} + \lambda^{1/p} \left\| \frac{1}{\sqrt{\lambda^2 -\Delta}} f(v) \right\|_{\widetilde{L}^{p'} (\mathbb{R}, B^{\sigma+ \delta(p),s}_{p', 2}(\mathbb{Z}_\lambda))} \nonumber\\
 \lesssim &  \|u_{0,\lambda}\|_{E^{\sigma,s}}   +    \|u_{1,\lambda}\|_{E^{\sigma-1,s}} + \lambda^{-1+ 1/p}  \| f(v)\|_{\widetilde{L}^{p'} (\mathbb{R}, B^{\sigma+ \delta(p),s}_{p', 2}(\mathbb{Z}_\lambda))}.
\end{align}
Noticing that $2^j \leqslant \lambda$ for $j\in \mathbb{Z}_\lambda$, we have
\begin{align} \label{2NLE2}
\|f(v)\|_{\widetilde{L}^{p'} (\mathbb{R}, B^{\sigma+ \delta(p),s}_{p', 2}(\mathbb{Z}_\lambda))} \lesssim \lambda^{2\delta(p)} \| f(v)\|_{\widetilde{L}^{p'} (\mathbb{R}, B^{\sigma - \delta(p),s}_{p', \infty}(\mathbb{Z}_\lambda))}.
\end{align}
It follows from \eqref{2NLE1} and \eqref{2NLE2} that
\begin{align} \label{2NLE3}
\|\mathcal{T}v\|_{X_1}
 \lesssim &  \|u_{0,\lambda}\|_{E^{\sigma,s}}   +    \|u_{1,\lambda}\|_{E^{\sigma-1,s}} + \lambda^{- 1/p}  \| f(v)\|_{\widetilde{L}^{p'} (\mathbb{R}, B^{\sigma - \delta(p),s}_{p',\infty}(\mathbb{Z}_\lambda))}.
\end{align}

So, we need to make a nonlinear estimate to $ \|v^{1+\alpha}\|_{\widetilde{L}^{p'} (\mathbb{R}, B^{\sigma-\delta(p),s}_{p', \infty}(\mathbb{Z}_\lambda))}$.  We have the following
 \begin{lemma} \label{NonlinearestL1}
Let $d\geqslant 2$, $p= 2(d+1)/(d-1)$, $\alpha \in \mathbb{N}, \ \alpha \geqslant 4/(d-1)$, $\sigma \geqslant \sigma_c =d/2- 2/\alpha$, $s< 0$. Assume that ${\rm supp} \, \widehat{v} \subset \mathbb{ R}^d_I$.  Then we have
\begin{align}
\|v^{1+\alpha}   \|_{\widetilde{L}^{p'} (\mathbb{R}, B^{\sigma-\delta(p),s}_{p', \infty}(\mathbb{Z}_\lambda))}
 \leqslant \
& (C  \sqrt{\alpha})^{\alpha   + C }\left( \lambda^{1/p}\|v\|_{\widetilde{L}^p(\mathbb{R}, B^{\sigma-1/2,s}_{p, 2}(\mathbb{Z}^c_{\lambda}))} + \|v \|_{\widetilde{L}^p(\mathbb{R}, B^{\sigma-\delta(p),s}_{p, 2}(\mathbb{Z}_{\lambda}))} \right)
      \nonumber\\
   & \ \ \ \ \ \ \ \ \ \ \times  \|v\|^{ 4/(d-1)}_{ \widetilde{L}^p(\mathbb{R}, B^{\sigma_c-\delta(p), s}_{p,2}(\mathbb{Z}_{\lambda}))}
    \|v\|^{\alpha -4/(d-1)}_{\widetilde{L}^\infty(\mathbb{R}, B^{\sigma_c,s}_{2, 2})}.  \label{2NLE4}
\end{align}
\end{lemma}
{\bf Proof.} We can assume that $\lambda \in 2^{\mathbb{N}}$.  Recall that $j_\lambda = \log_2\lambda$.   We consider the estimate
\begin{align} \label{2NLE5}
\|u_1 ... u_{1+\alpha}   \|_{\widetilde{L}^{p'} (\mathbb{R}, B^{\sigma-\delta(p),s}_{p', \infty}(\mathbb{Z}_\lambda))} = \sup_{ j \leqslant j_\lambda } 2^{(\sigma-\delta(p)) j } \|\triangle_j 2^{s|\nabla|} (u_1...u_{1+\alpha}) \|_{L^{p'}_{t,x}}.
\end{align}
Denote $v_\ell = 2^{s|\nabla|} u_\ell$. By Lemma \ref{RegTransf} we have
\begin{align} \label{2NLE6}
\triangle_j 2^{s|\nabla|} (u_1...u_{1+\alpha}) =  \triangle_j   (v_1...v_{1+\alpha}) = \sum_{j_1,...,j_{1+\alpha} \in \mathbb{Z}_+} \triangle_j   (\triangle_{j_1} v_1... \triangle_{j_{1+\alpha}}v_{1+\alpha}).
\end{align}
We can assume that, without loss of generality that in the summation $\sum_{j_1,...,j_{1+\alpha} \in \mathbb{Z}_+}$,
$$
j_1 =\max_{1\leqslant \ell \leqslant 1+\alpha} j_\ell, \ \ \  j_2 =\max_{2 \leqslant \ell \leqslant 1+\alpha} j_\ell.
$$
One easily sees that in \eqref{2NLE6}
\begin{align} \label{2NLE7}
(j-c(\alpha)) \vee 0 \leqslant j_1 \leqslant j+1, \ \ c(\alpha) = [\log_2(1+\alpha)]+3, \ \ \ j_2\leqslant j \   \mbox{\rm for } \   j_1=j+1.
\end{align}
It follows that for $j=j_\lambda$,
\begin{align} \label{2NLE8}
\|\triangle_{j_\lambda}  (v_1...v_{1+\alpha})\|_{L^{p'}_{x,t}}  \leqslant & \sum_{(j_\lambda-c(\alpha)) \vee 0 \leqslant j_1 \leqslant j_\lambda +1} \  \sum_{j_2,..., j_{1+\alpha}\leqslant j_\lambda } \|  \triangle_{j_\lambda}   (\triangle_{j_1} v_1... \triangle_{j_{1+\alpha}}v_{1+\alpha})  \|_{L^{p'}_{x,t}} \nonumber\\
\leqslant &    \sum_{j_2,..., j_{1+\alpha}\leqslant j_\lambda } \|  \triangle_{j_\lambda}   (\triangle_{j_\lambda +1} v_1 \triangle_{j_2} v_2... \triangle_{j_{1+\alpha}}v_{1+\alpha})  \|_{L^{p'}_{x,t}} \nonumber\\
& +  \sum_{(j_\lambda-c(\alpha)) \vee 0 \leqslant j_1 \leqslant j_\lambda} \  \sum_{j_2,..., j_{1+\alpha}\leqslant j_\lambda } \|  \triangle_{j_\lambda}   (\triangle_{j_1} v_1... \triangle_{j_{1+\alpha}}v_{1+\alpha})  \|_{L^{p'}_{x,t}}  \nonumber\\
= : & I +II.
\end{align}
Let $d\geqslant 5$.  Using similar way as in \eqref{1NLE6},
\begin{align} \label{2NLE9}
 I \leqslant &      \| \triangle_{j_\lambda +1} v_1\|_{L^{p}_{x,t}}
     \sum_{j_2\leqslant j_\lambda}   \  \sum_{ j_3,...,j_{1+\alpha} \leqslant j_2}    \|  \triangle_{j_2} v_2 \|^{4/(d-1)}_{L^{p}_{x,t}}   \|  \triangle_{j_2} v_2 \|^{1- 4/(d-1)}_{L^{\infty}_{t,x} } \prod^{1+\alpha}_{k=3} \|\triangle_{j_k} v_k \|_{L^{\infty}_{t,x}} \nonumber\\
\leqslant &      \| \triangle_{j_\lambda +1} v_1\|_{L^{p}_{x,t}}
     \sum_{j_2\leqslant j_\lambda}  2^{-2j_2/(d+1)}    (2^{(\sigma_c -\delta(p))j_2}  \|  \triangle_{j_2} v_2 \|_{L^{p}_{x,t}})^{4/(d-1)}  (2^{\sigma_c j_2}  \|  \triangle_{j_2} v_2 \|_{L^{\infty}_{t}L^2_x })^{1- 4/(d-1)} \nonumber\\
 & \ \ \ \ \ \ \ \ \ \ \ \ \ \ \times (C\sqrt{\alpha})^{\alpha} \prod^{1+\alpha}_{k=3} \|u_k\|_{\widetilde{L}^\infty(\mathbb{R}, B^{\sigma_c,s}_{2, 2})} \nonumber\\
  \leqslant & C (C\sqrt{\alpha})^{\alpha  }    \| \triangle_{j_\lambda +1} v_1\|_{L^{p}_{x,t}}
       \|u_2\|^{4/(d-1)}_{\widetilde{L}^p(\mathbb{R}, B^{\sigma_c -\delta(p),s}_{p, 2}(\mathbb{Z}_{\lambda}))}   \|u_2\|^{1-4/(d-1)}_{\widetilde{L}^\infty(\mathbb{R}, B^{\sigma_c ,s}_{2, 2})}   \prod^{1+\alpha}_{k=3} \|u_k\|_{\widetilde{L}^\infty(\mathbb{R}, B^{\sigma_c,s}_{2, 2})}.
\end{align}
Hence, we have
\begin{align} \label{2NLE10}
2^{(\sigma -\delta(p))j_\lambda} I
  \leqslant & (C\sqrt{\alpha})^{\alpha +C} \lambda^{1/p}    \| u_1\|_{\widetilde{L}^p(\mathbb{R}, B^{\sigma -1/2,s}_{p, 2}(\mathbb{Z}^c_{\lambda})) }
       \|u_2\|^{4/(d-1)}_{\widetilde{L}^p(\mathbb{R}, B^{\sigma_c -\delta(p),s}_{p, 2}(\mathbb{Z}_{\lambda}))}  \nonumber\\
        & \times \|u_2\|^{1-4/(d-1)}_{\widetilde{L}^\infty(\mathbb{R}, B^{\sigma_c ,s}_{2, 2})}   \prod^{1+\alpha}_{k=3} \|u_k\|_{\widetilde{L}^\infty(\mathbb{R}, B^{\sigma_c,s}_{2, 2})}.
\end{align}
Similarly,
\begin{align} \label{2NLE11}
 II \leqslant &    \sum_{(j_\lambda-c(\alpha)) \vee 0 \leqslant j_1 \leqslant j_\lambda}   \| \triangle_{j_1} v_1\|_{L^{p}_{x,t}}
     \sum_{j_2\leqslant j_\lambda} \ \|  \triangle_{j_2} v_2 \|^{4/(d-1)}_{L^{p}_{x,t}}   \|  \triangle_{j_2} v_2 \|^{1- 4/(d-1)}_{L^{\infty}_{t,x} } \nonumber\\
      & \ \ \ \ \ \ \ \ \ \times \sum_{ j_3,...,j_{1+\alpha} \leqslant j_2}  \prod^{1+\alpha}_{\ell =3} \|\triangle_{j_\ell} v_\ell \|_{L^{\infty}_{t,x}} \nonumber\\
  \leqslant & (C\sqrt{\alpha})^{\alpha  }    \sum_{(j_\lambda-c(\alpha)) \vee 0 \leqslant j_1 \leqslant j_\lambda}   \| \triangle_{j_1} v_1\|_{L^{p}_{x,t}}
       \|u_2\|^{4/(d-1)}_{\widetilde{L}^p(\mathbb{R}, B^{\sigma_c -\delta(p),s}_{p, 2}(\mathbb{Z}_{\lambda}))}   \nonumber\\
          & \ \ \ \ \ \ \ \ \ \ \times \|u_2\|^{1-4/(d-1)}_{\widetilde{L}^\infty(\mathbb{R}, B^{\sigma_c ,s}_{2, 2})} \prod^{1+\alpha}_{\ell =3} \|u_\ell\|_{\widetilde{L}^\infty(\mathbb{R}, B^{\sigma_c,s}_{2, 2})}.
\end{align}
It follows from Young's inequality and $\delta(p) < 1/2$ that
\begin{align} \label{2NLE12}
2^{(\sigma-\delta(p)) j_\lambda } II
  \leqslant &  (C  \sqrt{\alpha})^{\alpha+ C}      \|u_1\|_{\widetilde{L}^p(\mathbb{R}, B^{\sigma -\delta(p),s}_{p, 2}(\mathbb{Z}_{\lambda}))}
       \|u_2\|^{4/(d-1)}_{\widetilde{L}^p(\mathbb{R}, B^{\sigma_c -\delta(p),s}_{p, 2}(\mathbb{Z}_{\lambda}))}   \|u_2\|^{1-4/(d-1)}_{\widetilde{L}^\infty(\mathbb{R}, B^{\sigma_c ,s}_{2,2})} \nonumber\\
          & \ \ \ \ \ \ \ \ \ \ \times \prod^{1+\alpha}_{\ell =3} \|u_\ell\|_{\widetilde{L}^\infty(\mathbb{R}, B^{\sigma_c,s}_{2, 2})}.
\end{align}
By \eqref{2NLE10} and \eqref{2NLE12},
\begin{align} \label{2NLE13}
2^{(\sigma-\delta(p)) j_\lambda } &  \|\triangle_{j_\lambda}  (v_1...v_{1+\alpha})\|_{L^{p'}_{x,t}}  \nonumber\\
  \leqslant &  (C\sqrt{\alpha})^{\alpha+C} \left( \lambda^{1/p}\|u_1\|_{\widetilde{L}^p(\mathbb{R}, B^{\sigma -1/2,s}_{p, 2}(\mathbb{Z}^c_{\lambda})) } + \|u_1\|_{ \widetilde{L}^p(\mathbb{R}, B^{\sigma -\delta(p),s}_{p, 2}(\mathbb{Z}_{\lambda}))} \right)
       \nonumber\\
          &   \times \|u_2\|^{4/(d-1)}_{\widetilde{L}^p(\mathbb{R}, B^{\sigma_c -\delta(p),s}_{p, 2}(\mathbb{Z}_{\lambda}))}   \|u_2\|^{1-4/(d-1)}_{\widetilde{L}^\infty(\mathbb{R}, B^{\sigma_c ,s}_{2, 2})} \prod^{1+\alpha}_{\ell =3} \|u_\ell\|_{\widetilde{L}^\infty(\mathbb{R}, B^{\sigma_c,s}_{2, 2})}.
\end{align}
If $j<j_\lambda$, the estimate of $2^{(\sigma-\delta(p)) j }  \|\triangle_{j }  (v_1...v_{1+\alpha})\|_{L^{p'}_{x,t}}$ is similar to \eqref{2NLE12} and we have
\begin{align} \label{2NLE14}
\sup_{j<j_\lambda} 2^{(\sigma-\delta(p)) j}    \|\triangle_{j}  (v_1...v_{1+\alpha})\|_{L^{p'}_{x,t}}
  \leqslant &  (C\sqrt{\alpha})^{\alpha+C}   \|u_1\|_{ \widetilde{L}^p(\mathbb{R}, B^{\sigma -\delta(p),s}_{p,\infty}(\mathbb{Z}_{\lambda}))}\|u_2\|^{4/(d-1)}_{\widetilde{L}^p(\mathbb{R}, B^{\sigma_c -\delta(p),s}_{p,\infty}(\mathbb{Z}_{\lambda}))}
       \nonumber\\
          &   \times \|u_2\|^{1-4/(d-1)}_{\widetilde{L}^\infty(\mathbb{R}, B^{\sigma_c ,s}_{2,\infty})} \prod^{1+\alpha}_{\ell =3} \|u_\ell\|_{\widetilde{L}^\infty(\mathbb{R}, B^{\sigma_c,s}_{2,\infty})}.
\end{align}
Next, we consider the case $d =4$. Let us connect the estimate with \eqref{2NLE8} and one can further assume that
$$
j_3 = \max_{3\leqslant \ell \leqslant 1+\alpha} j_\ell.
$$
Recall that $4/(d-1) =4/3$ for $d=4$. By H\"older's inequality,
\begin{align} \label{2NLE9a}
 I \leqslant &      \| \triangle_{j_\lambda +1} v_1\|_{L^{p}_{x,t}}
     \sum_{j_3\leqslant j_2\leqslant j_\lambda}   \  \sum_{ j_4,...,j_{1+\alpha} \leqslant j_3}    \|  \triangle_{j_2} v_2 \|_{L^{p}_{x,t}}   \|  \triangle_{j_3} v_3 \|^{1/3}_{L^{p}_{x,t}} \|  \triangle_{j_3} v_3 \|^{2/3}_{L^{\infty}_{t,x} } \prod^{1+\alpha}_{k=4} \|\triangle_{j_k} v_k \|_{L^{\infty}_{t,x}} \nonumber\\
\leqslant &    (C\sqrt{\alpha})^{\alpha -4/3}  \| \triangle_{j_\lambda +1} v_1\|_{L^{p}_{x,t}}
     \sum_{j_2\leqslant j_\lambda}  2^{-j_2/p}    (2^{(\sigma_c -\delta(p))j_2}  \|  \triangle_{j_2} v_2 \|_{L^{p}_{x,t}})  \nonumber\\
 & \ \ \ \ \ \ \ \ \  \ \ \ \  \times \sum_{j_3\leqslant j_2} 2^{-j_3/3p} (2^{(\sigma_c-\delta(p)) j_3}  \|  \triangle_{j_3} v_3 \|_{L^{p}_{x,t} })^{1/3}  \|u_3\|^{2/3}_{\widetilde{L}^\infty(\mathbb{R}, B^{\sigma_c ,s}_{2, 2})} \prod^{1+\alpha}_{k=4} \|u_k\|_{\widetilde{L}^\infty(\mathbb{R}, B^{\sigma_c,s}_{2, 2})} \nonumber\\
  \leqslant & (C\sqrt{\alpha})^{\alpha -4/3}    \| \triangle_{j_\lambda +1} v_1\|_{L^{p}_{x,t}}
       \|u_2\|_{\widetilde{L}^p(\mathbb{R}, B^{\sigma_c -\delta(p),s}_{p, 2}(\mathbb{Z}_{\lambda}))}   \|u_3\|^{1/3}_{\widetilde{L}^p(\mathbb{R}, B^{\sigma_c -\delta(p),s}_{p, 2}(\mathbb{Z}_{\lambda}))}  \nonumber\\
        & \times
        \|u_3\|^{2/3}_{\widetilde{L}^\infty(\mathbb{R}, B^{\sigma_c ,s}_{2, 2})}   \prod^{1+\alpha}_{k=4} \|u_k\|_{\widetilde{L}^\infty(\mathbb{R}, B^{\sigma_c,s}_{2, 2})}.
\end{align}
So, we have
\begin{align} \label{2NLE10a}
2^{(\sigma -\delta(p))j_\lambda} I
  \leqslant & (C\sqrt{\alpha})^{\alpha} \lambda^{1/p}    \| u_1\|_{\widetilde{L}^p(\mathbb{R}, B^{\sigma -1/2,s}_{p, 2}(\mathbb{Z}^c_{\lambda})) }
       \|u_2\| _{\widetilde{L}^p(\mathbb{R}, B^{\sigma_c -\delta(p),s}_{p, 2}(\mathbb{Z}_{\lambda}))}  \|u_3\|^{1/3} _{\widetilde{L}^p(\mathbb{R}, B^{\sigma_c -\delta(p),s}_{p, 2}(\mathbb{Z}_{\lambda}))}  \nonumber\\
        & \times \|u_3\|^{2/3}_{\widetilde{L}^\infty(\mathbb{R}, B^{\sigma_c ,s}_{2, 2})}   \prod^{1+\alpha}_{k=4} \|u_k\|_{\widetilde{L}^\infty(\mathbb{R}, B^{\sigma_c,s}_{2, 2})}.
\end{align}
Using an analogous way as in   \eqref{2NLE9a} and \eqref{2NLE12}, we have
\begin{align} \label{2NLE12a}
2^{(\sigma-\delta(p)) j_\lambda } II
  \leqslant &  (C  \sqrt{\alpha})^{\alpha+C}      \|u_1\|_{\widetilde{L}^p(\mathbb{R}, B^{\sigma -\delta(p),s}_{p, 2}(\mathbb{Z}_{\lambda}))}
       \|u_2\|_{\widetilde{L}^p(\mathbb{R}, B^{\sigma_c -\delta(p),s}_{p, 2}(\mathbb{Z}_{\lambda}))}  \|u_3\|^{1/3}_{\widetilde{L}^p(\mathbb{R}, B^{\sigma_c -\delta(p),s}_{p, 2}(\mathbb{Z}_{\lambda}))} \nonumber\\
          & \ \ \ \ \  \times \|u_3\|^{2/3}_{\widetilde{L}^\infty(\mathbb{R}, B^{\sigma_c ,s}_{2,\infty})} \prod^{1+\alpha}_{\ell =3} \|u_\ell\|_{\widetilde{L}^\infty(\mathbb{R}, B^{\sigma_c,s}_{2, 2})}.
\end{align}
By \eqref{2NLE10a} and \eqref{2NLE12a},
\begin{align} \label{2NLE13a}
2&^{(\sigma-\delta(p)) j_\lambda } \|\triangle_{j_\lambda}  (v_1...v_{1+\alpha})\|_{L^{p'}_{x,t}}  \nonumber\\
  & \leqslant (C\sqrt{\alpha})^{\alpha+C} \left( \lambda^{1/p}\|u_1\|_{\widetilde{L}^p(\mathbb{R}, B^{\sigma -1/2,s}_{p, 2}(\mathbb{Z}^c_{\lambda})) } + \|u_1\|_{ \widetilde{L}^p(\mathbb{R}, B^{\sigma -\delta(p),s}_{p, 2}(\mathbb{Z}_{\lambda}))} \right)
       \|u_2\|_{\widetilde{L}^p(\mathbb{R}, B^{\sigma_c -\delta(p),s}_{p, 2}(\mathbb{Z}_{\lambda}))}  \nonumber\\
          &   \ \ \ \times \|u_3\|^{1/3}_{\widetilde{L}^p(\mathbb{R}, B^{\sigma_c -\delta(p),s}_{p, 2}(\mathbb{Z}_{\lambda}))}   \|u_3\|^{2/3}_{\widetilde{L}^\infty(\mathbb{R}, B^{\sigma_c ,s}_{2, 2})} \prod^{1+\alpha}_{\ell =3} \|u_\ell\|_{\widetilde{L}^\infty(\mathbb{R}, B^{\sigma_c,s}_{2, 2})}.
\end{align}
For any $j<j_\lambda$, the estimate of $2^{(\sigma-\delta(p)) j } \|\triangle_{j }  (v_1...v_{1+\alpha})\|_{L^{p'}_{x,t}}$ is similar to that of $2^{(\sigma-\delta(p)) j } II$ and we omit the details.

Thirdly, noticing that $4/(d-1)$ are integers for $d=2, 3$, the proofs of the results for $d=2,3$ are easier than that of $d=4$ and the details are omitted. $\hfill\Box$

\begin{cor}\label{corNonlinearestL1}
Let $d\geqslant 2$, $p= 2(d+1)/(d-1)$, $\alpha \in \mathbb{N}, \ \alpha \geqslant 4/(d-1)$, $\sigma \geqslant \sigma_c =d/2- 2/\alpha$, $s< 0$, $\lambda\gg 1$. Let $X_1$ and $X_2$ be as in \eqref{1workspace1} and \eqref{1workspace2}.  Assume that ${\rm supp} \, \widehat{v} \subset \mathbb{ R}^d_I$.  Then we have
\begin{align}
\lambda^{-1/p}\|v^{1+\alpha} \|_{\widetilde{L}^{p'} (\mathbb{R}, B^{\sigma-\delta(p),s}_{p', 2}(\mathbb{Z}_\lambda))} \leqslant (C \sqrt{ \alpha})^{\alpha + C}   \lambda^{2/(d+1)} \|v\|^{1+\alpha}_{X_1 \cap X_2}.    \label{corNLE0L}
\end{align}
\end{cor}

In the case $4/d \leqslant \alpha < 4/(d-1)$ for $d \geqslant 2$, and for $4 \leqslant \alpha <\infty$ for $d=1$, we cannot reach the critical index $\sigma \geqslant \sigma_c =d/2-2/\alpha$. We need the condition $\sigma \geqslant 1/2$.
Put
\begin{align} \label{lowpindex2}
r = \frac{2(d+1+\theta)}{d-1+\theta}, \ \ \  \theta \in (0,1].
\end{align}
Let $\delta(r), \, \sigma(\theta, r), \, \gamma (\theta, r)$ be as in Proposition \ref{StrichartzG}.  We see that
\begin{align} \label{lowpindex3}
 \delta(r) =  \frac{1}{d+1+\theta}, \ \ \sigma (\theta, r) =  \frac{1}{2}, \ \ \gamma (\theta, r) =r.
\end{align}
\begin{align}
& \lambda^{-1/r}  \| U_\lambda (t) u_0\|_{\widetilde{L}^{r} (\mathbb{R}, B^{\sigma-(1-\theta)\delta(r),s}_{r,2} (\mathbb{Z}_\lambda)) } \lesssim \|u_0\|_{E^{\sigma,s}}, \label{lowpGStr-a}\\
& \lambda^{-1/r}  \left\| \mathscr{A}_\lambda f \right\|_{\widetilde{L}^{r} (\mathbb{R}, B^{\sigma-(1-\theta)\delta(r),s}_{r,2} (\mathbb{Z}_\lambda))}
   \lesssim \lambda^{1/r} \|f\|_{\widetilde{L}^{r'} (\mathbb{R}, B^{\sigma+(1-\theta)\delta(r),s}_{r',2} (\mathbb{Z}_\lambda))}.  \label{lowpGStr-b}
\end{align}
\begin{align}
& \lambda^{ \theta \delta(r)}  \| U_\lambda (t) u_0\|_{\widetilde{L}^{r} (\mathbb{R}, B^{\sigma-1/2,s}_{r,2} (\mathbb{Z}^c_\lambda)) } \lesssim \|u_0\|_{E^{\sigma,s}}, \label{lowpGStr-c}\\
& \lambda^{ \theta \delta(r) }   \left\| \mathscr{A}_\lambda f \right\|_{\widetilde{L}^{r} (\mathbb{R}, B^{\sigma-1/2,s}_{r,2} (\mathbb{Z}^c_\lambda))}
  \lesssim  \lambda^{ -\theta \delta(r) }  \|f\|_{\widetilde{L}^{r'} (\mathbb{R}, B^{\sigma+1/2,s}_{r',2} (\mathbb{Z}^c_\lambda))}. \label{lowpGStr-d}
\end{align}
In view of the above Strichartz estiates, one can choose the following resolution spaces:
\begin{align}
\|v\|_{W_1} = &  \lambda^{-1/r} \|u\|_{\widetilde{L}^{r} (\mathbb{R}, B^{\sigma-(1-\theta)\delta(r) ,s}_{r, 2}(\mathbb{Z}_\lambda))} +
 \|v\|_{\widetilde{L}^{\infty} (\mathbb{R}, B^{\sigma,s}_{2, 2}(\mathbb{Z}_\lambda))}, \label{lowp1workspace1} \\
 \|v\|_{W_2} = &  \lambda^{\theta\delta (r)}\|u\|_{\widetilde{L}^{r} (\mathbb{R}, B^{\sigma-1/2,s}_{r, 2}(\mathbb{Z}^c_\lambda))}
 + \|v\|_{\widetilde{L}^{\infty} (\mathbb{R}, B^{\sigma,s}_{2, 2}(\mathbb{Z}^c_\lambda))}, \label{lowp1workspace2} \\
 \|v\|_{W} =&  \|v\|_{W_1} + \|v\|_{W_2}. \label{lowp1workspace3}
\end{align}
Our goal is to solve the integral equation \eqref{1intNLKG3}
in the space
\begin{align}
   \label{1metricspace}
\mathcal{D} & =\{v \in W: \ \|v\|_W \leqslant \delta \lambda^{-2(1-\theta)/(d+1+\theta)}, \ {\rm supp }\, \widehat{v(t)} \subset \mathbb{R}^d_I \},  \nonumber\\
 & d(v,w) = \|v-w\|_W,
\end{align}
where $\delta>0$ is sufficiently small. By \eqref{lowpGStr-c} and \eqref{lowpGStr-d},
\begin{align} \label{lowpworkspace3}
\|\mathcal{T}v\|_{W_2}  \lesssim &
  \|u_{0,\lambda}\|_{E^{\sigma,s}}   +    \|u_{1,\lambda}\|_{E^{\sigma-1,s}} +  \lambda^{-\theta \delta(r)} \| f(v)\|_{\widetilde{L}^{r'} (\mathbb{R}, B^{\sigma-1/2,s}_{r', 2}(\mathbb{Z}^c_\lambda))}.\\
  \|\mathcal{T}v\|_{W_1}  \lesssim &
  \|u_{0,\lambda}\|_{E^{\sigma,s}}   +    \|u_{1,\lambda}\|_{E^{\sigma-1,s}} +  \lambda^{-1/r-2\theta \delta(r)} \| f(v)\|_{\widetilde{L}^{r'} (\mathbb{R}, B^{\sigma - (1-\theta) \delta(r),s}_{r', 2}(\mathbb{Z}_\lambda))}. \label{lowpworkspace3a}
\end{align}
We have the following nonlinear estimates:
\begin{lemma}\label{lowpNonlinearestH1}
Let $d\geqslant 1$,  $\alpha \in \mathbb{N}, \ \alpha  = 4/(d-1+ \theta), \, \theta \in (0,1]$, $\sigma >1/2$, $s< 0$. Let $r$ be as in \eqref{lowpindex2}.  Assume that ${\rm supp} \, \widehat{v} \subset \mathbb{ R}^d_I$.  Then we have
\begin{align}
\|v^{1+\alpha}  \|_{\widetilde{L}^{r'} (\mathbb{R}, B^{\sigma-1/2,s}_{r', 2}(\mathbb{Z}^c_\lambda))}
 \leqslant \
& C^{\alpha+1}    \|v\|_{\widetilde{L}^r(\mathbb{R}, B^{\sigma-1/2,s}_{r,2}(\mathbb{Z}^c_{\lambda}))} \|v \|^{ \alpha}_{\widetilde{L}^r(\mathbb{R}, B^{\sigma -1/2,s}_{r,2}(\mathbb{Z}^c_{\lambda})) \cap \widetilde{L}^r(\mathbb{R}, B^{\sigma-(1-\theta)\delta(r),s}_{r,2}(\mathbb{Z}_{\lambda}))}
      \nonumber\\
 & \ \ \ \ \ \ \   +   C^{\alpha+1}  \lambda^{-1/r -\theta \delta(r)}  \|v\|^{1+ \alpha}_{\widetilde{L}^r(\mathbb{R}, B^{\sigma-(1-\theta) \delta(r),s}_{r,2}(\mathbb{Z}_{\lambda}))}
    .  \label{lowpNLE0}
\end{align}
\end{lemma}

  \begin{lemma} \label{lowpNonlinearestL1}
Let $d\geqslant 1$,  $\alpha \in \mathbb{N}, \ \alpha  = 4/(d-1+ \theta), \, \theta \in (0,1]$, $\sigma >1/2$, $s< 0$. Let $r$ be as in \eqref{lowpindex2}.  Assume that ${\rm supp} \, \widehat{v} \subset \mathbb{ R}^d_I$.  Then we have
\begin{align}
\|v^{1+\alpha}   \|_{\widetilde{L}^{r'} (\mathbb{R}, B^{\sigma - (1-\theta) \delta(r),s}_{r', 2}(\mathbb{Z}_\lambda))}
 \leqslant \
&  C^{\alpha }  \lambda^{1/r +\theta \delta(r)}\|v\|_{\widetilde{L}^r(\mathbb{R}, B^{\sigma-1/2,s}_{r, 2}(\mathbb{Z}^c_{\lambda}))}  \|v \|^\alpha_{\widetilde{L}^r(\mathbb{R}, B^{\sigma-(1-\theta) \delta(r),s}_{r, 2}(\mathbb{Z}_{\lambda}))}
      \nonumber\\
   & \ \ \ +  \|v\|^{1+\alpha}_{ \widetilde{L}^r(\mathbb{R}, B^{\sigma -(1-\theta)\delta(r), s}_{r,2}(\mathbb{Z}_{\lambda}))}
    .  \label{lowp2NLE4}
\end{align}
\end{lemma}

  \begin{cor} \label{lowpNonlinearestLH}
Let $d\geqslant 1$,  $\alpha \in \mathbb{N}, \ \alpha  = 4/(d-1+ \theta), \, \theta \in (0,1]$, $\sigma >1/2$, $s< 0$.    Assume that ${\rm supp} \, \widehat{v} \subset \mathbb{ R}^d_I$. Let $W$ be as in \eqref{lowp1workspace3}.  Then we have
\begin{align} \label{lowpworkspace3aa}
\|\mathcal{T}v\|_{W}  \lesssim &
  \|u_{0,\lambda}\|_{E^{\sigma,s}}   +    \|u_{1,\lambda}\|_{E^{\sigma-1,s}} +  \lambda^{2(1-\theta)/(d+1+\theta)} \|v\|^{1+\alpha}_{W}.
\end{align}
\end{cor}

\subsection{Estimates to $\sinh u-u $ in $2$D} \label{GlobalNLKGB}

Observing  Taylor's expansion of $\sinh u -u = \sum^\infty_{k=1} \frac{u^{2k+1}}{(2k+1) !} $, we see that its lowest power is $u^3/6$.  The technique in Section \ref{GlobalNLKGA} works very well  for $u^3$  when $d\geq 3$, but it does not work for $d=2$. So, we will apply Corollary \ref{StrichartzG1} to handle $u^3$ in 2D.  Taking  $d=2$ in Corollary \ref{StrichartzG1}, one sees that
$$
\delta (4) = \frac{1}{4}, \ \ \sigma (4) =\frac{1}{2}, \ \ \gamma (4) = 4.
$$
According to the Strichartz estimates in  Corollary \ref{StrichartzG1}, we choose the following function spaces $Y$ to solve NLKG:
\begin{align}
\|v\|_{Y_1} = & \lambda^{-1/4} \|v\|_{\widetilde{L}^{4} (\mathbb{R}, B^{\sigma,s}_{4, 2}(\mathbb{Z}_\lambda))} +
 \|v\|_{\widetilde{L}^{\infty} (\mathbb{R}, B^{\sigma,s}_{2, 2}(\mathbb{Z}_\lambda ))}  \label{2dworkspace1}\\
\|v\|_{Y_2} = &    \lambda^{ 1/4}  \|v\|_{\widetilde{L}^{4} (\mathbb{R}, B^{\sigma-1/2,s}_{4, 2}(\mathbb{Z}^c_\lambda))} +
 \|v\|_{\widetilde{L}^{\infty} (\mathbb{R}, B^{\sigma,s}_{2, 2}(\mathbb{Z}^c_\lambda ))}.  \label{2dworkspace2}
\end{align}
where $\mathbb{Z}_\lambda $  and  $ \mathbb{Z}^c_\lambda  $ are the same ones as in Section \ref{GlobalNLKGA}.  Denote
$$
\|v\|_Y =  \|v\|_{Y_1} +  \|v\|_{Y_2}.
$$

Our goal is to solve the integral  equation \eqref{1intNLKG3} by considering the mapping \eqref{1intNLKG3map} for $f(v)= \sinh v -v$
in the space
\begin{align}
   \label{2dmetricspace}
\mathcal{D}_2 & =\{v \in Y: \ \|v\|_Y \leqslant M, \ {\rm supp }\, \widehat{v(t)} \subset \mathbb{R}^d_I \},  \nonumber\\
 & d(v,w) = \|v-w\|_Y,
\end{align}
where $M>0$ will be chosen later. Using Corollary \ref{StrichartzG1},
we have
\begin{align} \label{2dworkspace3}
\|\mathcal{T}v\|_{Y_2}
 \lesssim  \|u_{0,\lambda}\|_{E^{\sigma,s}}   +    \|u_{1,\lambda}\|_{E^{\sigma-1,s}} +  \lambda^{-1/4} \| f(v)\|_{\widetilde{L}^{4/3} (\mathbb{R}, B^{\sigma-1/2,s}_{4/3, 2}(\mathbb{Z}^c_\lambda))}.
\end{align}
So, we need to make a nonlinear estimate to $  \lambda^{-1/4} \|f(v)\|_{\widetilde{L}^{4/3} (\mathbb{R}, B^{\sigma-1/2,s}_{4/3, 2}(\mathbb{Z}^c_\lambda))}$.

\begin{lemma}\label{2dNonlinearestH1}
Let $d = 2$,   $\alpha \in \mathbb{N}, \ \alpha \geqslant 2$, $\sigma \geqslant \sigma_\alpha:= 1-1/\alpha$, $\sigma >1/2$, $s< 0$. Assume that ${\rm supp} \, \widehat{v} \subset \mathbb{ R}^d_I$.  Then we have
\begin{align}
 \lambda^{-1/4} \|v^{1+\alpha} & \|_{\widetilde{L}^{4/3} (\mathbb{R}, B^{\sigma-1/2,s}_{4/3, 2}(\mathbb{Z}^c_\lambda))}
 \leqslant \  (C \sqrt{\alpha})^{\alpha  + C}    \|v\|^{1+\alpha}_Y.  \label{2dNLE0}
\end{align}
\end{lemma}
In Lemma \ref{NonlinearestH1}, we need condition $\alpha \geq 4$ in 2D, where the critical regularity index $\sigma_c$ is attainable. In order to handle the nonlinear power $\alpha=2$,  we introduce $\sigma_\alpha \, (> \sigma_c) $ in Lemma \ref{2dNonlinearestH1}, which is not the critical index for the nonlinearity $v^{1+\alpha}$. However, if we consider $f(v) = e^{u^2} u-u$ or $f(v) =\sinh v-v$ as nonlinearity in in 2D, $\sigma =1$ is the critical regularity index, Lemma \ref{2dNonlinearestH1} is enough for our purpose.

Now we estimate
\begin{align} \label{2d1NLE1}
\left\| u_1...u_{1+\alpha} \right\|_{\widetilde{L}^{4/3}  (\mathbb{R}, B^{\sigma-1/2,s}_{4/3, 2}(\mathbb{Z}^c_\lambda))} = \left( \sum_{j>        j_\lambda}  2^{2(\sigma-1/2)j}  \|\triangle_j 2^{s|\nabla|} (u_1 u_2...u_{\alpha+1})\|^2_{L^{4/3}_{x,t}} \right)^{1/2}.
\end{align}
Take $v_i = 2^{s|\nabla|} u_i$. Following the same way as in \eqref{1NLE2}--\eqref{1NLE2+}, we have
\begin{align}
 \| u_1... & u_{1+\alpha}  \|_{\widetilde{L}^{4/3}  (\mathbb{R}, B^{\sigma-1/2,s}_{4/3, 2}(\mathbb{Z}^c_\lambda))} \nonumber\\
  & \leqslant \sum^4_{\ell=0}  \left( \sum_{j>j_\lambda}  2^{2(\sigma-1/2)j} \left(\sum_{(j_1,...,j_{1+\alpha}) \in  \mathcal{Z}_\ell }  \| \triangle_{j_1} v_1  ... \triangle_{j_{1+\alpha}} v_{\alpha+1} \|_{L^{4/3}_{x,t}} \right)^2 \right)^{1/2} \nonumber\\
  & = I +...+V.
  \label{2d1NLE2+}
\end{align}
First, we estimate $II$. By H\"older's  and Bernstein's inequalities,  for $(j_1,...,j_{1+\alpha}) \in  \mathcal{Z}_1$,
\begin{align} \label{2d1NLE3+}
\|  \triangle_{j_1} v_1 ... \triangle_{j_{1+\alpha}}v_{\alpha+1} \|_{L^{4/3}_{x,t}}
   & \leqslant C^{\alpha} \prod^{3}_{\ell =1} \|  \triangle_{j_\ell} v_\ell \|_{L^{4}_{x,t}}      \prod^{\alpha+1}_{k=4} (2^{dj_k/2} \|\triangle_{j_k} v_{k} \|_{L^{\infty}_{t}L^2_x}).
\end{align}
If $\alpha \geq 3$,  we have from  H\"older's and Young's inequalities that
\begin{align} \label{2d1NLE6}
 \sum_{(j_1,...,j_{1+\alpha})  \in \mathcal{Z}_1}  &   \| \triangle_{j_1} v_1  ... \triangle_{j_{1+\alpha}} v_{\alpha+1} \|_{L^{4/3}_{x,t}}  \nonumber\\
 \leqslant & \, C^{\alpha}  \sum_{j_1 > (j-c(\alpha)) \vee j_\lambda \geqslant j_2 \geqslant j_3 }  \prod^{3}_{\ell =1} \|  \triangle_{j_\ell} v_\ell \|_{L^{4}_{x,t}}    \sum_{j_4 ,..., j_{1+\alpha} \leqslant j_3}       \prod^{\alpha+1}_{k=3} (2^{dj_k/2} \|\triangle_{j_k} v_{k} \|_{L^{\infty}_{t}L^2_x}) \nonumber\\
 \leqslant &   (C\sqrt{\alpha})^{\alpha}  \sum_{j_1 > (j-c(\alpha)) \vee  j_\lambda \geqslant j_2 \geqslant j_3 } 2^{(1-2/\alpha) j_3}    \prod^{3}_{\ell =1} \|  \triangle_{j_\ell} v_\ell \|_{L^{4}_{x,t}}   \prod^{1+\alpha}_{\ell =4}  \|u_\ell\|_{\widetilde{L}^\infty(\mathbb{R}, B^{\sigma_\alpha,s}_{2,2})} \nonumber  \\
\leqslant & (C\sqrt{\alpha})^{\alpha}  \sum_{(j-c(\alpha)) \vee j_\lambda  <  j_1 \leqslant j+1}  \| \triangle_{j_1} v_1\|_{L^{4}_{x,t}}   \sum_{ j_2\leqslant j_\lambda}  (2^{j_2(\sigma_\alpha -1/2)} \|  \triangle_{j_2} v_2 \|_{L^{4}_{x,t}})   \nonumber\\
     &\times
     \sum_{j_3\leqslant j_2}  2^{2 (j_3 -j_2) (1/2-1/\alpha)}   (2^{j_3 (\sigma_\alpha-1/2) } \|\triangle_{j_3} v_3\|_{L^{4}_{x,t}})   \prod^{1+\alpha}_{k=4} \|u_k\|_{\widetilde{L}^\infty(\mathbb{R}, B^{\sigma_\alpha,s}_{2,2})}   \nonumber\\
   \leqslant &  C (C\sqrt{\alpha})^{\alpha }  \sum_{(j-c(\alpha)) \vee j_\lambda  < j_1 \leqslant j+1}  \| \triangle_{j_1} v_1\|_{L^{4}_{x,t}}   \nonumber\\
     &\times \|u_2\|_{\widetilde{L}^4(\mathbb{R}, B^{\sigma_\alpha -1/2,s}_{4,2}(\mathbb{Z}_{\lambda}))}   \|u_3\|_{\widetilde{L}^4(\mathbb{R}, B^{\sigma_\alpha-1/2 ,s}_{4,2}(\mathbb{Z}_{\lambda}) )}   \prod^{1+\alpha}_{k=4} \|u_k\|_{\widetilde{L}^\infty(\mathbb{R}, B^{\sigma_\alpha,s}_{2,2})}.
\end{align}
Inserting the above estimate into $II$, and using Young's inequality, we have
\begin{align} \label{2d1NLE7}
II
   \leqslant &  (C\sqrt{\alpha})^{\alpha+C}  \|u_1\|_{\widetilde{L}^4(\mathbb{R}, B^{\sigma -1/2,s}_{4,2}(\mathbb{Z}^c_{\lambda}))}
      \prod^{3}_{k=2} \|u_k\|_{\widetilde{L}^4(\mathbb{R}, B^{\sigma_\alpha -1/2,s}_{4,2}(\mathbb{Z}_{\lambda}))}   \prod^{1+\alpha}_{k=4} \|u_k\|_{\widetilde{L}^\infty(\mathbb{R}, B^{\sigma_\alpha,s}_{2,2})}.
\end{align}
Similarly,
\begin{align} \label{2d1NLE8}
& III \leqslant  (C\sqrt{\alpha})^{\alpha+C} \prod^{2}_{k=1} \|u_k\|_{\widetilde{L}^4(\mathbb{R}, B^{\sigma -1/2,s}_{4,2}(\mathbb{Z}^c_{\lambda}))}
      \|u_3\|_{\widetilde{L}^4(\mathbb{R}, B^{\sigma_\alpha -1/2,s}_{4,2}(\mathbb{Z}_{\lambda}))}   \prod^{1+\alpha}_{k=4} \|u_k\|_{\widetilde{L}^\infty(\mathbb{R}, B^{\sigma_\alpha,s}_{2,2})},\\
 & IV+V \leqslant  (C\sqrt{\alpha})^{\alpha+C} \|u_1\|_{\widetilde{L}^4(\mathbb{R}, B^{\sigma -1/2,s}_{4,2}(\mathbb{Z}^c_{\lambda}))} \prod^{3}_{k=2} \|u_k\|_{\widetilde{L}^4(\mathbb{R}, B^{\sigma_\alpha -1/2,s}_{4,2}(\mathbb{Z}^c_{\lambda}))}
          \prod^{1+\alpha}_{k=4} \|u_k\|_{\widetilde{L}^\infty(\mathbb{R}, B^{\sigma_\alpha,s}_{2,2})}. \label{2d1NLE9}
\end{align}
Now we consider the estimate of $I$. Since $2^j > \lambda$ for $j>j_\lambda$, we have
\begin{align}
 I   & \leqslant  \lambda^{-1/2} \left( \sum_{j>j_\lambda}  2^{2 \sigma j} \left(\sum_{(j_1,...,j_{1+\alpha}) \in  \mathcal{Z}_\ell }  \| \triangle_{j_1} v_1  ... \triangle_{j_{1+\alpha}} v_{\alpha+1} \|_{L^{4/3}_{x,t}} \right)^2 \right)^{1/2}.
  \label{2d1NLE10}
\end{align}
Then we can use an analogous way as in the estimates of $II$ to obtain that
\begin{align}
  I \leqslant  \lambda^{-1/2}  (C\sqrt{\alpha})^{\alpha+C} \|u_1\|_{\widetilde{L}^4(\mathbb{R}, B^{\sigma,s}_{4,2}(\mathbb{Z}_{\lambda}))} \prod^{3}_{k=2} \|u_k\|_{\widetilde{L}^4(\mathbb{R}, B^{\sigma_\alpha,s}_{4,2}(\mathbb{Z}_{\lambda}))}
          \prod^{1+\alpha}_{k=4} \|u_k\|_{\widetilde{L}^\infty(\mathbb{R}, B^{\sigma_\alpha,s}_{2,2})}. \label{2d1NLE11}
\end{align}
Next, for $\alpha =2$, the proof becomes easier than that of $\alpha \geqslant 3$. Say, we consider the estimate of $II$. Noticing that $\sigma >1/2$, we have
\begin{align} \label{2d1NLE12}
 \sum_{(j_1,j_2,j_3)  \in \mathcal{Z}_1} \| \triangle_{j_1} v_1  ... \triangle_{j_3} v_{3} \|_{L^{4/3}_{x,t}}
 \leqslant  &  \sum_{j_1 > (j-3) \vee j_\lambda \geqslant j_2 \geqslant j_3 }  \prod^{3}_{\ell =1} \|  \triangle_{j_\ell} v_\ell \|_{L^{4}_{x,t}}    \nonumber\\
 \lesssim  &     \sum_{j_1 > (j-3) \vee  j_\lambda  }   \|  \triangle_{j_1} v_1 \|_{L^{4}_{x,t}}   \prod^{3}_{\ell =2}  \|u_\ell\|_{\widetilde{L}^4(\mathbb{R}, B^{\sigma-1/2,s}_{4,2} (\mathbb{Z}_\lambda))}.
\end{align}
It follows that
\begin{align} \label{2d1NLE13}
II
   \lesssim &    \|u_1\|_{\widetilde{L}^4(\mathbb{R}, B^{\sigma -1/2,s}_{4,2}(\mathbb{Z}^c_{\lambda}))}
      \prod^{3}_{k=2} \|u_k\|_{\widetilde{L}^4(\mathbb{R}, B^{\sigma -1/2,s}_{4,2}(\mathbb{Z}_{\lambda}))}.
\end{align}

Now we estimate $\|\mathscr{T}v\|_{Y_1}$. By the Strichartz estimates in Corollary \ref{StrichartzG1},
\begin{align} \label{2d2NLE1}
\|\mathcal{T}v\|_{Y_1}
 \lesssim &  \|u_{0,\lambda}\|_{E^{\sigma,s}}   +    \|u_{1,\lambda}\|_{E^{\sigma-1,s}} + \lambda^{-3/4}  \| f(v)\|_{\widetilde{L}^{4/3} (\mathbb{R}, B^{\sigma,s}_{4/3, 2}(\mathbb{Z}_\lambda))}.
\end{align}
Using Taylor's expansion of $\sinh v -v$, one needs to bound $\|v^{1+\alpha}\|_{\widetilde{L}^{4/3} (\mathbb{R}, B^{\sigma,s}_{4/3, 2}(\mathbb{Z}_\lambda))}$. We have the following

\begin{lemma}\label{2dcorNonlinearestL1}
Let $d= 2$,   $\alpha \in \mathbb{N}, \ \alpha \geqslant 2$, $\sigma \geqslant \sigma_c =d/2- 2/\alpha$, $\sigma>0, \ s< 0$, $\lambda\gg 1$. Let $Y_1$ and $Y_2$ be as in \eqref{2dworkspace1} and \eqref{2dworkspace2}.  Assume that ${\rm supp} \, \widehat{v} \subset \mathbb{ R}^d_I$.  Then we have
\begin{align}
\lambda^{-3/4}\|v^{1+\alpha} \|_{\widetilde{L}^{4/3} (\mathbb{R}, B^{\sigma,s}_{4/3, 2}(\mathbb{Z}_\lambda))} \leqslant  (C\sqrt{\alpha}) ^{\alpha  + C}     \|v\|^{1+\alpha}_{Y_1 \cap Y_2}.    \label{2dcorNLE0L}
\end{align}
\end{lemma}
\begin{proof}
By definition,
\begin{align} \label{2d1NLE1}
\left\| u_1...u_{1+\alpha} \right\|_{\widetilde{L}^{4/3}  (\mathbb{R}, B^{\sigma,s}_{4/3, 2}(\mathbb{Z}_\lambda))} = \left( \sum_{j \leqslant       j_\lambda}  2^{2 \sigma j}  \|\triangle_j 2^{s|\nabla|} (u_1 u_2...u_{\alpha+1})\|^2_{L^{4/3}_{x,t}} \right)^{1/2}.
\end{align}
Take $v_i = 2^{s|\nabla|} u_i$. One has that
\begin{align}
 \| u_1...   u_{1+\alpha}  \|_{\widetilde{L}^{4/3}  (\mathbb{R}, B^{\sigma,s}_{4/3, 2}(\mathbb{Z}_\lambda))}
  \leqslant   \left( \sum_{j \leqslant j_\lambda}  2^{2 \sigma j} \left(\sum_{ j_1,...,j_{1+\alpha}  \in  \mathbb{Z}_+ }  \| \triangle_{j} ( \prod^{1+\alpha}_{\ell=1}\triangle_{j_\ell} v_\ell  ) \|_{L^{4/3}_{x,t}} \right)^2 \right)^{1/2}.
  \label{2dlNLE2+}
\end{align}
By the symmetry, we can assume that
$$
j_1 = \max_{1 \leqslant k \leqslant 1+\alpha} j_k, \ \  j_2 = \max_{2 \leqslant k \leqslant 1+\alpha} j_k, \ \ j_3 = \max_{3 \leqslant k \leqslant 1+\alpha} j_k
$$
in the above summation $\sum_{j_1,...,j_{1+\alpha}}$. Noticing that
$$
\triangle_{j} (   \triangle_{j_1} v_1 ...  \triangle_{j_{1+\alpha}} v_{1+\alpha}  ) =0, \ \ j_1, \, j_2\geqslant j+1,
$$
one has that
\begin{align}
 \| u_1...  &  u_{1+\alpha}  \|_{\widetilde{L}^{4/3}  (\mathbb{R}, B^{\sigma-1/2,s}_{4/3, 2}(\mathbb{Z}_\lambda))} \nonumber\\
  \leqslant &  \,  2^{ \sigma j_\lambda}  \sum_{ j_2,...,j_{1+\alpha}  \leqslant j_\lambda }  \| \triangle_{j_\lambda} ( \triangle_{j_\lambda +1} v_1 \triangle_{j_2} v_2 ...  \triangle_{j_{1+\alpha}} v_{1+\alpha})    \|_{L^{4/3}_{x,t}}   \nonumber\\
 & +  \left( \sum_{j \leqslant j_\lambda}  2^{2 \sigma j} \left(\sum_{ j_1,...,j_{1+\alpha} \leqslant j_\lambda }  \| \triangle_{j} ( \triangle_{j_1} v_1 ...  \triangle_{j_{1+\alpha}} v_{1+\alpha} ) \|_{L^{4/3}_{x,t}} \right)^2 \right)^{1/2} \nonumber\\
 : =  & I+II.
  \label{2dlNLE2+a}
\end{align}
Using an analogous way as in Lemma \ref{NonlinearestH1}, we have
\begin{align}
  II \leqslant     (C\sqrt{\alpha})^{\alpha+C} \prod^{3}_{k=1} \|u_k\|_{\widetilde{L}^4(\mathbb{R}, B^{\sigma,s}_{4,2}(\mathbb{Z}_{\lambda}))}
          \prod^{1+\alpha}_{k=4} \|u_k\|_{\widetilde{L}^\infty(\mathbb{R}, B^{\sigma_c,s}_{2,2})}. \label{2dl1NLE11}
\end{align}
Using H\"older's inequality and similar way as in  Lemma \ref{NonlinearestH1},
\begin{align}
I \leqslant &  \,  2^{\sigma j_\lambda}  \sum_{ j_2,...,j_{1+\alpha}  \leqslant j_\lambda }  \|  \triangle_{j_\lambda +1} v_1\|_{L^4_{x,t}}  \prod^{3}_{k=2} \|\triangle_{j_k} v_k\|_{L^{4}_{x,t}}    \prod^{1+\alpha}_{k=4} \| \triangle_{j_k} v_{k}   \|_{L^\infty_{x,t}}   \nonumber\\
 \leqslant &   (C\sqrt{\alpha})^{\alpha}   2^{ \sigma j_\lambda}     \|  \triangle_{j_\lambda +1} v_1\|_{L^4_{x,t}}  \prod^{3}_{k=2} \|u_k\|_{\widetilde{L}^4(\mathbb{R}, B^{\sigma,s}_{4,2} (\mathbb{Z}_\lambda)) }  \prod^{1+\alpha}_{k=4} \|u_k\|_{\widetilde{L}^\infty(\mathbb{R}, B^{\sigma_c,s}_{2,2})} \nonumber\\
  \leqslant &  (C\sqrt{\alpha})^{\alpha}   \lambda^{1/2}     \| u_1\|_{\widetilde{L}^4(\mathbb{R}, B^{\sigma-1/2,s}_{4,2} (\mathbb{Z}^c_\lambda)) }  \prod^{3}_{k=2} \|u_k\|_{\widetilde{L}^4(\mathbb{R}, B^{\sigma,s}_{4,2} (\mathbb{Z}_\lambda)) }  \prod^{1+\alpha}_{k=4} \|u_k\|_{\widetilde{L}^\infty(\mathbb{R}, B^{\sigma_c,s}_{2,2})}
  \label{2dl1NLE12}
\end{align}
Combining the estimates of $I$ and $II$, we obtain the result, as desired.
\end{proof}

\subsection{Proof of Theorem \ref{mainresult}}
\label{proofofthm1.2}

\begin{rem}
We remark that the support condition in Theorem \ref{mainresult} can be slightly improved by the following conditions:
\begin{align*}
& {\rm supp} \,\widehat{u}_0 \subset \mathbb{R}^d_I \ \mbox{for} \ \alpha > 4/d,  \ \
{\rm supp} \,\widehat{u}_0 \subset \mathbb{R}^d_I \setminus \{0\} \ \mbox{for} \ \alpha = 4/d \in \mathbb{N};   \\
& {\rm supp} \,\widehat{u}_1 \subset \mathbb{R}^d_I \  \mbox{for} \ \alpha > \frac{4(d+2)}{(d-2)(d+1)}, \ d \geqslant 3,\\
& {\rm supp} \,\widehat{u}_1 \subset \mathbb{R}^d_I\setminus \{0\} \ \mbox{for} \   \alpha \leqslant \frac{4(d+2)}{(d-2)(d+1)} \ (\alpha< \infty \ \mbox{for} \ d=1, 2).
\end{align*}
We will prove our result by assuming the above support conditions.
\end{rem}

{\it Step 1.} We consider the case $\alpha \geqslant 4/(d-1), \, \sigma >1/2$.	Let us connect the proof with \eqref{1workspace3} and \eqref{2NLE3}. By Corollaries \ref{corNonlinearestH1} and \ref{corNonlinearestL1}, noticing that $X=X_1\cap X_2$,  we have
\begin{align} \label{3NLE3}
 & \|\mathcal{T}v\|_{X}
 \leqslant C  (\|u_{0,\lambda}\|_{E^{\sigma,s}}   +    \|u_{1,\lambda}\|_{E^{\sigma-1,s}}) +   C \lambda^{2/(d+1)}  \|v\|^{1+\alpha}_{X}, \\
& \|\mathcal{T}v_1 - \mathcal{T}v_2\|_{X }
 \leqslant C   \lambda^{2/(d+1)}  (\|v_1\|^{ \alpha}_{X } + \|v_2\|^{ \alpha}_{X })\|v_1 -v_2\|_{X } . \label{3NLE4}
\end{align}
$\alpha \geqslant 4/(d-1)$ implies that $d/2-2/\alpha > 2/\alpha(1+d)$. So, there exists $\epsilon>0$  such that
$$
\frac{d}{2} -\frac{2}{\alpha} \geqslant 2 \epsilon + \frac{2}{\alpha(1+d)}.
$$
By Corollary \ref{Scaling4}, there exists $\lambda_0>1$ such that for any $\lambda > \lambda_0$
\begin{align}
& \| u_{0,\lambda} \|_{E^{\sigma,s} } \leqslant C \lambda^{-d/2+\sigma+ \epsilon}   \| u_0 \|_{E^{\sigma,s}   } \leqslant C \lambda^{-2/\alpha(1+d)  -\epsilon}   \| u_0 \|_{E^{\sigma,s}   },   \label{scalinges22}
\end{align}
If $\alpha > 4(d+2)/(d-2)(1+\alpha)$, then  there exists $\epsilon>0$  such that
$$
\frac{d}{2} -\frac{2}{\alpha} - 1 \geqslant 2 \epsilon+  \frac{2}{\alpha(1+d)},
$$
where we can assume that $\epsilon$ is the same one as in \eqref{scalinges22}.  It follows from Corollary \ref{Scaling4} that  for $\lambda>\lambda_0$,
\begin{align}
& \| u_{1,\lambda}  \|_{E^{\sigma-1,s} } \leqslant C \lambda^{1-d/2+2/\alpha + \epsilon }   \|u_1 \|_{E^{\sigma-1, s} } \leqslant C \lambda^{-2/\alpha(1+d)  -\epsilon}   \| u_0 \|_{E^{\sigma,s}   }.  \label{1scalinges22}
\end{align}
 If $ {\rm supp }\, \widehat{u}_1 \subset \{\xi: |\xi|\geqslant \varepsilon_0\}$ for some $\varepsilon_0>0$,  by Corollary \ref{Scaling4}, there exists $\lambda_1>\lambda_0$ such that for any $\lambda>\lambda_1$,
\begin{align}
  \| u_{1,\lambda}  \|_{E^{\sigma-1,s} } & \leqslant C \lambda^{1-d/2+2/\alpha + (\sigma-1)\vee 0 } 2^{s\lambda \varepsilon_0 /3}   \|u_1 \|_{E^{\sigma-1, s} } \nonumber\\
& \leqslant C \lambda^{-2/\alpha(1+d)  -\epsilon}  \|u_1 \|_{E^{\sigma-1, s} }.     \label{1scalinges22a}
\end{align}
Let $\delta >0$  satisfy $C \delta^\alpha \leqslant 1/100$ and
$$
M_\lambda =  \lambda^{-2/\alpha(1+d)} \delta.
$$
Take  a $\lambda_2 > \lambda_1$ verifying
\begin{align}
  C \lambda^{-\epsilon } (\|u_0\|_{E^{\sigma,s}}  +  \|u_1\|_{E^s_{\sigma-1}})  \leqslant \delta/2.  \label{3NLE5}
\end{align}
If $\lambda > \lambda_2$, then we have
$$
C  (\|u_{0,\lambda}\|_{E^{\sigma,s}}   +    \|u_{1,\lambda}\|_{E^{\sigma-1,s}}) \leqslant M_\lambda/2.
$$
Then we see that for $\lambda>\lambda_2$ and for any $v, v_1, v_2 \in \mathcal{D}$,
\begin{align} \label{3NLE6}
 & \|\mathcal{T}v\|_{X}
 \leqslant M_\lambda /2  +  C \lambda^{2/(d+1)} M_\lambda^{1+\alpha} \leqslant M_\lambda/2 + C\delta^\alpha M_\lambda \leqslant M_\lambda, \\
& \|\mathcal{T}v_1 - \mathcal{T}v_2\|_{X }
 \leqslant   2C\delta^\alpha \|v_1 -v_2\|_{X } \leqslant \frac{1}{2} \|v_1 -v_2\|_{X } . \label{3NLE7}
\end{align}
It follows from \eqref{3NLE6} and \eqref{3NLE7} that $\mathcal{T}$ has a unique fixed point  $v\in \mathcal{D}$, which is a mild solution of
\eqref{1intNLKG3} with $f(v) = \pm v^{1+\alpha}$.

In order to go back to the solution of NLKG \eqref{NLKG}, we make the scaling to the solution $v$ of \eqref{1intNLKG3}  and let
$$
u(x,t) = \lambda^{-2/\alpha} v(\lambda^{-1}x, \lambda^{-1}t),
$$
By the scaling Lemma \ref{Scaling1},  we have
\begin{align*}
& 	\| \varphi(\lambda^{-1}\, \cdot) \|_{E^{\sigma, s\lambda_1}} \lesssim \lambda^{d/2+} \|\varphi \|_{E^{\sigma,s} }, \ \ \sigma \geqslant 0,\\
& \|  \varphi (\lambda^{-1}\, \cdot) \|_{E^{\sigma, s\lambda_1} } \lesssim \lambda^{ d/2-\sigma}  \|\varphi \|_{E^{ \sigma,s} }, \ \ \sigma < 0.
\end{align*}
We see that $u$ satisfies
\begin{align}
   \label{1NLKG3}
u(t)= K'_1 (t) u_0 + K_1 (t) u_1 + \int^t_0 K_1 (t-\tau ) u^{1+\alpha}(\tau) d\tau,
\end{align}
which is the mild solution of NLKG \eqref{NLKG} in the space $\widetilde{L}^\infty(\mathbb{R}, B^{\sigma_c,s\lambda}_{2, 2}) \cap \widetilde{L}^p(\mathbb{R}, B^{\sigma_c-1/2,s\lambda}_{p, 2})$ by Lemma \ref{Scaling2}. We have proven Theorem \ref{mainresult} in the case $\alpha \geqslant 4/(d-1), \,\sigma >1/2$.

{\it Step 2.}
 We consider the case $4/d \leqslant \alpha \leqslant 4/(d-1)$, $\alpha \neq \infty$ for $d=1$. First, we consider the case $\sigma=1/2$.  We can find some $\theta \in [0,1]$ ($\theta \in (0,1]$ for $d=1$) satisfying
\begin{align} \label{index1}
\alpha = \frac{4}{d-1+\theta}.
\end{align}
Put
\begin{align} \label{index2}
r = \frac{2(d+1+\theta)}{d-1+\theta}, \ \ \  \theta \in
\left\{
\begin{array}{ll}
[0,1] & d\geqslant 2;\\
(0,1] & d=1.
\end{array}
\right.
\end{align}
Recall that
$
 \delta(r) =  1/(d+1+\theta), \, \sigma (\theta, r) =  1/2, \, \gamma (\theta, r) =r.
$
Applying Lemmas \ref{Embedding2} and \ref{Isomorphism3},  we have
$$
B^{\rho,s}_{r,2} \subset  F^{\rho,s}_{r,2} = H^{\rho,s}_{r}, \ \  H^{\rho,s}_{r'} = F^{\rho,s}_{r',2} \subset  B^{\rho,s}_{r',2},
$$
where $1/r+1/r'=1$. Denote $P_{>\lambda} = \mathscr{F}^{-1} \chi_{|\xi|>\lambda} \mathscr{F}, \, P_{\leqslant \lambda} = I- P_{>\lambda}.$  It follows from Proposition \ref{StrichartzG} that
\begin{align}
\lambda^{-1/r}  \| U_\lambda (t) P_{\leqslant \lambda} u_0\|_{ {L}^{r} (\mathbb{R}, H^{\sigma-(1-\theta)\delta(r),s}_{r}  ) } & \lesssim \|u_0\|_{E^{\sigma,s}}, \label{HGStr-a}\\
\lambda^{-1/r}  \left\| \mathscr{A}_\lambda    P_{\leqslant \lambda} f  \right\|_{ {L}^{r} (\mathbb{R}, H^{\sigma-(1-\theta)\delta(r),s}_{r} )}
& \lesssim \lambda^{1/r} \|P_{\leqslant \lambda} f\|_{ {L}^{r'} (\mathbb{R}, H^{\sigma+(1-\theta)\delta(r),s}_{r'} )}.  \label{HGStr-b} \\
\lambda^{ \theta \delta(r)}  \| U_\lambda (t) P_{> \lambda} u_0\|_{ {L}^{r} (\mathbb{R}, H^{\sigma-1/2,s}_{r}  ) } & \lesssim \|u_0\|_{E^{\sigma,s}}, \label{HGStr-c}\\
\lambda^{ \theta \delta(r) }   \left\| \mathscr{A}_\lambda     P_{> \lambda} f  \right\|_{ {L}^{r} (\mathbb{R}, H^{\sigma-1/2,s}_{r} )}
   & \lesssim  \lambda^{ -\theta \delta(r) }  \| P_{> \lambda} f\|_{ {L}^{r'} (\mathbb{R}, H^{\sigma+1/2,s}_{r'}  )}. \label{HGStr-d}
\end{align}

According to the Strichartz estimates, we choose the following function space $W$ to solve NLKG:
\begin{align}
\|u\|_{W_1} = & \lambda^{-1/r} \|P_{\leqslant \lambda} u\|_{ {L}^{r} (\mathbb{R}, H^{ 1/2-(1-\theta)\delta(r) ,s}_{r} )} +
 \|P_{\leqslant \lambda} u\|_{ {L}^{\infty} (\mathbb{R}, E^{1/2,s} )}, \label{workspace1} \\
 \|u\|_{W_2} = & \lambda^{\theta \delta (r)} \|P_{> \lambda} u\|_{ {L}^{r} (\mathbb{R}, H^{0,s}_{r} )}
 + \|P_{> \lambda} u\|_{ {L}^{\infty} (\mathbb{R}, E^{1/2,s} )}.
  \label{workspace2}\\
 \|u\|_{W} = & \|u\|_{W_1} + \|u\|_{W_2}.
\end{align}
Let $u_{0,\lambda}= \lambda^{2/\alpha} u_0(\lambda \, \cdot), \, u_{1,\lambda}= \lambda^{1+ 2/\alpha} u_1(\lambda \, \cdot)$.  We consider the mapping
\begin{align}
   \label{intNLKG3}
\mathcal{T}: u(t) \to  K'_\lambda(t) u_{0,\lambda} + K_\lambda (t) u_{1,\lambda} + \int^t_0 K_\lambda (t-\tau ) u^{1+\alpha}(\tau) d\tau.
\end{align}
in the space
\begin{align}
   \label{metricspace}
\mathcal{D} & =\{u\in W: \ \|u\|_W \leqslant  \lambda^{-2(1-\theta)/\alpha (d+1+\theta)} \delta, \ {\rm supp }\, \widehat{u(t)} \subset \mathbb{R}^d_I \},  \nonumber\\
 & d(u,v) = \|u-v\|_W,
\end{align}
where $\delta>0$ will be chosen later. Using the Strichartz estimates and Lemma \ref{Multiplier}, we have
\begin{align} \label{workspace3}
\|\mathcal{T} u\|_{W_2}
 \lesssim &  \|u_{0,\lambda}\|_{E^{1/2,s}}   +    \|u_{1,\lambda}\|_{E^{-1/2,s}} + \lambda^{ -\theta\delta(r) }  \|  P_{> \lambda} (u^{1+\alpha}) \|_{ {L}^{r'} (\mathbb{R}, H^{0,s}_{r'} )}.
\end{align}
So, we need to make a nonlinear estimate to $ \|  P_{> \lambda} (u_1... u_{1+\alpha}) \|_{ {L}^{r'} (\mathbb{R}, H^{0,s}_{r'} )}$.
  Let ${\rm supp} \, \widehat{u}_i \subset \mathbb{R}^d_I, \ v_i = 2^{s|\nabla|}u_i, \ i=1,...,\alpha+1$. By Lemma \ref{RegTransf}, we have
\begin{align} \label{1dest-1}
\lambda^{ -\theta\delta(r) } \| P_{> \lambda} (u_1 u_2...u_{\alpha+1}) \|_{ {L}^{r'} (\mathbb{R}, H^{0,s}_{r'} )}  = \lambda^{ -\theta\delta(r) } \| P_{> \lambda} (v_1v_2...v_{\alpha+1})\|_{{L}^{r'}_{x,t}}.
\end{align}
In the right hand side of \eqref{1dest-1}, one can assume that $v_1$ has the highest frequency. It follows that
$$
P_{>\lambda} ((P_{ \leqslant \lambda/(1+\alpha)}v_1)v_2...v_{\alpha+1}) =0.
$$
Hence, we have
\begin{align} \label{1dest-2}
  \lambda^{ -\theta\delta(r) } \| P_{> \lambda} (v_1v_2...v_{\alpha+1})\|_{{L}^{r'}_{x,t}} =  & \lambda^{ -\theta\delta(r) } \| P_{> \lambda} ( (P_{ > \lambda/(1+\alpha)} v_1) v_2...v_{\alpha+1})\|_{{L}^{r'}_{x,t}} \nonumber\\
  \leqslant   & \lambda^{ -\theta\delta(r) } \| P_{> \lambda} ( (P_{   \lambda/(1+\alpha) <\cdot \leqslant \lambda } v_1) v_2...v_{\alpha+1})\|_{{L}^{r'}_{x,t}} \nonumber\\
   &  +    \lambda^{ -\theta\delta(r) } \| P_{> \lambda} ( (P_{ > \lambda } v_1) v_2...v_{\alpha+1})\|_{{L}^{r'}_{x,t}} := I+II.
\end{align}
By H\"older's inequality and Lemma \ref{Multiplier},
 \begin{align} \label{1dest-3}
 I \leqslant     \lambda^{ -\theta \delta(r) } \| P_{\lambda/(1+\alpha) <\cdot  \leqslant \lambda } v_1 \|_{L^r_{x,t}} \prod^{1+\alpha}_{j=2}\| P_{  \leqslant \lambda } v_j \|_{L^r_{x,t}}
   \lesssim     \lambda^{ 2(1-\theta)/(d+1+\theta) } \prod^{1+\alpha}_{j=1}\|  u_j \|_{W_1} .
\end{align}
By decomposing $v_j =  P_{ > \lambda } v_j + P_{ \leqslant \lambda } v_j$, we see that
 \begin{align} \label{1dest-4}
 II & \leqslant     \lambda^{ -\theta \delta(r) } \| P_{> \lambda } v_1 \|_{L^r_{x,t}} \prod^{1+\alpha}_{j=2}( \| P_{  \leqslant \lambda } v_j \|_{L^r_{x,t}} +  \| P_{> \lambda } v_j \|_{L^r_{x,t}}) \nonumber\\
  & \leqslant     \lambda^{ -2 \theta \delta(r) } \| u_1 \|_{W_2} \prod^{1+\alpha}_{j=2}(\lambda^{ -\theta \delta(r) } \| u_j \|_{W_2} + \lambda^{ 1/r } \| u_j \|_{W_1}) \nonumber\\
  & \lesssim    \lambda^{ 2(1-\theta)/(d+1+\theta)  } \| u_1 \|_{W_2} \prod^{1+\alpha}_{j=2}  \| u_j \|_{W}.
   \end{align}
Inserting the estimates of $I$ and $II$ into \eqref{workspace3},
 \begin{align} \label{1dest-5}
\|\mathcal{T} u\|_{W_2}
 \lesssim  \|u_{0,\lambda}\|_{E^{1/2,s}}   +    \|u_{1,\lambda}\|_{E^{-1/2,s}} +   \lambda^{ 2(1-\theta)/(d+1+\theta)  }   \prod^{1+\alpha}_{j=1}  \| u_j \|_{W}.
\end{align}
Next, by the Strichartz estimates and Lemma \ref{Multiplier}, we have
\begin{align} \label{1dest-6}
\|\mathcal{T} u\|_{W_1}  \lesssim &  \|u_{0,\lambda}\|_{E^{1/2,s}}   +    \|u_{1,\lambda}\|_{E^{-1/2, s}} + \lambda^{1/r-1} \left\| P_{\leqslant \lambda} (u^{1+\alpha}) \right\|_{ {L}^{r'} (\mathbb{R}, H^{1/2+ (1-\theta)\delta(r),s}_{r'} )} \nonumber\\
 \lesssim &  \|u_{0,\lambda}\|_{E^{1/2,s}}   +    \|u_{1,\lambda}\|_{E^{-1/2,s}} + \lambda^{1/r-1 +2(1-\theta)\delta(r) }  \|  P_{\leqslant \lambda} (u^{1+\alpha}) \|_{ {L}^{r'} (\mathbb{R}, H^{1/2- (1-\theta)\delta(r),s}_{r'} )}.
\end{align}
Since ${\rm supp} \, \widehat{u} \subset \mathbb{R}^d_I $, we see that $P_{\leqslant \lambda } u^{1+\alpha} = P_{\leqslant \lambda } (P_{\leqslant \lambda } u)^{1+\alpha}$. Taking $v= 2^{s|\nabla|} u$, we have
\begin{align} \label{1dest-7}
   \|  P_{\leqslant \lambda} (u^{1+\alpha}) \|_{ {L}^{r'} (\mathbb{R}, H^{1/2- (1-\theta)\delta(r),s}_{r'} )} &  \leqslant    \| ( P_{\leqslant \lambda}v  )^{1+\alpha}\|_{ {L}^{r'} (\mathbb{R}, H^{1/2- (1-\theta)\delta(r) }_{r'} )} \nonumber\\
    &\leqslant      \| P_{\leqslant \lambda} v \|^{1+\alpha}_{ {L}^{r} (\mathbb{R}, H^{1/2- (1-\theta)\delta(r) }_{r} ) }.
\end{align}
The last inequality in \eqref{1dest-7} can be obtained by making a multi-linear interpolation to
\begin{align} \label{1dest-7a}
   \| v_1... v_{1+\alpha} \|_{ {L}^{r'} (\mathbb{R}, H^{k}_{r'} )}
    &\leqslant     \prod^{1+\alpha}_{i=1} \|   v_i \|^{1+\alpha}_{ {L}^{r} (\mathbb{R}, H^{k}_{r} ) }, \ \ k =0,1.
\end{align}

It follows that
\begin{align} \label{1dest-8}
\|\mathcal{T} u\|_{W_1}
 \lesssim &  \|u_{0,\lambda}\|_{E^{1/2,s}}   +    \|u_{1,\lambda}\|_{E^{-1/2,s}} +  \lambda^{ 2(1-\theta)/(d+1+\theta) }  \| u  \|^{1+\alpha}_{W_1}.
\end{align}
Summarizing \eqref{1dest-5} and \eqref{1dest-8}
\begin{align} \label{1dest-9}
 \|\mathcal{T} u\|_{W}
 \lesssim &  \|u_{0,\lambda}\|_{E^{1/2,s}}   +    \|u_{1,\lambda}\|_{E^{-1/2,s}} +  \lambda^{  2(1-\theta)/(d+1+\theta) }  \| u  \|^{1+\alpha}_{W}.
\end{align}
Notice that
$$
\frac{d}{2} - \frac{2}{\alpha} > \frac{2(1-\theta)}{2\alpha (d+1+\theta)} \ \mbox{for} \ \alpha = \frac{4}{d-1+\theta}, \ \theta <1,
$$
We have for some $\epsilon>0$,
$$
\frac{d}{2} - \frac{2}{\alpha} = \frac{2(1-\theta)}{2\alpha (d+1+\theta)} + 2\epsilon.
$$
By the scaling on $E^{1/2,s}$ one sees that
\begin{align} \label{1dest-10}
  \|u_{0,\lambda}\|_{E^{1/2,s}}  \leqslant C \lambda^{-d/2+2/\alpha + \epsilon } \|u_0\|_{E^{1/2,s}}   \leqslant C \lambda^{-\epsilon}  \lambda^{-\frac{2(1-\theta)}{2\alpha (d+1+\theta)}} \|u_0\|_{E^{1/2,s}}.
\end{align}
Since $\widehat{u}_1$ is supported in $\{\xi: \, |\xi| \geqslant \varepsilon_0 \}$ for some $\varepsilon_0>0$, we see that
\begin{align} \label{1dest-11}
  \|u_{1,\lambda}\|_{E^{1/2,s}}   \leqslant C \lambda^{-\epsilon}  \lambda^{-\frac{2(1-\theta)}{2\alpha (d+1+\theta)}} \|u_1\|_{E^{1/2,s}}.
\end{align}
Then we can repeat the procedure as in the above to finish the proof for the case $\sigma =1/2$. If $\sigma >1/2$,  in Corollary \ref{lowpNonlinearestLH}, we see that \eqref{lowpworkspace3aa} plays the same role as \eqref{1dest-9} and the details of the proof are omitted.

\subsection{Proof of Theorem \ref{mainresult2}}

The support conditions on initial data in Theorem \ref{mainresult2} is not optimal and they can be slightly improved by the following
$$
\left\{
\begin{array}{l}
  {\rm supp} \widehat{u}_0, \, {\rm supp} \widehat{u}_1 \subset \mathbb{R}^d_I \setminus \{0\}  \ \mbox{for} \ d=2; \\
  {\rm supp} \widehat{u}_0 \subset \mathbb{R}^d_I, \ {\rm supp} \widehat{u}_1 \subset \mathbb{R}^d \setminus \{0\}  \ \mbox{for} \ d=3,4; \\
 {\rm supp} \widehat{u}_0, \, {\rm supp} \widehat{u}_1 \subset \mathbb{R}^d_I   \ \mbox{for} \ d \geqslant 5. \\
\end{array}
\right.
$$
We will prove Theorem \ref{mainresult2} by assuming the above support conditions.  If $u$ is a solution of NLKG \eqref{NLKG} with $f(u)= \sinh u -u$,
then $u_\lambda (x,t): =   u(\lambda x, \lambda t) $  solves
\begin{align}\label{NLKG3}
\left\{
\begin{array}{l}
\partial^2_t v  + \lambda^2 v-  \Delta v  + \lambda^2 (\sinh v -v)  =0,  \\
  v(x,0) = u_0(\lambda x), \ v_t (x,0) = \lambda  u_1(\lambda x),
\end{array}
\right.
\end{align}
Let $X,\, X_1, \, X_2$ be as in \eqref{1workspace}--\eqref{1workspace2}. Our goal is to solve the following integral equation
\begin{align}
   \label{2intNLKG3}
v(t)= K'_\lambda(t) v_{0,\lambda} + K_\lambda (t) v_{1,\lambda} + \lambda^2 \int^t_0 K_\lambda (t-\tau ) f(v(\tau)) d\tau,
\end{align}
where $f(v) = \sinh v -v $, $v_{0,\lambda} = u_0(\lambda\, \cdot), \, v_{1,\lambda} = \lambda u_1(\lambda\, \cdot)$. First, we assume that $d\geqslant 3$.  We consider the mapping
\begin{align}
   \label{2intNLKG3map}
\mathcal{T}_s: v(t) \to  K'_\lambda(t) v_{0,\lambda} + K_\lambda (t) v_{1,\lambda} +  \lambda^2 \int^t_0 K_\lambda (t-\tau ) f(v(\tau)) d\tau
\end{align}
in the space
\begin{align}
   \label{2metricspace}
\mathcal{D}_\delta  =\{v \in X: \ \|v\|_X \leqslant \lambda^{-(d+2)/(d+1)} \delta, \ {\rm supp }\, \widehat{v(t)} \subset \mathbb{R}^d_I \},  \
  \ d(v,w) = \|v-w\|_X,
\end{align}
where $\delta >0$ will be chosen later. By the same way as in \eqref{1workspace3}
we have
\begin{align} \label{2workspace3}
\|\mathcal{T}_s v\|_{X_2}
 \lesssim &  \|v_{0,\lambda}\|_{E^{\sigma,s}}   +    \|v_{1,\lambda}\|_{E^{\sigma-1,s}} +  \lambda^2 \| \sinh v -v\|_{\widetilde{L}^{p'} (\mathbb{R}, B^{\sigma-1/2,s}_{p', 2}(\mathbb{Z}^c_\lambda))}.
\end{align}
In view of Taylor's expansion of $\sinh v$, we see that
$$
\sinh v - v = \sum^\infty_{k=1} \frac{1}{(2k+1)!} \, v^{2k+1} =  \sum_{m\in 2\mathbb{N}+1} \frac{v^m}{m!} .
$$

Notice that $X= X_1 \cap X_2$. By Corollary \ref{corNonlinearestH1}  and \eqref{2workspace3}, for any $v\in \mathcal{D}_\delta$,
\begin{align} \label{2workspace6}
\|\mathcal{T}_s v\|_{X_2}
 \leqslant &  C(\|v_{0,\lambda}\|_{E^{\sigma,s}}   +    \|v_{1,\lambda}\|_{E^{\sigma-1,s}})   +  C \lambda^{2 }  \sum_{m\in 2\mathbb{N}+1} \frac{1}{m!} \|v^{m}\|_{\widetilde{L}^{p'} (\mathbb{R}, B^{\sigma-1/2,s}_{p', 2}(\mathbb{Z}^c_\lambda))} \nonumber \\
\leqslant &  C(\|v_{0,\lambda}\|_{E^{\sigma,s}}   +    \|v_{1,\lambda}\|_{E^{\sigma-1,s}})   +  C \lambda^{2+2/(d+1)}  \sum_{m\in 2\mathbb{N}+1} a_m \|v\|^{m}_{X}  \nonumber \\
  \leqslant &  C(\|v_{0,\lambda}\|_{E^{\sigma,s}}   +    \|v_{1,\lambda}\|_{E^{\sigma-1,s}})   + C\lambda^{-(d+2)/(d+1)} \delta^3 \sum_{m\in 2\mathbb{N}+1} a_m (\lambda^{-(d+2)/(d+1)}\delta)^{m-3}.
\end{align}
where
$$
a_m = \frac{(C\sqrt{m})^{m+C}}{m!}.
$$
If ${\rm supp} \widehat{\varphi}$ is supported in $\mathbb{R^d}_I \setminus \{0\}$,  using (ii) of Corollary \ref{Scaling4}, there exists $\varepsilon_0, \, \lambda_0 >1$ such that for any $\lambda\geqslant \lambda_0$,
\begin{align} \label{2workspace5}
 \lambda \|\varphi(\lambda\, \cdot)\|_{E^{\tilde{\sigma},s}}      \leqslant 2^{s\lambda \varepsilon_0 /4}  \|\varphi\|_{E^{\tilde{\sigma},s}} \leqslant  \lambda^{-(d+2)/(d+1)-1} \|\varphi\|_{E^{\tilde{\sigma},s}}, \ \tilde{\sigma} \in \mathbb{R}.
\end{align}
$d \geqslant 3$ implies that $d/2 > (d+2)/(d+1)$. Let $\epsilon_1 >0$ satisfy $d/2- 2 \epsilon_1 = (d+2)/(d+1)$. By (i) of Corollary \ref{Scaling4}, there exists $ \lambda_1 >1$ such that for any $\lambda\geqslant \lambda_1$,
\begin{align} \label{2workspace7}
\|v_{0,\lambda}\|_{E^{\sigma,s}} \leqslant C \lambda^{-d/2 + \epsilon_1} \|u_0\|_{E^{\sigma,s}} =  C \lambda^{-(d+2)/(d+1) - \epsilon_1} \|u_0\|_{E^{\sigma,s}}.
\end{align}
Similarly,  $d \geqslant 5$ implies that $d/2 > 1+ (d+2)/(d+1)$. Let $\epsilon_2 >0$ satisfy $d/2- 2 \epsilon_2 =1+ (d+2)/(d+1)$. By (i) of Corollary \ref{Scaling4}, there exists $ \lambda_1 >1$ such that for any $\lambda\geqslant \lambda_2$,
\begin{align} \label{2workspace7a}
\|v_{1,\lambda}\|_{E^{\sigma-1,s}} \leqslant C \lambda^{-d/2 +1+ \epsilon_2} \|u_0\|_{E^{\sigma-1,s}} =  C \lambda^{-(d+2)/(d+1) - \epsilon_2} \|u_1\|_{E^{\sigma-1,s}}.
\end{align}
Taking $\epsilon = \min (\epsilon_1, \, \epsilon_2)$. We have
\begin{align} \label{2workspace6a}
\|\mathcal{T}_s v\|_{X_2}
 \leqslant &  C \lambda^{-(d+2)/(d+1) - \epsilon } (\|u_{0}\|_{E^{\sigma,s}}   +    \|u_{1}\|_{E^{\sigma-1,s}})  \nonumber\\
   & \ + C\lambda^{-(d+2)/(d+1)} \delta^3 \sum_{m\in 2\mathbb{N}+1} a_m (\lambda^{-(d+2)/(d+1)}\delta)^{m-3}.
\end{align}
By \eqref{2NLE3} and Corollary \ref{corNonlinearestL1} we see that $\|\mathcal{T}_s v\|_{X_1}$ has the same upper bound as in \eqref{2workspace6a}.
It follows that, for any $v\in \mathcal{D}_\delta$, $\lambda \geqslant \max (\lambda_0, \lambda_1, \lambda_2 )$,
\begin{align} \label{2workspace8}
\|\mathcal{T}_s v\|_{X}   \leqslant &  C \lambda^{-(d+2)/(d+1) - \epsilon } (\|u_{0}\|_{E^{\sigma,s}}   +    \|u_{1}\|_{E^{\sigma-1,s}})  \nonumber\\
   & \ + C\lambda^{-(d+2)/(d+1)} \delta^3 \sum_{m\in 2\mathbb{N}+1} a_m  (\lambda^{-(d+2)/(d+1)} \delta)^{m-3}.
\end{align}
Since $\lim_{k\to \infty} a_{2k+3}/a_{2k+1}=0$, we see that for any $\delta >0$,
$$
L_\delta : =  \sum_{m\in 2\mathbb{N}+1} a_m  \delta^{m-3}
$$
is a convergent series.  Let $\delta>0$ satisfy $ C L_{1}    \delta^{2 } \leqslant 1/100.$  One can find a $\lambda_*  \geqslant \max (\lambda_0, \lambda_1, \lambda_2 ) $ satisfying
\begin{align*}
2  C  \lambda_*^{- \epsilon} (\|u_0\|_{E^{\sigma,s}} + \|u_1\|_{E^{\sigma-1,s}} )  < \delta.
\end{align*}
It follows that for $\lambda=\lambda_*$,  $v\in \mathcal{D}_\delta$,
\begin{align} \label{2workspace9}
\|\mathcal{T}_s v\|_{X}
 \leqslant \,  \delta/2   +   C L_{1} \lambda^{-(d+2)/(d+1)} \delta^3     \leqslant \delta.
\end{align}
Similarly, for any $v_1,v_2 \in  \mathcal{D}_\delta$,
\begin{align} \label{2workspace10}
\|\mathcal{T}_s v_1 - \mathcal{T}_s v_2 \|_{X}
 \leqslant \, \frac{1}{2} \| v_1 -   v_2 \|_{X}.
\end{align}
So, $\mathcal{T}_s$ has a unique fixed point in $\mathcal{D}_\delta$ which is a solution of \eqref{2intNLKG3}.  The left part of the proof of Theorem \ref{mainresult2} for $d\geqslant 3$ follows the same way as that of Theorem \ref{mainresult}. In dimensions $d=2$, we can used the estimates in Section \ref{GlobalNLKGB} and imitate the above argument. Put $M= \lambda^{-1}\delta$ for some small $\delta >0$ in\eqref{2dmetricspace}.
By \eqref{2dworkspace3}, Lemmas \ref{2dNonlinearestH1} and \ref{2dcorNonlinearestL1}, for any $v\in \mathcal{D}_2$,
\begin{align} \label{2dworkspace3}
\|\mathcal{T}v\|_{Y}
 \lesssim  \|u_{0,\lambda}\|_{E^{\sigma,s}}   +    \|u_{1,\lambda}\|_{E^{\sigma-1,s}} +    C\lambda^{-1} \delta^3 \sum_{m\in 2\mathbb{N}+1} a_m (\lambda^{-1}\delta)^{m-3}.
\end{align}
 Then we can repeat the proof for the case $d \geqslant 3$ to get the result and the details are omitted.  \\

{\it Sketch Proof of (iii) of Remark \ref{Rmk1.3}.} Using the Taylor expansion of $\sin u$, we have
$$
\sin u -u = \sum^\infty_{k=1}   \frac{(-1)^k}{(2k+1)!} \, v^{2k+1}.
$$
So, one can handle the nonlinearity $\sin u -u$ in the same way as $\sinh u - u$ and Theorem \ref{mainresult2} also holds for the sine-Gordon equation.
Recall that
$$
e^{v^2} v - v = \sum^\infty_{m=1} \frac{1}{m!} \, v^{2m+1}  .
$$
So, the growth of $e^{v^2} v$ is faster than $\sinh v$. However, the proof of Theorem \ref{mainresult2} is still adapted to the nonlinearity $
f(v) = e^{v^2} v - v  $. Say, for $d\geqslant 3$, using Taylor's expansion of $e^{v^2} v - v$   and following \eqref{2workspace8} in the proof of Theorem \ref{mainresult2}, we have
\begin{align} \label{2workspace8exp}
\|\mathcal{T}_s v\|_{X}   \leqslant &  C \lambda^{-(d+2)/(d+1) - \epsilon } (\|u_{0}\|_{E^{\sigma,s}}   +    \|u_{1}\|_{E^{\sigma-1,s}})  \nonumber\\
   & \ + C\lambda^{-(d+2)/(d+1)} \delta^3 \sum_{m\in  \mathbb{N} } b_m  (\lambda^{-(d+2)/(d+1)} \delta)^{2(m-1)}.
\end{align}
where
$$
b_m = \frac{(C\sqrt{2m+1})^{2m+1+C}}{m!}.
$$
It is easy to see that
$$
\lim_{m\to \infty} \frac{b_{m+1} (\lambda^{-(d+2)/(d+1)} )^{2(m+1)} }{b_m (\lambda^{-(d+2)/(d+1)} )^{2(m-1)}} = C^2 2  e \lambda^{-2(d+2)/(d+1)}.
$$
It follows that $\sum_{m\in  \mathbb{N} } b_m  (\lambda^{-(d+2)/(d+1)} \delta)^{2(m-1)}$ is a convergent series if
$$
\delta^2 < \eta = \lambda^{ 2(d+2)/(d+1)} / C^2 2  e.
$$
Then we can use a similar way as in the proof of Theorem \ref{mainresult2} to get the results of (iii) in  Remark \ref{Rmk1.3}.

The global well posedness of NLKG with small data in $H^\sigma\times H^{\sigma-1}$, $\sigma \geqslant d/2$ is obtained by Nakamura and Ozawa \cite{NaOz2001}, where they applied the differential-difference norms on Besov spaces together with an interesting limiting version of Sobolev embedding.

\section{Further discussions on initial data} \label{lookingup}

Just like $\lim_{\sigma \to -\infty} \|\varphi\|_{H^\sigma} =0$, it follows from the dominating convergence theorem that for $\varphi\in E^{\sigma,s}$,
$$
\lim_{b \to -\infty} \|\varphi\|_{E^{\sigma,b}} =0, \ \sigma \in \mathbb{R}.
$$
So, for any $(u_0, u_1) \in E^{\sigma,s} \times E^{\sigma-1,s}$, we can find some $s_0  \leqslant s$ such that
\begin{align}
\|u_0 \|_{E^{\sigma,s_0}} + \|u_1\|_{E^{\sigma-1,s_0}}  \ll 1.  \label{anysmall}
\end{align}
By Lemma \ref{RegTransf}, if $u$ solves NLKG \eqref{NLKG}, so does $2^{s|\nabla|}u$ with initial data $(2^{s|\nabla|}u_0, 2^{s|\nabla|}u_1)$  which are supported in the first octant in frequency spaces. By \eqref{anysmall}, one can directly applied the techniques in \cite{Wa1998,Wa1999},  Nakamura and Ozawa \cite{NaOz2001} to get similar results of Theorems \ref{mainresult} and \ref{mainresult2}, respectively.   However, $s_0 =s_0 (u_0, u_1)$ in \eqref{anysmall} will depends on $u_0, \, u_1 \in E^{\sigma,s}$, which is unbounded for  all $u_0, \, u_1 \in E^{\sigma,s}$  and we have
\begin{lemma} \label{unboundedness}
Let $\chi_A$ be the characteristic function on $A \subset \mathbb{R}^d$.
Let $b<s<0, \, \sigma \in \mathbb{R}$.  Let $\widehat{\varphi}_k = k^{d/2} \chi_{B(0,1/k)}$, $k\gg 1$. Then we have $ \varphi_k \in E^{\sigma,s}$.  If
$$
\|\varphi_k\|_{E^{\sigma,b}}  \leqslant \varepsilon,
$$
then we have $ - b \gtrsim k \log_2 {\varepsilon^{-1}} $.
\end{lemma}
\begin{proof} We have
\begin{align*}
 \|\varphi_k \|_{E^{\sigma,s}} &  = \|\langle\xi \rangle^{\sigma}  2^{s |\xi|}\widehat{\varphi_k}(\xi)\|_{2} \lesssim_\sigma \|\varphi_k\|_2  \lesssim 1.
\end{align*}
Conversely,
\begin{align*}
 \|\varphi_k \|_{E^{\sigma,b}} &  = \|\langle\xi \rangle^{\sigma}  2^{b |\xi|}\widehat{\varphi_k}(\xi)\|_{L^2(\mathbb{R}^d)}
    \gtrsim_{\sigma} 2^{b /k}    \|\varphi_k\|_2  \gtrsim_{\sigma} 2^{b /k}.
\end{align*}
So, $\|\varphi_k\|_{E^{\sigma,b}}  \leqslant \varepsilon$ implies that $|b| \gtrsim k \log_2 {\varepsilon^{-1}}$.
\end{proof}
By Lemma \ref{unboundedness}, we see that
$$
b(\varphi, \varepsilon):= \sup \{b<s : \, \varphi \in E^{\sigma,s},  \ \|\varphi \|_{E^{\sigma,b}}  \leqslant \varepsilon\}
$$
depends on both $\varphi  \in E^{\sigma,s}$ and $0< \varepsilon\ll 1$ and $\inf_{\varphi \in E^{\sigma,s}} b(\varphi, \varepsilon) = -\infty$.

However, if we use the scaling argument in the paper or assume that
${\rm supp} \, \widehat{\varphi} \, \subset \mathbb{R}^d_I \setminus B(0, \varepsilon_0)$ for some $\varepsilon_0 >0$, then we can realize that
$\inf_{\varphi \in E^{\sigma,s}} b(\varphi, \varepsilon) < -\infty$. In the later case $ b(\varphi, \varepsilon)= b(\|\varphi\|_{E^{\sigma, s}} , \varepsilon_0)$.

\appendix
\section{Higher order derivatives}
\begin{lemma} \label{Deriv-control}
Let $\alpha\in \mathbb{R}$,  $p(r) = (\mu^2+ r^2)^\alpha$, where $\mu>0$ is a variable parameter.  Then there exist  $a_{k,j} := a_{k,j} (\alpha, k)$  $(j=0,...,k)$ which are independent of $\mu>0$ such that
\begin{align}
p^{(2k)} (r) & = \sum^{k}_{j=0} a_{k,j} \mu^{2j} (\mu^2+r^2)^{\alpha-k-j},   \label{appendix1}\\
p^{(2k+1)} (r) & = 2r\sum^{k}_{j=0} (\alpha-k-j) a_{k,\, j} \mu^{2j} (\mu^2+r^2)^{\alpha-k-j-1}.   \label{appendix2}
\end{align}
Moreover, we have
\begin{align}
|p^{(\ell)} (r)| \leqslant C_{k}  (\mu^2+r^2)^{\alpha-\ell/2},    \ \ \forall \, \ell \, \in \mathbb{N} \cup\{0\}.  \label{appendix5}
\end{align}
\end{lemma}
\begin{proof}
Differentiating \eqref{appendix1}, we get  \eqref{appendix2}. So, it suffices to show  \eqref{appendix1} by induction. The case $k=0$ is obvious. Now let us assume that \eqref{appendix1} holds. We calculate $p^{(2k+2)} (r)$. Taking the derivative to \eqref{appendix2},
\begin{align}
p^{(2k+2)} (r) = & 2\sum^{k}_{j=0} (\alpha-k-j) a_{k,j} \mu^{2j} (\mu^2+r^2)^{\alpha-k-j-1} \nonumber  \\
& +    4 \sum^{k}_{j=0} (\alpha-k-j)(\alpha-k-1-j) a_{k,j} \mu^{2j}r^2 (\mu^2+r^2)^{\alpha-k-j-2} \nonumber\\
= & 2\sum^{k}_{j=0} (\alpha-k-j)(2\alpha -2k-2j-1) a_{k,j} \mu^{2j} (\mu^2+r^2)^{\alpha-(k+1)-j} \nonumber  \\
& -    4 \sum^{k+1}_{j=1} (\alpha-k-j+1)(\alpha-k-j) a_{k,j-1} \mu^{2j}  (\mu^2+r^2)^{\alpha-(k+1)-j}.
\label{appendix3}
\end{align}
Putting
\begin{align}
a_{k+1,0} = & 2 a_{k,0} (\alpha-k) (2\alpha-2k-1),  \ \   a_{k+1, k+1} =   -4 a_{k,k}(\alpha-k)( \alpha - k+1),   \nonumber  \\
a_{k+1, j} =  &  2 a_{k,j}(\alpha-k-j)(2\alpha-2k-2j-1) - 4 a_{k,j-1}(\alpha-k-j)( \alpha- k- j+1)   \nonumber
\end{align}
for $j=1,...,k$, we get \eqref{appendix1} and \eqref{appendix2}, as desired.
By \eqref{appendix1} and \eqref{appendix2}, we see that
\begin{align}
|p^{(2k)} (r)| \leqslant C_{k}  (\mu^2+r^2)^{\alpha-k},    \ \
|p^{(2k+1)} (r) | \leqslant C_{k}   (\mu^2+r^2)^{\alpha-k-1/2},   \label{appendix4}
\end{align}
which means that \eqref{appendix5} holds.
\end{proof}

\begin{lemma} \label{Deriv-control2}
Let $h (r), \, \Phi (r)$ be smooth functions.  Denote
\begin{align}
\sigma^{(1)} (h) = (h\Phi)', \ \ \sigma^{(k+1)} (h) = (h \sigma^{(k)}(h))' , \ \ k=1,2,...
\end{align}
Then we have
\begin{align}
\sigma^{(\ell)} (h) = \sum_{\alpha_1+...+\alpha_\ell \leqslant \ell}  C_{\alpha_1,...,\alpha_\ell} h^{(\alpha_1)} ... h^{(\alpha_{\ell})}  \Phi^{(\ell - \alpha_1-...-\alpha_\ell)}.  \label{appendix6}
\end{align}
\end{lemma}
\begin{proof}
By Leibnitz rule, we have the result for $k=1$. The general result can be obtained by induction. Indeed, assume that \eqref{appendix6} holds. Then we have
\begin{align}
\sigma^{(\ell+1)} (h) = \sum_{\alpha_1+...+\alpha_\ell \leqslant \ell}  C_{\alpha_1,...,\alpha_\ell} \frac{d}{dr}(h h^{(\alpha_1)} ... h^{(\alpha_{\ell})}  \Phi^{(\ell - \alpha_1-...-\alpha_\ell)}).  \label{appendix7}
\end{align}
Using the Leibnitz rule, we have the result, as desired.
\end{proof}

\begin{lemma} \label{Deriv-control3}
We have
\begin{align}
 \left( f^{-1} \right)^{(k)}=  \frac{1}{f^{k+1}}  \sum_{\alpha_1+...+\alpha_k = k}  C_{\alpha_1,...,\alpha_k} f^{(\alpha_1)} ... f^{(\alpha_{k})}. \label{appendix8}
\end{align}
\end{lemma}
\begin{proof}
By Leibnitz rule and induction, we have
 the result, as desired.
\end{proof}

\section{Proof of Lemma \ref{Strichartz}}

We only sketch the proofs of \eqref{Str-a} and  \eqref{Str-b}. By \eqref{Str1} one has that for $2^ j  \lesssim   \lambda  $,
\begin{align}
\|\triangle_j \mathscr{A}_\lambda f \|_p \lesssim \lambda^{2/\gamma(\theta,p)} 2^{ 2 j(1-\theta)\delta(p)}  \int^t_0 |t-\tau|^{-2/\gamma(\theta,p)} \|f(\tau)\|_{p'} d\tau.  \label{appStr-1}
\end{align}
Using Hardy-Littlewood-Sobolev inequality,
\begin{align}
\|\triangle_j \mathscr{A}_\lambda f \|_{L^{\gamma(\theta,p)}_t L^p_x } \lesssim \lambda^{2/\gamma(\theta,p)} 2^{ 2 j(1-\theta)\delta(p)}   \|f \|_{L^{\gamma(\theta,p)'}_t L^{p'}_x}  .  \label{appStr-2}
\end{align}
By Cauchy-Schwarz inequality,
\begin{align}
\left|\int \left( \triangle_j U_\lambda (t) u_0 , f(t) \right) dt \right|   \leqslant \left\| \int   \triangle_j U_\lambda (-t)  f(t) dt \right\|_2 \|u_0\|_2.   \label{appStr-3}
\end{align}
By duality,
\begin{align}
  \left\| \int   \triangle_j U_\lambda (-t)  f(t) dt \right\|^2_2  & \leqslant \|f\|_{L^{\gamma(\theta,p)'}_t L^{p'}_x}  \left\| \int   \triangle_j U_\lambda (t-\tau)  f(\tau) d\tau \right\|_{L^{\gamma(\theta,p)}_t L^{p}_x}  \nonumber\\
  &  \lesssim \lambda^{2/\gamma(\theta,p)} 2^{ 2 j(1-\theta)\delta(p)}  \|f\|^2_{L^{\gamma(\theta,p)'}_t L^{p'}_x} .
      \label{appStr-4}
\end{align}
By \eqref{appStr-3} and \eqref{appStr-4},
\begin{align}
\left|\int \left( \triangle_j U_\lambda (t) u_0 , f(t) \right) dt \right|   \lesssim   \lambda^{1/\gamma(\theta,p)} 2^{   j(1-\theta)\delta(p)}  \|f\|_{L^{\gamma(\theta,p)'}_t L^{p'}_x} \|u_0\|_2.   \label{appStr-5}
\end{align}
By duality, we have \eqref{Str-a}.  Using an analogous way as in \eqref{appStr-4},
\begin{align}
\left\| \triangle_j \mathscr{A}_\lambda   f(t)  \right\|_{L^\infty_t L^2_x }    \lesssim   \lambda^{ 1/\gamma(\theta,r) } 2^{   j(1-\theta)\delta(r)}  \|f\|_{ L^{\gamma(\theta,r)'}_t L^{r'}_x } .   \label{appStr-6}
\end{align}
In view of \eqref{appStr-2},
\begin{align}
 \lambda^{- 1/\gamma(\theta,r) } 2^{ - j(1-\theta)\delta(r)} \left\| \triangle_j \mathscr{A}_\lambda   f(t)  \right\|_{L^{\gamma(\theta,r)}_t L^{r}_x }    \lesssim   \lambda^{ 1/\gamma(\theta,r) } 2^{   j(1-\theta)\delta(r)}  \|f\|_{ L^{\gamma(\theta,r)'}_t L^{r'}_x } .   \label{appStr-7}
\end{align}
For any $p\in [2,r]$, there exists $\eta \in [0,1]$ such that
$$
\frac{1}{p} = \frac{1-\eta}{2} + \frac{\eta}{r}.
$$
It follows that
$$
\delta(p) = (1-\eta)\delta(2) + \eta \delta(r), \ \  \frac{1}{\gamma(\theta,p)} = \frac{1-\eta}{\gamma(\theta,2)} + \frac{\eta}{\gamma(\theta,r)}.
$$
By H\"older's inequality, we have
\begin{align}
\lambda^{- 1/\gamma(\theta,p) }&  2^{ - j(1-\theta)\delta(p)} \left\| \triangle_j \mathscr{A}_\lambda   f(t)  \right\|_{L^{\gamma(\theta,p)}_t L^{p}_x } \nonumber\\
 &  \leqslant \left(  \left\| \triangle_j \mathscr{A}_\lambda   f(t)  \right\|_{L^{\infty}_t L^{2}_x } \right)^{1-\eta} \left( \lambda^{- 1/\gamma(\theta,r) } 2^{ - j(1-\theta)\delta(r)} \left\| \triangle_j \mathscr{A}_\lambda   f(t)  \right\|_{L^{\gamma(\theta,r)}_t L^{r}_x } \right)^\eta \nonumber\\
 & \lesssim   \lambda^{ 1/\gamma(\theta,r) } 2^{   j(1-\theta)\delta(r)}  \|f\|_{ L^{\gamma(\theta,r)'}_t L^{r'}_x } .   \label{appStr-8}
\end{align}
Again, in view of the duality,
\begin{align}
\lambda^{- 1/\gamma(\theta,p)} 2^{ -  j(1-\theta)\delta(p)}  \left\| \triangle_j \mathscr{A}_\lambda   f(t)  \right\|_{L^{\gamma(\theta,p)}_t L^{p}_x}    \lesssim   \|f\|_{L^1_t L^2_x} .   \label{appStr-9}
\end{align}
If $r\in [2,p]$,  by making an interpolation between \eqref{appStr-2} and \eqref{appStr-9}, we can obtain the result, as desired. This finishes the proof of \eqref{Str-a} and  \eqref{Str-b}. \\

\textbf{Acknowledgments.} The author is grateful to Professor K.~Nakanishi for his enlightening discussions on the paper. The project is supported in part by the NSFC, grant 12171007.
	
	\phantomsection 
	\bibliographystyle{amsplain}

\vspace{10pt}

\scriptsize\textsc{ Baoxiang Wang: School of Sciences, Jimei University,  Xiamen, 361021  P.~R.~of China.}
	
	\textit{E-mail address}: \verb"wbx@jmu.edu.cn"	


\end{document}